\documentclass[11pt]{amsart}
\usepackage{amscd}
\usepackage{amsmath}
\usepackage{amssymb}
\usepackage{latexsym}
\usepackage{xypic}
\usepackage{color}

\input
xy \xyoption{all}

\newcommand{\quash}[1]{}  

\newtheorem{thm}{Theorem}[section]
\newtheorem{prop}[thm]{Proposition}
\newtheorem{lem}[thm]{Lemma}

\newtheorem{lem-def}[thm]{Lemma-Definition}
\newtheorem{thm-def}[thm]{Theorem-Definition}
\newtheorem{cor}[thm]{Corollary}
\theoremstyle{remark}

\newtheorem{rmk}{Remark}[section]
\theoremstyle{definition}
\newtheorem{dfn}{Definition}[section]

\numberwithin{equation}{section}

\newcommand{\frakm}{{\mathfrak m}}

\newcommand{\frakp}{{\mathfrak p}}
\newcommand{\frakq}{{\mathfrak q}}

\newcommand{\fraku}{{\mathfrak u}}

\newcommand{\bbG}{{\mathbb G}}
\newcommand{\bbH}{{\mathbb H}}

\newcommand{\bbL}{{\mathbb L}}
\newcommand{\bbP}{{\mathbb P}}

\newcommand{\bbV}{{\mathbb V}}
\newcommand{\bbZ}{{\mathbb Z}}
\newcommand{\bbM}{{\mathbb M}}

\newcommand{\bbW}{{\mathbb W}}
\newcommand{\bbS}{{\mathbb S}}
\newcommand{\bbT}{{\mathbb T}}
\newcommand{\bbU}{{\mathbb U}}

\newcommand{\da}{{\mathbb A}}

\newcommand{\calD}{{\mathcal D}}
\newcommand{\calE}{{\mathcal E}}
\newcommand{\calF}{{\mathcal F}}

\newcommand{\calL}{{\mathcal L}}

\newcommand{\calO}{{\mathcal O}}
\newcommand{\calP}{{\mathcal P}}

\newcommand{\calT}{{\mathcal T}}

\newcommand{\ff}{{\mathcal F}}
\newcommand{\g}{{\mathcal G}}
\newcommand{\h}{{\mathcal H}}
\newcommand{\nc}{\newcommand}
\nc{\ve}{\varepsilon}

\newcommand{\oo}{{\mathcal O}}
\newcommand{\ho}{\hat{{\mathcal O}}}
\newcommand{\ad}{{\mathcal A}}
\newcommand{\B}{\mathcal B}

\nc{\on}{\operatorname}

\nc{\Aff}{{\mathbf{Aff}}}
\newcommand{\Aut}{{\on{Aut}}}

\newcommand{\Hom}{{\mathrm{Hom}}}
\newcommand{\Gal}{{\mathrm{Gal}}}
\nc{\id}{{\on{id}}}

\newcommand{\spec}{{\on{Spec} \,}}
\newcommand{\Frac}{{\on{Frac}}}
\nc{\Tate}{{\on{Tate}}} \nc{\pT}{{\on{preTate}}}

\newcommand{\Nm}{{\mathrm{Nm}}}
\newcommand{\Tr}{{\mathrm{Tr}}}

\newcommand{\GL}{{\on{GL}}}
\newcommand{\gl}{{\mathfrak g\mathfrak l}}
\nc{\Comm}{{\on{Comm}}} \nc{\Det}{{\calD et}}
\newcommand{\n}{{\mathcal N}}
\newcommand{\lrto}{\longrightarrow}
\nc{\res}{{\on{res}}} \nc{\Res}{{\on{Res}}} \nc{\Pic}{{\calP ic}}

\newcommand{\Ker}{\mathop{{\mathrm{Ker}}}}

\newcommand{\df}{{\mathbf F}}
\newcommand{\cha}{{\mathrm{char}}}

\begin{document}

\author{Denis Osipov, Xinwen
Zhu}\thanks{D.~O. partially  supported by
Russian Foundation for Basic Research (grants no.~14-01-00178-a, \linebreak no.~13-01-12420 ofi\_m2, and no.~12-01-33024 mol\_a\_ved)
and by the Programme for the Support of Leading Scientific Schools of the Russian Federation (grant no.~NSh-2998.2014.1). }
\thanks{X.~Z. partially supported by NSF grant DMS-1001280/1313894 and DMS-1303296 and by AMS Centennial Fellowship.}

\title{The two-dimensional Contou-Carr\`{e}re symbol and reciprocity laws}

\maketitle

\begin{abstract}
We define a two-dimensional Contou-Carr\`{e}re symbol, which is a deformation of the two-dimensional tame symbol
and is a natural generalization
of the (usual) one-dimensional  Contou-Carr\`{e}re symbol. We give several constructions of this
symbol and investigate its properties. Using higher categorical methods, we prove reciprocity laws on
algebraic surfaces for this symbol. We also relate the two-dimensional Contou-Carr\`{e}re  symbol to
the two-dimensional class field theory.
\end{abstract}

\section{Introduction}
This paper is a continuation of our previous paper~\cite{OsZh} and
we refer to the introduction of that paper for the general
background. In that paper, we developed some categorical
constructions such as categorical central extensions and generalized
commutators of these central extensions and applied them to the
construction of the two-dimensional tame symbol and to the proof of
reciprocity laws on algebraic surfaces. We also systematically used
the adeles on algebraic surfaces and  the categories of $1$-Tate and
$2$-Tate vector spaces and
their graded-determinantal theories  introduced by
M.~Kapranov in~\cite{Kap}.

The main goal of this paper is to extend constructions and theorems from a ground field $k$ (as in~\cite{OsZh}) to a
ground commutative ring $R$, by applying the categorical constructions developed there to the category
of Tate $R$-modules, which was introduced and studied by V.~Drinfeld in~\cite{Dr}. The generalized
commutator now gives us some new tri-multiplicative anti-symmetric map:
$$
R((u))((t))^* \times R((u))((t))^* \times R((u))((t))^*  \to R^* \mbox{,}
$$
which we call the two-dimensional Contou-Carr\`{e}re symbol. This
symbol coincides with the two-dimensional tame symbol when $R=k$ is
a field. Using adelic complexes on an algebraic surface we prove
reciprocity laws along a curve and around a point  for this  symbol
when $R $ is an artinian ring, see Theorem~\ref{res-law}.

An analogous deformation of the (usual) one-dimensional tame symbol
is known as the Contou-Carr\`{e}re symbol, see~\cite{Del2}
and~\cite{CC}. The reciprocity laws for the one-dimensional
Contou-Carr\`{e}re symbol on an algebraic curve were proved by using
the (usual) commutators of central extensions of groups,
see~\cite{BBE} and~\cite{AP}. The one-dimensional Contou-Carr\`{e}re
symbol can also be explicitly expressed by formulas. When ${\mathbf
Q} \subset R$  and $f$ and $g$ are appropriate elements from
$R((t))^*$, then (see formula~\eqref{cexp})
$$ (f,g)= \exp \res (\log f \frac{d g}{g})   \mbox{.}  $$
Another formula~\eqref{multform} applicable to any ring $R$ expresses
$(f,g)$ as certain finite product.

For the applications of our reciprocity laws (e.g. see
Section~\ref{cft}), we need various explicit formulas for the
two-dimensional Contou-Carr\`{e}re symbol. When ${\mathbf Q} \subset
R$,  we introduce in Section~\ref{tdt} an obvious generalization of
the previous one-dimensional formula. Namely, for appropriate elements $f$,
$g$ and $h$ from $R((u))((t))^*$ our formula looks as
$$
(f,g,h) = \exp \Res (\log f \frac{d g}{g}  \wedge \frac{d h}{h})  \mbox{,}
$$
where $\Res $ is the two-dimensional residue. Then we prove in
Theorem~\ref{th-main} that this explicit formula coincides with the
two-dimensional Contou-Carr\`{e}re symbol defined by the generalized
commutator as above. We remark that the proof of this fact in the
two-dimensional situation is much more difficult than in the
one-dimensional situation, because we need to deal with some
infinite sums (and products) here, while in the one-dimensional
situation only finite sums (and products) are involved.

Basically because of the same reason, we do not have an explicit
formula applicable to general commutative rings. In other words, unlike the
one-dimensional situation, we do not know how to define the
two-dimensional Contou-Carr\`{e}re symbol for general commutative
rings in an elementary way. (Our definition uses categorical central
extensions! Another approach via algebraic $K$-theory, is outlined in Section~\ref{Kapp}.) But if $R$ is Noetherian (or more generally, if the nil-radical of $R$ is an
nilpotent ideal), we obtain in
Section~\ref{explicform} an explicit formula for the two-dimensional
Contou-Carr\`{e}re symbol as certain finite product, which
generalizes the similar formula in the one-dimensional case.

We also investigate in Section~\ref{propert} various properties of the  two-dimensional  Contou-Carr\`{e}re symbol by elementary methods.
When the nil-radical of $R$ is a nilpotent ideal, we prove that this symbol satisfies the Steinberg property (and later in Corollary  \ref{Steinb} we remove this assumption). We also show that if $R = k[\epsilon]/\epsilon^4$ where $k$ is a field, then for any elements $f$,
$g$ and $h$ from $k((u))((t))$ there is an identity:
$$
(1+ \epsilon f, 1+ \epsilon g, 1+ \epsilon h) = 1+ \epsilon^3 \Res f dg \wedge dh \mbox{.}
$$
Thus, we obtain the two-dimensional residue (for two-dimensional
local fields). From Section~\ref{main}   it follows that the
two-dimensional  Contou-Carr\`{e}re symbol is invariant under the
change of local parameters $u $ and $t$ in $R((u))((t))$.

In Section~\ref{Kapp}, we outline how to obtain the Contou-Carr\`{e}re symbols via algebraic $K$-theory, developing some ideas suggested to us by one of the editors. It is widely believed that the Contou-Carr\'{e}re symbols can be obtained from certain boundary maps in algebraic $K$-theory (e.g. \cite[Remark 4.3.7]{KV}). However, it seems a detailed comparison did not exist in literature before.

Finally, in Section~\ref{cft} we relate the  two-dimensional  Contou-Carr\`{e}re symbol to the two-dimensional
class field theory.
For $R=\df_{q}[s]/s^{n+1}$, where $\df_q$ is a finite field, we derive from the two-dimensional Contou-Carr\`{e}re
symbol the two-dimensional generalization of the Witt symbol introduced and studied by A.~N.~Parshin.
Our reciprocity laws for  the  two-dimensional  Contou-Carr\`{e}re symbol imply the reciprocity laws for
the generalization of the Witt symbol.   We interpret the reciprocity laws for the two-dimensional tame
symbol (studied in~\cite{OsZh})
and for the Witt symbol as indication on some subgroup in the   kernel of the global  reciprocity map in
two-dimensional class field theory.
We also relate the reciprocity laws for the generalization of the Witt symbol with the reciprocity laws
obtained earlier by  K.~Kato and S.~Saito.

\bigskip

\noindent\bf Notation. \rm  By $\Aff$ we denote the category of
affine schemes, i.e. the category opposite to the category of
commutative rings. We equip $\Aff$ with the flat topology.

By an ind-scheme we always mean an object $ \mbox{``}
\mathop{\underrightarrow{\lim}}\limits_{i \in I} \mbox{''} \, X_i
\in \on{Ind} \Aff$, where all the structure morphisms $X_i  \to X_j$
are closed embeddings of affine schemes, and $I$ is a directed set.

For any functor $\mathcal F$ from the category $\Aff$ to the
category of sets we denote by $L\mathcal F$ the loop space functor
which assigns the set $L {\mathcal F } (R) = {\mathcal F}(R((t))) $
to every ring $R$.

We introduce the formal scheme $\n=\on{Spf} {\mathbf Z}[[T]]$. In
other words, for $R\in\Aff$, $\n(R)$ is the set of all nilpotent elements of $R$, i.e. the nil-radical of $R$. (Sometimes
$\n(R)$ is denoted by $\n R$ for simplicity.)

\section{One-dimensional Contou-Carr\`{e}re symbol}  \label{first}
In this section we recall some known facts about the one-dimensional
Contou-Carr\`{e}re symbol.

First recall the following statement, see, e.g.,~\cite[Lemme(1.3)]{CC}, \cite[\S0]{CC2}.
Let $R$ be a commutative ring. Then for any invertible element $s =
\sum\limits_{i > -\infty}^{\infty} a_i t^i \in R((t))^*$ there is
a decomposition of $R$ into finite product
of rings $R_i$:
\begin{equation} \label{de}
R = \bigoplus_{i=1}^{N} R_i
\end{equation}
such that if $s = \mathop{\oplus}\limits_i s_i$ under the induced
decomposition
\begin{equation} \label{de2}
R((t)) = \bigoplus_{i=1 }^N R_i((t))  \mbox{,}
\end{equation}
then every $r= s_i$ can be uniquely
decomposed into
the following product in $R_i((t))$:
\begin{equation} \label{decomp}
r=r_{-1} \cdot r_0 \cdot  t^{\nu(r)} \cdot r_1 \mbox{,}
\end{equation}
where $r_{-1} \in 1+ t^{-1} \cdot \n R_i[t^{-1}]$, $r_0 \in R_i^*$,
$r_1 \in 1 + t \cdot R_i[[t]]$. In addition, such a
decomposition~\eqref{de} is unique if we require $\nu(s_i) \ne
\nu(s_j) $ for any $i \ne j$.

Let us rephrase
decompositions~\eqref{de}--\eqref{decomp}. Let $L\bbG_m$ be the loop
group of the multiplicative group $\bbG_m$, where  $\bbG_m(R)=R^*$ for any commutative ring
$R$. Let $\bbW$, $ \widehat{\bbW}$ be  contravariant functors from
the category $\Aff$ to the category of abelian groups
 (or covariant functors from the category of commutative rings) defined in the following way: for any commutative
ring $R$
$$
\bbW(R) = \left\{ 1 + \sum_{i=1}^\infty b_i t^i \mid \quad b_i \in R \right\}
\mbox{,}
$$
$$
\widehat{\bbW}(R) = \left\{ 1+\sum_{i=-1}^{-n}c_i t^i \mid\quad n\in
{\mathbf Z}_{\geq 0}, c_i \in \n R \right\} \mbox{.}
$$
Note that the functor $\bbW$ with its group structure is usually called the
additive (big) Witt vectors. Then $\bbW$ is represented by an affine
scheme
$$
\bbW = \spec {\mathbf Z}[b_1,b_2,\ldots] \mbox{,}
$$
while the functor $\widehat{\bbW}$ is represented by an ind-scheme
$$
\widehat{\bbW}=\mathop{\underrightarrow{\lim}}\limits_{
\{\epsilon_i\}}\ \spec {\mathbf Z}[c_{-1},c_{-2},\ldots]/
I_{\{\epsilon_i\}}  \mbox{,}
$$ where
the limit is taken over all the sequences
$\{\epsilon_i\}=(\epsilon_{-1},\epsilon_{-2},\ldots)$ of
non-negative integers such that all but finite many $\epsilon_{i}$
equal $0$, and the ideal $I_{\{\epsilon_i\}}$ is generated by
elements $c_{i}^{\epsilon_{i}+1}$ for all $i <0$.

Denote by $\bbZ$ the constant group scheme over $\spec {\mathbf Z}$
with the fiber equal to the additive group of integers ${\mathbf
Z}$. In other words, the group scheme $\bbZ = \coprod\limits_{i \in
{\mathbf Z}} (\spec {\mathbf Z})_i$, and $\bbZ(R)$ is the group of
locally constant functions on $\spec R$ with values in ${\mathbf
Z}$. The group $\bbZ(R)$
naturally embeds into the group $R((t))^*$ in the following way.
Any $f \in \bbZ(R)$ determines decomposition~\eqref{de} and the set
of integers $n_i$, where $1 \le i \le N$. Then $f \mapsto
\mathop{\oplus}\limits_i t^{n_i} \in R((t))^*$ under
decomposition~\eqref{de2}. We can now summarize
decompositions~\eqref{de}--\eqref{decomp}  and above reasonings in
the following lemma.

\begin{lem} \label{decom}
There is a canonical isomorphism of group ind-schemes
$$
L\bbG_m \simeq \widehat{\bbW} \times \bbZ \times \bbG_m \times \bbW
\mbox{.}
$$
\end{lem}

Let $\nu: L\bbG_m\to\bbZ$ be the projection to the $\bbZ$ factor
under the above decomposition.

\medskip

Now we recall the one-dimensional Contou-Carr\`{e}re symbol (or
simply Contou-Carr\`{e}re symbol) (see~\cite[\S 2.9]{Del2},
\cite{CC}).
\begin{lem-def} \label{ld}
The Contou-Carr\`{e}re symbol is the unique bimultiplicative,
anti-symmetric map
$$
(\cdot ,\cdot) \: : \: L\bbG_m \times L\bbG_m  \lrto \bbG_m
$$
such that if $ \mathbf{Q} \subset R$ and $f,g \in
L\bbG_m(R)=R((t))^*$, then
\begin{equation} \label{cexp}
(f,g) =  \exp \res ( \log f \cdot \frac{dg}{g}) \quad \mbox{when}
\quad f\in \widehat{\bbW}(R)\times \bbW(R)  \quad (\mbox{see Lemma}
~\ref{decom}) \mbox{,}
\end{equation} \label{ag}
\begin{equation}  \label{mm}
(a,g)= a^{\nu(g)}  \quad \mbox{when } \quad a \in R^* \mbox{,}
\end{equation}
\begin{equation} \label{tt}
(t,t)=(-1,t)=-1  \mbox{.}
\end{equation}
We note that formula~\eqref{cexp} is well-defined, since  $\res (
\log f \cdot \frac{dg}{g}) \in \n R$.
\end{lem-def}
\begin{rmk} \label{pow} The expression $a^{\nu(g)}$ in~\eqref{mm} is defined in the following way. The element
 $\nu(g) \in \bbZ(R)$ determines  decomposition~\eqref{de}
and the set of integers $n_i$, where $1 \le i \le N$. Let $a =
\mathop{\oplus}\limits_i a_i$ with respect to this decomposition,
then $a^{\nu(g)} = \mathop{\oplus}\limits_i a_i^{n_i}$. Further we
will also use the following expression $t^{\nu(g)}=
\mathop{\oplus}\limits_i t^{n_i} \in R((t))^*$.
\end{rmk}
\begin{rmk} Our expression for $(f,g)$ is inverse to the corresponding expression from~\cite{Del2}.
\end{rmk}

We recall the uniqueness. First, note that that if $ \mathbf{Q} \subset R$, then the
symbol is uniquely defined by the conditions given above. In
general, if $X,Y$ are two flat ${\mathbf Z}$-schemes and
$f_{{\mathbf Q}}:X_{{\mathbf Q}}\to Y_{{\mathbf Q}}$ is a morphism,
then there is at most one morphism $f:X\to Y$ that extends
$f_{{\mathbf Q}}$. Now the uniqueness follows from the fact that
$L\bbG_m$ is represented by an inductive limit of schemes which are
flat over ${\mathbf Z}$.

Next we recall the existence. Note that the symbol
$(f,g)$ defined by the formula~\eqref{cexp} can be expressed as a formal series on
coefficients of $f,g \in R((t))^*$. To
extend this definition to arbitrary ring $R$ amounts to show that this formal series is defined
over~${{\mathbf Z}}$. In formula~\eqref{decomp} the element $r_{-1}$
can be uniquely decomposed as
$$
r_{-1} = \prod_{i < 0}^{i > -\infty} (1 - d_i t^i) \mbox{, where}
\quad d_i \in \n R \mbox{,}
$$
and the element $r_1$ can be uniquely decomposed as
$$
r_1 = \prod_{i >0}^{\infty} (1 - c_i t^i ) \mbox{, where} \quad c_i
\in R \mbox{.}
$$
Using the bimultiplicativity of $(f,g)$ and the continuity of
expression~\eqref{cexp}  with respect to the natural topology on $\bbW$
given there by the congruent subgroups $1+t^n R[[t]]$  it is enough to show that $(1+g_1t^{i_1}, 1+ g_2t^{i_2})$ is a formal series on $g_1,g_2$ defined over~$\mathbf Z$,
where $1+g_kt^k \in R((t))^*$ and $i_1, i_2\in ({{\mathbf
Z}} \setminus 0)^2$. This leads us to the following explicit formula
for the Contou-Carr\`{e}re symbol. Let $f,g \in
R((t))^*$, and decompose
$$
f = \prod_{i < 0}^{i > -\infty} (1-a_it^i) \cdot a_0 \cdot
t^{\nu(f)} \cdot \prod_{i > 0}^{\infty} (1-a_it^i) \mbox{,}
$$
$$
g = \prod_{i < 0}^{i > -\infty} (1-b_it^i) \cdot b_0 \cdot
t^{\nu(f)} \cdot \prod_{i > 0}^{\infty} (1-b_it^i)  \mbox{,}
$$
then
\begin{equation} \label{multform}
(f,g)= (-1)^{\nu(f) \nu(g)} \frac {a_0^{\nu(g)} \prod_{i>0}^{\infty}
\prod_{j>0}^{j < \infty} (1-a_i^{j/(i,j)} b_{-j}^{i/(i,j)})^{(i,j)}}
{b_0^{\nu(f)}    \prod_{i>0}^{i < \infty} \prod_{j>0}^{\infty}
(1-a_{-i}^{j/(i,j)} b_{j}^{i/(i,j)})^{(i,j)} } \mbox{.}
\end{equation}
This proves the existence.

Now we see that the Contou-Carr\`{e}re symbol is bimultiplicative.
This is clear from expression~\eqref{cexp} if $ \mathbf{Q} \subset R$.  We consider two
maps from $X = L\bbG_m \times L\bbG_m  \times L\bbG_m$ to $\bbG_m$.
Let $x,y,z \in L\bbG_m(R)$, then the first map is given as $(xy,z)$,
and the second map is given as $(x,z)(y,z)$. As these two maps coincide when $ \mathbf{Q} \subset
R$ and $X$ is
represented by ind-flat schemes over
${\mathbf Z}$, these two maps coincide for any ring $R$. Therefore $(\cdot, \cdot)$ is bimultiplicative with respect to the first
argument. By the same argument, it is also bimultipicative with respect to the second argument.

\begin{prop}[Steinberg property] Let $f, 1-f \in L\bbG_m(R)= R((t))^*$. Then $(f,1-f)=1$.
\end{prop}
\begin{proof}
Using Lemma~\eqref{decom} we will prove that condition $x + y =1$
for $x \times y \in R((t))^* \times R((t))^*$ defines an ind-scheme,
which is an inductive limit of flat schemes over ${\mathbf Z}$. Indeed,
it is clear that it is enough to prove it  by restriction to every
connected component of $L\bbG_m  \times L\bbG_m$ which is uniquely
defined by the pair of integers $(\nu(x), \nu(y))$. If we fix such a
pair, then we can write $x= \sum_{i > -\infty}^{\infty} e_i t^i$ and $y= \sum_{i > -\infty}^{\infty} g_i t^i$, where
$e_{\nu(x)}, g_{\nu(y)} \in R^*$, and $e_{i}, g_{j} \in \n R$ for any $i < \nu(x), j<\nu(y)$.
The condition $x+y =1$ is equivalent to: $e_0 +g_0 =1$ and $e_i + g_i =0$ for every
integer $i \ne 0$.
It is enough to study the subscheme defined by these equations in the ind-scheme $V_1 \times V_2$,
where
$$
V_1= \mathop{\underrightarrow{\lim}}\limits_{i \to -\infty} \left( \prod_{n=i}^{\nu(x)}  V_{1,n} \right ) \times \spec {\mathbf Z}[e_{\nu(x)+1},e_{\nu(x)+2},\ldots] \mbox{,}
$$
and $V_{1,{\nu(x)}} =\spec {\mathbf Z}[e_{\nu(x)}, e_{\nu(x)}^{-1}]=\bbG_m$, $V_{1, m} = \on{Spf} {\mathbf Z}[[e_m]]=\n$ for $m < \nu(x)$;
$$
V_2= \mathop{\underrightarrow{\lim}}\limits_{i \to -\infty} \left( \prod_{n=i}^{\nu(y)}   V_{2,n} \right) \times \spec {\mathbf Z}[g_{\nu(y)+1},g_{\nu(y)+2},\ldots] \mbox{,}
$$
and $V_{1,{\nu(y)}} =\spec {\mathbf Z}[g_{\nu(y)}, g_{\nu(y)}^{-1}]=\bbG_m$, $V_{2, m} =\on{Spf} {\mathbf Z}[[g_m]]= \n$ for $m < \nu(y)$.
Fix some integer $k > \max(\nu(x), \nu(y), 0) +1 $, and write
$$\spec {\mathbf Z}[e_{\nu(x)+1},e_{\nu(x)+2},\ldots] = \left( \prod_{n=\nu(x)+1}^{k-1} V_{1,n} \right) \times  \spec {\mathbf Z}[e_k,e_{k+1},\ldots] \mbox{,}$$
$$\spec {\mathbf Z}[g_{\nu(y)+1},e_{\nu(y)+2},\ldots] = \left( \prod_{n=\nu(y)+1}^{k-1} V_{2,n} \right) \times  \spec {\mathbf Z}[g_k, g_{k+1},\ldots] \mbox{,}$$
where  $V_{1,n}=  \spec {\mathbf Z}[e_n] = \bbG_a$ and $V_{2,n}=  \spec {\mathbf Z}[g_n] = \bbG_a$.
Hence it follows
\begin{eqnarray} \nonumber
V_1 \times V_2 /(x+y =1) = \mathop{\underrightarrow{\lim}}\limits_{i \to -\infty}
\left(  \prod_{n=i}^{k-1} V_{1,n} \times V_{2,n } / (e_n +g_n = \delta_{n,0}) \right)
\times \\ \label{sch} \times \spec {\mathbf Z}[e_k,e_{k+1}, \ldots, g_k, g_{k+1},\ldots]/ (e_k=g_k, e_{k+1} = g_{k+1}, \ldots) \mbox{,}
\end{eqnarray}
where $V_{j,n}$ (for $j=0$ or $j=1$) is one of the following schemes: $\n$,  $\bbG_a$, $\bbG_m$.
The  scheme in  expression~\eqref{sch} is an affine space.
By analyzing all possible cases for $V_{j,n}$, it is easy to see that $V_{1,n} \times V_{2,n } / (e_n +g_n = \delta_{n,0})$ is also (ind)-flat over ${\mathbf Z}$ (or possibly empty). Therefore, $V_1 \times V_2 /(x+y =1)$ is ind-flat over ${\mathbf Z}$ (or possibly empty).

Therefore it is enough to check the Steinberg property only for the
case $ \mathbf{Q} \subset R$. In this case this property follows from Lemma-Definition~\eqref{ld} by explicit
calculations
with series in formula~\eqref{cexp}. Indeed, we just need to consider
three cases: $\nu(f) > 0$, $\nu(f) =0$, $\nu(f) <0$, where we can
for simplicity assume that $N=1$ in decomposition~\eqref{de}. See
the analogous but more difficult analysis  for  the  two-dimensional
Contou-Carr\`{e}re symbol in Proposition~\ref{st}.
\end{proof}

\medskip

The following reciprocity law for the Contou-Carr\`{e}re symbol was
proved in~\cite{AP}, in~\cite[\S 3.4]{BBE}\footnote{The statement in~\cite[\S 3.4]{BBE} is more general than we formulate here. In {\em loc.cit} the authors proved the reciprocity law for a smooth projective family of curves over any base ring $R$.} and later by another methods in~\cite{Pal}. A proof via $K$-theory, as suggested by one of the editors, is outlined in Remark~\ref{K-th-res} later.
\begin{thm} \label{thres1}
Let $C$ be a smooth projective algebraic curve over an algebraically
closed field $k$. Let $R$ be a finite local $k$-algebra.
 Let $K = k(C) \otimes_k R$.
Let $f,g \in K^*$. Then the following product in $R^*$ contains only
finitely many non-equal to $1$ terms and
$$
\prod_{p \in C} (f,g)_p =1  \mbox{,}
$$
where $( \cdot , \cdot )_p$ is the Contou-Carr\`{e}re symbol on $K_p
\otimes_k R$, $K_p = k((t_p))$ is the completion field of the point
$p$ on $C$, and $K \hookrightarrow K_p \otimes_k R$ for any $p \in
C$.
\end{thm}

The authors of~\cite{AP} and~\cite{BBE} realized the
Contou-Carr\`{e}re symbol as the commutator of some central
extension. A consequence is that this symbol is invariant under the
change of the local parameter $t$ in the following sense. Any $t'
\in R((t))^*$ with $\nu(t')=1$ defines a well-defined continuous
automorphism $\phi_{t'}$ of the ring $R((t))$ by the rule: $\sum a_i
t^i \mapsto \sum a_i {t'}^i$, see e.g.~\cite[\S~1]{Mo} . Then for
any $f , g \in R((t))^*$ we have that $(f,g) = (\phi_{t'}(f),
\phi_{t'}(g))$.

\begin{rmk}
One can regard the Contou-Carr\`{e}re symbol as some
deformation of the usual tame symbol. Indeed, if $R=k$ is a field,
then $(\cdot, \cdot)$ is the tame symbol. If $R= k[\epsilon]/
\epsilon^3$, then for any $f,g \in k((t))$ we have that $(1+
\epsilon f, 1 + \epsilon g )= 1 + \epsilon^2 \res f dg $.
\end{rmk}

\section{Definition of the two-dimensional  Contou-Carr\`{e}re symbol} \label{2cc}
\subsection{Double loop group of $\bbG_m$} \label{dlg}
Let us define the double loop group of $\bbG_m$ as
$L^2\bbG_m=L(L\bbG_m)$, i.e. $L^2\bbG_m(R)=R((u))((t))^*$ for
every commutative ring $R$. Let us show that
\begin{prop} \label{2-loop}
$L^2\bbG_m$ is represented by an ind-(affine) scheme  over ${\mathbf
Z}$.
\end{prop}
\begin{proof} We have from Lemma~\ref{decom} that
\begin{equation} \label{deccon2}
L^2\bbG_m\simeq L\widehat{\bbW}\times L\bbZ\times L\bbG_m\times
L\bbW \mbox{.}
\end{equation}
The proposition then is the consequence of the following Lemmas
\ref{conn}-\ref{loop of hatW}.
\end{proof}

\begin{lem} \label{conn}
The natural map $\bbZ \to L \bbZ$ is an isomorphism of functors.

Concretely, let $R$ be a commutative ring, then any decomposition of
the ring $ R((t))$ into the product of two rings:
$$
R((t)) = L_1 \oplus L_2
$$
is induced by the decomposition of the ring $R$ into the product of
two rings:
$$
R =  R_1 \oplus R_2 \mbox{,}
$$
such that $L_1= R_i((t))$, $L_2=R_{j}((t))$, where $i,j \in \{1,2 \}$,
$i \ne j$.
\end{lem}
\begin{proof}
Every non-trivial idempotent $e $ in a commutative ring $A$, i.e,
$e^2=e$, $e \ne 0$, $e \ne 1$, defines the non-trivial
decomposition: $A = Ae \oplus A(1-e)$. Therefore the study of
decompositions of rings  is the same as the study of idempotents in
this ring.

Let $R_{{\rm red}} = R / \n R$. As the
topological spaces of $\spec  R$ and $\spec  R_{{\rm red}}$
coincide. Therefore idempotents of the ring $R$ are in one-to-one
correspondence with idempotents of the ring $R_{{\rm red}}$ under
the natural map $R \to R_{{\rm red}}$.

 Suppose that there is an element  $f \in R((t))$ such that
 $f^2=f$ and $f \ne 1, f\ne 0$.
Let $\tilde{f}$ be the image of the element $f$ in $R_{{\rm
red}}((t))$. Then it is easy to see that $\tilde{f}^2 =\tilde {f}$
implies $\tilde{f} = \sum_{n \ge 0} a_n t^n$,  where $a_0^2=a_0$,
$a_0 \in R_{{\rm red}}$. If $a_0 =0$, then $\tilde {f}=0$. If
$a_0=1$, then $\tilde {f}=1$, since $\tilde{f}$ is invertible in
this case. If
 $a_0 \ne 0 $, $a_0 \ne 1$, then consider
$R_{{\rm red}} = B_1 \oplus B_2$, where $B_1= R_{{\rm red}} a_0$,
$B_2= R_{{\rm red}} (1-a_0)$. It is clear that the image of
$\tilde{f}$ in $B_1((t))$ is equal $1$, and the image of $\tilde{f}$
in $B_2((t))$ is equal $0$. Hence we have $\tilde{f} = a_0$.

 If
 $\tilde{f} =a_0 =0$, then $f \in (\n R) ((t))$ and $f^2 =f$. Hence we have  $f (1-f) =0$ and $1-f \in R((t))^*$. The latter
 is possible only if $f =0$.

 Suppose that  $\tilde{f} = a_0=1$. Then we have  $f = 1 + f'$, where $f' \in (\n R)((t))$. We will prove that $f$ is
 invertible
 in $R((t))$. Together with $f^2=f$ it would imply that $f=1$.  So, we have to prove that $f = 1+f'$ is invertible in
 $R((t))$. Let $f' = \sum_{n > -\infty}^{\infty} f_n t^n$, where $f_n \in \n R$ for any $n \in {{\mathbf Z}}$.
 The elements $f$
  and $(1+f_0)^{-1}f$ are both invertible or not invertible in $R((t))$. Therefore,  after replacement $f$
  by $(1+f_0)^{-1}f$, we can assume that $f_0=0$. Let $I'$ be the ideal in $R((t))$ which is  generated by all elements
  $f_n$ with $ n< 0$. Then $I'$ is a nilpotent ideal. After replacement $f$ by $(1+ \sum_{n>0}^{\infty} f_n t^n)^{-1}f$
  we can
  assume that $f = 1+ f'$, where $f' \in I'((t))$.
 Let $f' = \sum_{n > \infty}^{\infty} f_n t^n$. Again, as above, we can assume that $f_0 =0$. After replacement $f$ by
 $(1 + \sum_{n < 0}^{n > -\infty} f_n t^n)^{-1}(1 + \sum_{n>0}^{\infty} f_n t^n)^{-1}f$ we can assume that
  $f \in I'^2((t))$.
  Repeating this process a sufficient number of times we will obtain that the element $f$ is invertible.

  Thus we have $\tilde{f} = a_0$, where $a_0 \ne 0$, $a_0 \ne 1$, as $f \ne 0$, $f \ne 1$.
  Let $b_0 \in R$ be the only idempotent in the ring $R$ such that the image of $b_0$ in the ring $R_{{\rm red}}$ is $a_0$.
  Consider
$R = C_1 \oplus C_2$, where $C_1= R b_0$, $C_2= R (1-b_0)$. By the
cases considered above we know that the image of $f$ in $C_2((t))$
is  $0$, and the image of $f$ in $C_1((t))$ is $1$. Therefore, $f = b_0$.
\end{proof}

\begin{lem}\label{loop of W}
The loop group of $\bbW$, denoted by $L\bbW$ is represented by an
ind-(affine) scheme, which is an "inductive limit" of
infinite-dimensional affine spaces over ${\mathbf Z}$.
\end{lem}
\begin{proof}
Let $I=\prod\limits_{i>0}{\mathbf Z}$ be the index set with a
partial order given as $(k_1,k_2,\ldots)\leq (l_1,l_2,\ldots)$ if
$k_i\leq l_i$ for all $i$.

Fix $\underline{k}=(k_1,\ldots)\in I$, consider the functor
$(L\bbW)_{\underline{k}}$ which represents
$$
(L\bbW)_{\underline{k}}(R) = \left\{ 1 + f \mid f = \sum_{i > 0,j \geq
k_i} f_{ij} u^j t^i, \ f_{ij} \in R \right\} \mbox{.}
$$
Clearly, $(L\bbW)_{\underline{k}}=\spec {\mathbf Z}[f_{ij}\mid
i\geq 0, j\geq k_i]$ is an affine space, and
$L\bbW=\mathop{\underrightarrow{\lim}}\limits_{\underline{n}\in
I}(L\bbW)_{\underline{n}}$.
\end{proof}

\begin{lem}\label{loop of hatW}
The loop group of $\widehat{\bbW}$, denoted by $L\widehat{\bbW}$ is
represented by an ind-(affine) scheme.
\end{lem}
\begin{proof}
Let us consider the loop space $L\n$ of the formal scheme $\n$, i.e.
$L\n(R)=\n (R((t)))$. We have
$\n=\mathop{\underrightarrow{\lim}}\limits_n \n_n$, where
 $ \n_n = \spec {\mathbf Z}[T]/ T^n$. Since   $L \n_n$ is an ind-affine scheme over ${\mathbf Z}$, we obtain that
$L\n=\mathop{\underrightarrow{\lim}}\limits_n L \n_n $ is also an
ind-affine scheme.

Observe that for any commutative ring $R$, elements in
$L\widehat{\bbW}(R)$ can by written as
$f=\sum_{i<0}^{>-\infty}f_it^i$, with $f_i\in \n (R((u)))$.
Therefore,
\begin{equation} \label{presen}
L\widehat{\bbW}=\mathop{\underrightarrow{\lim}}\limits_n(L\n)^n.
\end{equation}
\end{proof}

Observe that unlike the case of $L\bbG_m$, the functor $L^2\bbG_m$
is not represented as an inductive limit of flat schemes over
$\mathbf Z$ when we use decomposition~\eqref{deccon2} and
formula~\eqref{presen}. Indeed, the ind-scheme $L\n_n$ is not flat.
Observe that we can write
\begin{equation}  \label{nilind}
L\n_n=\mathop{\underrightarrow{\lim}}\limits_m L_m\n_n, \quad L_m\n_n(R) = \left\{ f = a_mt^m+a_{m+1}t^{m+1}+\cdots \mid a_i \in R \; \mbox{,}
\; f^n =0 \right\}\mbox{.}
\end{equation}
Therefore,
\[L_m\n_n =\spec A_{m,n}, \quad \mbox{where } A_{m,n}={\mathbf Z}[a_m, a_{m+1},
\ldots]/(a_m^n, na_{m}^{n-1}a_{m+1},\ldots).\] In particular,
$L_m\n_n$ is not flat over ${\mathbf Z}$, since $a_m^{n-1}a_{m+1}$
is a torsion element. However, we observe that
$I_{m,n}=(a_m^n,na_m^{n-1}a_{m+1},\ldots)$ is a homogeneous ideal
generated by elements of degree $n$.  In particular, we observe
that, for any nonzero element $a\in A_{m,n}$, there exists some $n'\gg n$ such that
for any $m'<m$, all preimages of $a$ under the surjective map
$A_{m',n'}\to A_{m,n}$ are not torsion.

We consider the following class $\calE\calF$ of ind-affine
schemes over ${\mathbf Z}$.
\begin{dfn}
 An affine ind-scheme $X/{\mathbf Z}$
belongs to $\calE\calF$ if we can write
$X=\mathop{\underrightarrow{\lim}}\limits_{i\in J} \spec R_i$ (where $J$ is a directed set) such
that for any $i\in J$, and any nonzero element $a\in R_i$, there exists some $j>i$
such that all preimages  of $a$ under the map $R_j\to R_i$ are not torsion
elements.
\end{dfn}
It is clear that for any integer $m$ we have
$\mathop{\underrightarrow{\lim}}\limits_{n} L_m\n_n \in\calE\calF$.
Hence, $L\n=\mathop{\underrightarrow{\lim}}\limits_{m,n} L_m\n_n
\in\calE\calF$.

 We have the
following easy lemma.
\begin{lem}\label{essen flat}
For any positive integer $k$, $(L^2\bbG_m)^k\in\calE\calF$.
\end{lem}
\begin{proof}
By reasonings similar to the above reasonings, for any positive
integer $l$ we have $(L\n)^l\in\calE\calF$. Observe that if $X\in
\calE\calF$ and a ring $B$ is a free $\mathbf Z$-module, then $X
\times \spec B \in\calE\calF$. Therefore for an affine ind-scheme $Y
= \mathop{\underrightarrow{\lim}}\limits_{s\in S} \spec B_s$ such
that for any $s \in S$ the ring $B_s$ is a free $\mathbf Z$-module
we have  $X\times Y\in\calE\calF$. Hence,
$(L^2\bbG_m)^k\in\calE\calF$.
\end{proof}

\medskip

It will be convenient to introduce the following group
ind-schemes
\begin{equation} \label{nnd}
\bbM=\bbW\times L\bbW,\quad \bbP=\widehat{\bbW}\times
L\widehat{\bbW}.
\end{equation}
Therefore, from the proof of Proposition~\ref{2-loop} and
Lemma~\ref{decom}, we have
\begin{equation}  \label{schdec}
L^2\bbG_m\simeq \bbP\times\bbZ\times\bbZ\times\bbG_m\times \bbM \mbox{.}
\end{equation}
Let us distinguish these two $\bbZ$s. Recall that we have the group
homomorphism $\nu:L\bbG_m\to\bbZ$, which induces $L\nu:L^2\bbG_m\to
L\bbZ\simeq\bbZ$. We denote this map by $\nu_1$. On the other hand,
we have the group homomorphism
\[L^2\bbG_m\simeq L\widehat{\bbW}\times L\bbZ\times L\bbG_m \times L\bbW\to L\bbG_m\stackrel{\nu}{\to}\bbZ,\]
denoted by $\nu_2$. Therefore, for any commutative ring $R$ and any
$f\in R((u))((t))^*$, we have the following unique
decomposition
\begin{equation} \label{decomp2}
f = f_{-1} \cdot f_0 \cdot u^{\nu_2(f)}  t^{\nu_1(f)} \cdot f_1
\mbox{,}
\end{equation}
where $f_{-1} \in (\widehat{\bbW}\times L\widehat{\bbW})(R)$, $f_0
\in R^*$, $f_1 \in (\bbW\times L\bbW)(R)$, $\nu_2(f), \nu_1(f) \in
\bbZ(R)$\footnote{See Remark~\ref{pow} about the meaning of
expressions~$u^{\nu_2(f)}$ and $t^{\nu_1(f)}$.}.

\medskip

For the later purpose, let us also introduce the following group functors.
\begin{dfn}  \label{nd} Let $\frakm$ and $\frakp$ be a group functors from the category $\Aff$  given by
\[\frakm(R)=\left\{ f \mid f \in u \cdot R[[u]] + t \cdot R((u))[[t]]
\right\}\]
\[\frakp(R)=\left\{ f \mid f \in u^{-1} \cdot R[u^{-1}] + t^{-1} \cdot
R((u))[t^{-1}] \mbox{ is nilpotent } \right\}\] with the usual addition,
where $R$ is any commutative ring.
\end{dfn}

By the same proofs of the ind-representability of $\bbM$ and $\bbP$, one can show  that $\frakm$ and $\frakp$
are also represented by ind-schemes. (For more details, see the proof of Propositions~\ref{lemma-main} and~\ref{weak version}.)

\subsection{Natural topologies} \label{nattop}

Recall that there is a natural linear topology on \linebreak $R((u))((t))$ which
comes from the topology of inductive and projective limits: a basis
of open neighborhoods of $0\in R((u))((t))$ consists of the $R$
submodules
\begin{equation} \label{top}
W_{m, \{ m_i \} }= t^m R((u))[[t]]+\sum\limits_{i \in {\mathbf Z}}
u^{m_i}t^i R[[u]] \mbox{,}
\end{equation}
for some $m \in {\mathbf Z}$, and the set of integers $\{ m_i \}$.  In other words, the $R$-module $R((u))((t))$ is a topological
$R$-module when we put  the discrete topology on the ring $R$.

We consider the induced topologies on the $R$-modules $\frakm(R)$ and $\frakp(R)$.
The $R$-module $\frakm(R)$  is a topological $R$-module such that
the base of neighborhoods of $0 \in \frakm(R)$ consists of
$R$-submodules
\begin{equation}\label{uij}
\fraku_{i,j}(R) = u^j R[[u,t]]  + t^i R((u))[[t]]
\end{equation}
with $i \in \mathbf{N}$, $j \in \mathbf{N}$. Likewise, $\frakp(R)$
has a topology with the base of neighborhoods of $0$ given by
\begin{equation}\label{uni}
\fraku_{\{n_i\}}(R)=\left\{ f \mid  f\in \sum t^{-i}u^{n_i}R[[u]] \mbox{ is
nilpotent} \right\},
\end{equation}
with $i\in \mathbf{N}$, $n_i\in  \mathbf{Z}$.

\begin{dfn} \label{catB}
We denote by $\mathcal B$  the full subcategory of the category of
commutative rings consisting of those $R$ such that
$$
\n R \quad \mbox{is a nilpotent
ideal in} \quad R \mbox{.}
$$
\end{dfn}
Note that Noetherian rings belong to $\B$.

Later on we will need the following proposition.
\begin{prop} \label{lemma-main}
Let $R$ be any commutative ring. We consider the discrete topology on the group $R^*$.
\begin{enumerate}
\item \label{pa1} Let  $\phi: \frakm_R \to {\bbG_m}_R$ be any  homomorphism of group ind-$R$-schemes.
Then  $\phi$ restricted to $R$-points is a continuous map from
$\frakm(R)$ to $R^*$.
\item \label{pa2} Let $R \in \B$ and $\psi: \frakp_R \to {\bbG_m}_R$ be any  homomorphism of group ind-$R$-schemes.
Then  $\psi$ restricted to $R$-points is a continuous map from
$\frakp(R)$ to $R^*$.
\end{enumerate}
\end{prop}
\begin{proof}
\eqref{pa1} We prove the continuous property for the map $\phi$. It is
enough to find an open subgroup $\fraku_{i,j}(R) \subset \frakm(R)$
such that $\phi (\fraku_{i,j}(R)) =1$. We assume the opposite, i.e.,
that for any $l \in \mathbf{N}$ there is an element $f_l \in
\fraku_{l,l}(R)$ such that $\phi(f_l) \ne 1$. It is clear that there
is an index $\underline{k} \in I$ (where $I$ is defined in the proof of Lemma~\ref{loop of W}) such that  for any $l \in \mathbf{N}$, $f_l \in
\frakm_{\underline{k}}(R)$, where
\[\frakm_{\underline{k}}(R)=\left\{ f=\sum_{j>0} f_{0j}u^j+\sum_{i>0,j>k_i} f_{ij}u^jt^i  \mid f_{ij}\in R \right\} \mbox{,}\] is a subgroup of $\frakm$ represented by an infinite
dimensional affine space, and $\frakm=\mathop{\underrightarrow{\lim}}\limits_{\underline{k}\in
I} \frakm_{\underline{k}}$.

We consider
the restriction $\phi_{\underline{k}}$ of the morphism $\phi$ to ${\frakm_{\underline{k}}}_R$, which in turn is
given by a map of $R$-algebras:
 $$\phi_{\underline{k}}^* \; : \;  R[T, T^{-1}]  \longrightarrow R[\{f_{ij}\}] \mbox{.}$$
 Since only finite many of variables $f_{ij}$ appear in the polynomial $\phi_{\underline{k}}^*(T)$, the morphism $\phi_{\underline{k}}$ factors through a finite-dimensional quotient
of the group scheme ${\frakm_{\underline{k}}}_R$. Therefore there is $p \in \mathbf{N}$ such that
$ \phi(\frakm_{\underline{k}}(R) \cap \fraku_{p,p}(R)) =1$. It contradicts to the assumption $\phi(f_p) \ne 1$. Thus, the map $\phi$ is continuous.

\eqref{pa2}
We prove the continuous property for the map $\psi$.
Let
\[\frakp_0=\{f\in u^{-1}R[u^{-1}]\mid f \mbox{ is nilpotent }\} \]
and
\[\frakp_-=\{f\in t^{-1}R((u))[t^{-1}]\mid f \mbox{ is nilpotent }\} \mbox{.}\]
Then the natural map $\frakp_0\times \frakp_-\to \frakp$ is an isomorphism of groups. Let us fix $k\in \mathbf{N}$ and $n\in \mathbf{Z}$, and let
\[\frakp_k=\{f\in \sum_{i=1}^k R((u))t^{-i} \mid f \mbox{ is nilpotent }\} \mbox{,}\]
\[\frakp_{k,n}=\{f\in \sum_{i=1}^k R[[u]]t^{-i}u^n \mid f \mbox{ is nilpotent }\} \mbox{.}\]
Then $\frakp_-=\mathop{\underrightarrow{\lim}}\limits_{k\in
\mathbf{N}} \frakp_{k}$.
In addition, it is enough to show that if $R\in \mathcal B$, then for every $k$, there is $n_k$ such that
$\psi(\frakp_{k,n_k}(R)) =1$.

Since $R \in \B$, there is $q \in \mathbf{N}$ such that
$(\n R)^q=0$.  Hence for any $k \in \mathbf{N}$ we have an equality
$$
\frakp_{k,n}(R) = \frakp_{k,n,q}(R)  \mbox{,}
$$
where
$$
\frakp_{k,n,q}=\{ f \in \frakp_{k,n}\mid f^q=0 \} \mbox{.}
$$
Therefore, it is enough to construct for every $k$, an integer $n_k$ such that $\psi|_{\frakp_{k,n_k,q}}=1$.
We consider the restriction $\psi_{k, 0, q}$ of the morphism $\psi$ on the affine $R$-scheme ${\frakp_{k, 0,q}}_R$.
This restriction is given by a map of $R$-algebras:
 $$\psi_{k, 0, q}^* \; : \;  R[T, T^{-1}]  \longrightarrow  A_{k, 0, q}  \mbox{,}$$
  where $A_{k, 0, q}=  R[a_{ij}  \mid 1 \le i \le k, j \ge 0]/I_q $ such that $\frakp_{k, 0,q} =
  \spec A_{k, 0, q}$. It is clear that there is $n_k \in \mathbf{N}$ such that the image of the element
  $\psi_{k, 0, q}^*(T)$  in the quotient-ring
  $$ \frac {R[\{ a_{ij} \} \mid 1 \le i \le k, j \ge 0]}{I_q +\{a_{ij}\}_{j \le n_k-1}}$$
  is equal to some $c \in R$. Then
  $$
  \psi(\frakp_{k,n_k,q}) = c \mbox{.}
  $$
Since $\psi(0)=1$, we obtain that $c=1$.
\end{proof}

If we do not assume that $R\in \mathcal B$, we can obtain a weaker statement which is sufficient for many purposes.

\begin{prop}\label{weak version}
 Let $R$ be a commutative ring.
\begin{enumerate}
\item \label{par1} Let  $\phi: \frakm_R \to {\bbG_m}_R$ be a  homomorphism of group ind-$R$-schemes.
If for every commutative $R$-algebra $R'$, and every $r\in R'$, $i>
0, j\in \mathbf{Z}$ or $i=0, j > 0$, $\phi(ru^jt^i)=1$, then
$\phi=1$ is the trivial homomorphism.
\item  \label{par2} Let
$\psi: \frakp_R\to {\bbG_{m}}_R$ be a homomorphism of  group
ind-$R$-schemes. If for every commutative $R$-algebra $R'$, and
every $r\in \n R'$, $i< 0, j\in \mathbf{Z}$ or $i=0, j < 0$,
$\psi(ru^jt^i)=1$, then $\psi=1$ is the trivial homomorphism.
\end{enumerate}
\end{prop}
\begin{proof}
Part~\eqref{par1} clearly follows from part~\eqref{pa1} of
Proposition~\ref{lemma-main}. We consider part~\eqref{par2}. Similar to the proof of Proposition~\ref{lemma-main}, it is enough to show that $\psi=1$ on
$\frakp_{k,n}$.
Let $\frakq=\{f\in R[[u]] \mid f \mbox{ is nilpotent }\}$. Then
\[\frakq^k\to \frakp_{k, n}\quad (f_1,....f_k)\mapsto u^n\sum f_it^i\]
is an isomorphism. Therefore, it is enough to prove a similar statement for homomorphisms $\psi:\frakq_R\to\bbG_{mR}$.

Let $\frakq_s=\{ f\in\frakq \mid f^s=0\}$ for $s \in \mathbf{N}$.
Then  the affine $R$-scheme ${\frakq_s}_R=\spec A_s$, where $A_s=R[a_i]/I_s$, and $I_s$ is the
ideal generated by $(a_0^s, sa_0^{s-1}a_1,\ldots)$. We equip the
ring $R[a_i]$ with a grading such that $\deg a_i=i$. Then $I_s$ is
a homogeneous ideal generated by elements
$(g_0=a_0^s,g_1=sa_0^{s-1}a_1,\ldots)$ with $\deg g_i=i$. We claim
that
$$I_s=\bigcap_{k \ge 0} (I_s+(a_k,a_{k+1},\ldots)) \mbox{.}$$
Indeed, let $f\in \cap_{k} (I_s+(a_k,a_{k+1},\ldots))$. We can
assume that each monomial appearing in $f$  has the degree less than
$l$ for some $l \in \mathbf{N}$. We write $f=f_1 + f_2$, where $f_1
\in I_s$ and $f_2 \in (a_{l},a_{l+1},\ldots)$. Then $f_1=f- f_2 \in
I_s$. Since $I_s$ is a homogeneous ideal, each homogeneous component
of $f_1$ belongs to the ideal $I_s$. Since all the monomials
appearing in $f_2$ have the degree greater than $l-1$, we must have
$f\in I_s$.

Now, let $\psi: \frakq_R \to {\bbG_{m}}_R$ be as in the assumption
of part~\eqref{par2} of this proposition. If we write
${\bbG_{m}}_R=\spec R[T,T^{-1}]$, then $\psi$ induces for every $s
\in \mathbf{N}$ the function $\psi_s=\psi^*(T)\in R[a_i]/I_s$. By
the assumption, $\psi=1$ on the ind-subgroup $\{f=
a_0+a_1u+\cdots+a_ku^k\}\subset\frakq$ for every $k \ge 0 $.
Therefore, $\psi_s=1 \mod (a_{k+1},a_{k+2},\ldots)$ for every $k$
and $s$. Therefore, $\psi_s=1$ by the above claim. This implies that
$\psi=1$.
\end{proof}

\medskip

If $\mathbf{Q} \subset R$, then it is clear that the formal series
$\exp$ and $\log$ applied to elements from $\frakp(R)$ and
from $\bbP(R)$ are well-defined\footnote{These series will contain only finite number non-zero terms.} and induce mutually inverse isomorphisms of group ind-schemes:
\begin{equation}\label{explog}
\exp \, : \, \frakp_{\mathbf{Q}} \lrto \bbP_{\mathbf{Q}}   \qquad \mbox{and}     \qquad  \log \, : \,  \bbP_{\mathbf{Q}} \lrto \frakp_{\mathbf{Q}}   \mbox{.}
\end{equation}

To have a similar statement for the group ind-schemes  $\frakm_{\mathbf{Q}}$ and  $\bbM_{\mathbf{Q}}$ we introduce a topology on the group $\bbM(R)$.

Let $R$ be any ring. The group $\bbM(R)$ is
a topological group such that the base of neighborhoods of $1 \in
\bbM(R)$ consists of  subgroups
\begin{equation}\label{Uij}
U_{i,j}(R)=1 + \fraku_{i,j}(R)
\end{equation}
with $i \in \mathbf{N}$, $j \in \mathbf{N}$.

If $\mathbf{Q} \subset R$, then the maps $\exp$ and $\log$ given by formal series are continuous maps on the topological groups $\frakm(R)$ and
$\bbM(R)$. Moreover,  we have equalities
\begin{equation}\label{explogtop}
\exp(\fraku_{i,j}(R)) = U_{i,j}(R)  \qquad  \mbox{and} \qquad   \log(U_{i,j}(R) ) =\fraku_{i,j}(R)  \mbox{.}
\end{equation}
Hence we obtain mutually inverse isomorphisms of group ind-schemes
(because of the corresponding mutually inverse isomorphisms on $R$-points of these schemes):
\begin{equation} \label{explog1}
\exp \, : \, \frakm_{\mathbf{Q}}  \lrto  \bbM_{\mathbf{Q}} \qquad \mbox{and}  \qquad  \log \, : \, \bbM_{\mathbf{Q}} \lrto \frakm_{\mathbf{Q}} \mbox{.}
\end{equation}

It is possible to introduce topology on the group $\bbP(R)$ (we will need this topology later on.)
Let $R$ be any ring.  The group $\bbP(R)$ is a topological group with
the base of neighborhoods of $1$ given by the following subsets (which are not subgroups in general):
\begin{equation}\label{Uni}
U_{\{n_i\}}(R)=1+\fraku_{\{n_i\}}(R),
\end{equation}
with $i\in\mathbf{N}$, $n_i\in\mathbf{Z}$.

We do not have any formula analogous to~\eqref{explogtop} for the subsets $\fraku_{\{n_i\}}(R)$ and  $U_{\{n_i\}}(R)$. Instead we will prove the following lemma.
\begin{lem} \label{logexp}
If a ring $R \in \B$ and $\mathbf{Q} \subset R$, then the maps $\exp$ and $\log$ are continuous maps on the topological groups $\frakp(R)$ and
$\bbP(R)$.
\end{lem}
\begin{proof}
We fix an integer $q$ such that  $(\n R)^q =0$. We will prove that the map $\exp$ is continuous. We consider arbitrary subset $U_{\{n_i\}}(R) \subset \bbP(R)$ where  $i\in \mathbf{N}$, $n_i\in\mathbf{Z}$.
Without loss of generality we can assume that all $n_i\in\mathbf{N}$ and $n_1 \le n_2 \le n_3 \le \ldots$. Let $l_i=n_{iq}$ for any $i\in \mathbf{N}$.
We consider a subset $\fraku_{\{l_i\}}(R) \subset \frakp(R)$ where  $i\in \mathbf{N}$, $l_i\in\mathbf{Z}$. We have $\exp (\fraku_{\{l_i\}}(R))
\subset  U_{\{n_i\}}(R)$. Therefore the map $\exp$ is continuous. The proof for the map $\log$ is analogous.
\end{proof}

\subsection{Two-dimensional symbol on $({L^2\bbG_m}_{\mathbf{Q}})^3$} \label{tdt}
We define the following map \linebreak $\nu \; : \; L^2\bbG_m \times L^2\bbG_m
\lrto \mathbb{Z}$ as
$$
\nu(f,g) =
\begin{vmatrix}
\nu_2(f) & \nu_2(g) \\
\nu_1(f) & \nu_1(g)
\end{vmatrix}   \mbox{.}
$$

From this definition we have the following proposition.
\begin{prop}
The map $\nu( \cdot, \cdot )$ is a bimultiplicative (with respect to
the additive structure on $\mathbb{Z}$) map and satisfies the
Steinberg relation: $\nu(f,1-f) =1$ for any $f, 1-f \in
L^2\bbG_m(R)$ where $R$ is a commutative ring.
\end{prop}

One of our main definitions is the following.
\begin{dfn} \label{cc2}
{\em The two-dimensional  Contou-Carr\`{e}re symbol} $(\cdot, \cdot,
\cdot )$ is the unique tri-mul\-ti\-pli\-ca\-ti\-ve, anti-symmetric map from
${L^2\bbG_m}_{\mathbf{Q}} \times {L^2\bbG_m}_\mathbf{Q} \times {L^2\bbG_m}_\mathbf{Q}$ to ${\bbG_m}_{\mathbf{Q}}$ such that
for any $\mathbf Q$-algebra $R$, and $f,g,h \in L^2\bbG_m(R)$,
\begin{equation} \label{2cexp}
(f,g,h) = \exp \Res (\log f \cdot \frac{dg}{g} \wedge \frac{dh}{h})
\; \mbox{when} \; (\nu_1, \nu_2)(f) = (0,0) \mbox{,} \; f_0 =1 \;
(\mbox{see}~\eqref{decomp2}) \mbox{,}
\end{equation}

\begin{equation} \label{cc3}
(a,g,h) = a^{\nu(g,h)} \quad \mbox{when} \quad a \in R^* \mbox{,}
\end{equation}

\begin{equation}  \label{cc4}
(f,g,h)= (-1)^A  \quad  \mbox{when} \quad f=u^{j_1}t^{i_1} \mbox{,}
\quad g=u^{j_2}t^{i_2} \mbox{,} \quad  h=u^{j_3}t^{i_3}
 \mbox{,}
 \end{equation}
 where
 \begin{equation} \label{sign}
 A = \nu(f,g)\nu(f,h) + \nu(g,h) \nu(g,f) + \nu(h,f) \nu(h,g) + \nu(f,g)\nu(f,h)\nu(g,h) \mbox{.}
\end{equation}
Here we set $df = \frac{\partial f}{\partial u} du + \frac{\partial
f}{\partial t} dt $, $df \wedge dg = (\frac{\partial f}{\partial u}
\frac{\partial g}{\partial t} - \frac{\partial g}{\partial u }
\frac{\partial g}{\partial t}   ) du \wedge dt$,
 and
$$
\Res \,(\sum a_{i,j} u^j t^i du \wedge dt) = a_{-1,-1}  \mbox{.}
$$
\end{dfn}
\begin{rmk}
To calculate $d f$, where $f$ is an infinite series, we used the
free two-dimensional $R$-module of continuous differentials
 $\tilde{\Omega}^1_{K/R}$, which is a
quotient module of the infinite-dimensional $R$-module
$\Omega^1_{K/R}$. See analogous constructions in~\cite[prop.~2]{O5}.
\end{rmk}
Note that the symbol given in Definition~\ref{cc2} for the
$\mathbf Q$-algebra $R$ is well-defined. First note that in formula~\eqref{2cexp} we have
\begin{equation} \label{resexp}
\Res \, (\log f \cdot \frac{dg}{g} \wedge \frac{dh}{h}) \in \n R
\mbox{.}
\end{equation}
Indeed, using decomposition~\eqref{decomp2} and tri-multiplicativity
of expression from~\eqref{resexp} (with respect to additive
structure on $R$): $\log f = \log f_{-1} + \log f_{1}$ and so on, it
is enough to verify~\eqref{resexp} for elements which appear
from decomposition~\eqref{decomp2}, which is clear.
Besides, expression from~\eqref{2cexp} is anti-symmetric. This is obvious for the permutation of  $g$ and $h$ from the definition of
the wedge product. And, for example, for the permutation of $f$ and
$g$ (when $g$ is also equal $g_{-1}g_1$ by~\eqref{decomp2}) it
follows from the following equalities:
$$
0= \Res \, d (\log f \cdot \log g  \cdot \frac{dh}{h}) = \Res (\log
f \cdot  d \log g  \wedge \frac{dh}{h}) + \Res (\log g \cdot  d \log
f \wedge  \frac{dh}{h}) \mbox{.}
$$
Finally, expressions given by formulas ~\eqref{2cexp}-\eqref{cc4}
are tri-mutltiplicative. (It is obvious for~\eqref{2cexp}
and~\eqref{cc3}. For~\eqref{cc4} it can be proved by using
calculations modulo $2$ with direct definitions of $A$
 and $\nu(\cdot, \cdot)$.)

\begin{rmk}
The formula~\eqref{2cexp} is very similar to an explicit
formula from~\cite[th.~3.6]{BM}, which is an analytic expression for
the two-dimensional tame symbol and coincides with  the holonomy of
some gerbe with connective structure and curving constructed by
three meromorphic functions on a complex algebraic surface.
\end{rmk}

One of the main problems is to extend the definition of the two-dimensional  Contou-Carr\`{e}re symbol to any ring $R$.
If such an extension exists, then it is unique by the following lemma.
\begin{lem}\label{unique} There is at
most one tri-multiplicative, anti-symmetric map  from $(L^2\bbG_m)^3
 $ to $\bbG_m$ such that this map restricted to $ ({L^2 \bbG_{m}}_{
\mathbf Q})^3   $ satisfies properties~\eqref{2cexp}-\eqref{sign}.
\end{lem}
\begin{proof} Clearly, such map over ${\mathbf Q}$ is unique. We want to extend the uniqueness to ${\mathbf Z}$.
Observe that the argument as in the 1-dimensional case can not be
applied directly. The problem, as we explained, is that $L^2\bbG_m$
is not inductive limit of flat $\mathbf Z$-schemes. However, recall
that from Lemma~\ref{essen flat}, $(L^2\bbG_m)^3\in\calE\calF$. Then
the following general fact will imply the lemma.

Let $X=\mathop{\underrightarrow{\lim}}\limits_{i \in J}\spec
R_i\in\calE\calF$ and $f_{\mathbf Q}: X_{\mathbf Q}\to Y_{\mathbf
Q}$ be any morphism, where $Y=\spec R$ is a ${\mathbf Z}$-scheme. We
prove that there is at most one extension $f:X\to Y$. Let $f_1,f_2$
be two such extensions. Then $f_l$ ($l=1,2$) is determined by a
compatible family $\{ f_{l,i}\mid i\in J \}$ of the ring
homomorphisms $f_{l,i}: R\to R_i$. By assumption, for any $r\in R$,
$f_{1,i}(r)-f_{2,i}(r)$ is a torsion element in $R_i$. This implies
$f_{1,i}=f_{2,i}$ for all $i \in J$, since $X \in \calE\calF$.
\end{proof}

Let us turn to the question of extension of the definition of
two-dimensional  Contou-Carr\`{e}re symbol. Let us indicate here
that the argument as in the $1$-di\-men\-si\-o\-nal case does not
admit an obvious generalization. Namely, it is tempting to calculate
$(1-au^jt^i,1-bu^lt^k,1-cu^nt^m)$ and to show that the result does
make sense for any commutative ring (we will do such calculation in
Section~\ref{explicform} and show that this is indeed the case). However,
unlike the $1$-dimensional case, if we write $f_{-1}$ (as in the
decomposition \eqref{decomp2}) as $\prod (1-a_{i,j}u^jt^i)$, there
will be infinite many terms in the product. Therefore, formula
\eqref{multform} in current setting will be an infinite product a
priori, and it is not obvious at all whether such expression makes
sense (we need to restrict ourself to the rings from the
category~$\B$, see Section~\ref{explicform}). Our strategy in
Section~\ref{main}  will be to construct a map
$(L^2\bbG_m)^3\to\bbG_m$ directly using  some general categorical
formalism  developed in \cite{OsZh}, and show that this map
satisfies properties~\eqref{2cexp}-\eqref{sign}. This categorical
construction will help to prove the reciprocity laws for the symbol
on an algebraic surface and to obtain some new non-trivial
properties of the symbol such as, for example,  the invariance under
the change of local parameters.

\subsection{Explicit formulas} \label{explicform}
In this subsection we will extend the definition of the two-dimensional
Contou-Carr\`{e}re symbol to  commutative rings from the category $\B$  by
means of some explicit functorial (with respect to $R$) formulas.

By Definition~\ref{cc2} we know an explicit formula for the
two-dimensional Contou-Carr\`{e}re symbol when $\mathbf{Q} \subset
R$. We want to extend the definition of  two-dimensional
Contou-Carr\`{e}re symbol to the case $\mathbf{Q} \nsubseteq R$,  $R \in \B$ also
by some explicit formulas. Similarly to the one-dimensional case it
would be  enough to consider  the formal series for $(f,g,h)$. This
series appears from formula~\eqref{2cexp} and depends on
coefficients of $f,g,h \in L^2\bbG_m(R) = R((u))((t))^*$ (see
decomposition~\eqref{canon}). We should prove that the coefficients
of the series, which are a priori from $\mathbf{Q}$, are from
${{\mathbf Z}}$, i.e. they does not contain the denominators. We
will not prove it in a direct way. We will use some infinite product
decomposition of elements from $L^2\bbG_m(R)$. To achieve this goal
we will need some restrictions on the ring $R$ (sometimes we will assume that $R \in \B$). But, first, we have
the following lemmas.
\begin{lem} \label{lll}
Let $\mathbf{Q} \subset R$. Let $f,g, h \in L^2\bbG_m(R)$ such that
$f = 1 - a_{i,j} u^j t^i$, $g = 1- b_{k,l} u^l t^k$, $h = 1 -c_{m,n}
u^n t^m$, where $i,j,k,l,m,n \in {{\mathbf Z}}$ and $a_{i,j},
b_{k,l}, c_{m,n} \in R$. Let
$$
p_0 =
\begin{vmatrix}
l & n \\
k & m
\end{vmatrix}
\mbox{,} \qquad q_0 =
\begin{vmatrix}
n & j \\
m & i
\end{vmatrix}
\mbox{,} \qquad r_0 =
\begin{vmatrix}
j & l \\
i & k
\end{vmatrix}
\mbox{.}
$$
Let $(p_0,q_0, r_0)$ be the greatest common divisor of $p_0, q_0,
r_0$, where we set $(p_0, q_0, r_0) >0$  if $p_0, q_0, r_0 >0$, and
$(p_0,q_0,r_0) <0$ if $p_0,q_0, r_0 < 0$. Then
\begin{equation} \label{expfor}
(f,g,h) = T(a_{i,j},b_{k,l}, c_{m,n})=\left(1 -
a_{i,j}^{\frac{p_0}{(p_0,q_0,r_0)}} \cdot
b_{k,l}^{\frac{q_0}{(p_0,q_0,r_0)}} \cdot
c_{m,n}^{\frac{r_0}{(p_0,q_0,r_0)}}       \right)^{(p_0,q_0,r_0)}
\end{equation}
iff $p_0, q_0, r_0 >0$ or $p_0, q_0, r_0 < 0$. In other cases of
signs of $p_0, q_0, r_0$ we have $(f,g,h) =T(a_{i,j},b_{k,l},
c_{m,n})=1$.
\end{lem}
\begin{proof}
The proof is by direct calculation with formula~\eqref{2cexp} and
explicit series for $\log$ and $\exp$. We have to calculate the
following expression:
\begin{equation} \label{form}
\exp \Res\, \log(1 -a_{i,j} u^j t^i) \cdot d \log (1- b_{k,l}u^l
t^k) \wedge d \log (1-c_{m,n}u^nt^m) \mbox{.}
\end{equation}
We obtain
\begin{equation} \label{form1}
\log (1-a_{i,j} u^j t^i) = - \sum_{p \ge 1} \frac{a_{i,j}^p}{p}
u^{jp} t^{ip} \mbox{.}
\end{equation}
Besides
$$
d \log (1- b_{k,l} u^l t^k) = - \sum_{q \ge 1} (l b_{k,l}^q
u^{lq-1}t^{kq} du + k b_{k,l}^q u^{lq} t^{kq-1} dt  ) \mbox{.}
$$
Now we have
\begin{equation} \label{form2}
d \log (1-b_{k,l} u^l t^k) \wedge \, d \log (1 -c_{m,n} u^n t^m) =
\mathop{\sum_{q \ge 1}}\limits_{r \ge 1}
\begin{vmatrix}
l & n \\
k & m
\end{vmatrix}
\cdot b_{k,l}^q \cdot c_{n,m}^r \cdot u^{lq+nr-1} t^{kq+mr-1} du
\wedge dt \mbox{.}
\end{equation}

Using the anti-symmetric property of $(\cdot, \cdot, \cdot)$ and
formula~\eqref{form2} we obtain from expression~\eqref{form} that
$(f,g,h) =1$ if $p_0 =0$, or $q_0=0$, or $r_0=0$. Therefore further
we assume that $p_0 \cdot q_0 \cdot r_0 \ne 0$.

To calculate $\Res$ from formula~\eqref{form} we can use
formulas~\eqref{form1} and~\eqref{form2} and need to find all integers $p
\ge 1$, $q \ge 1$, $r \ge 1$ which solve the following  system of
equations:
$$\left\{
\begin{array}{rcl}
jp +lq +nr & = & 0 \\
ip + kq + mr & = & 0 \mbox{.}
\end{array}
 \right.
$$
It is clear that all the solutions of this system of equations are
given by the following formula:
$$
\{ p,q,r \} = \{ s \frac{p_0}{(p_0,q_0,r_0)} , \, s
\frac{q_0}{(p_0,q_0,r_0)}, \, s \frac{r_0}{(p_0,q_0,r_0)}  \}
\mbox{,}
$$
where we take any integer $s \ge 1$, and the following condition has
to be satisfied: $p_0, q_0, r_0 >0$ or $p_0, q_0, r_0 < 0$.

Now since
\begin{multline*}
-\sum_{s \ge 1}
\frac{a_{i,j}^{\frac{sp_0}{(p_0,q_0,r_0)}}}{\frac{sp_0}{
(p_0,q_0,r_0)}} \cdot p_0 \cdot b_{k,l}^{\frac{sq_0}{(p_0,q_0,r_0)}}
\cdot c_{m,n}^{\frac{sr_0}{(p_0,q_0,r_0)}}  = \\ = \log \left(
\left( 1- a_{i,j}^{\frac{p_0}{(p_0,q_0,r_0)}} \cdot
b_{k,l}^{\frac{q_0}{(p_0,q_0,r_0)}} \cdot
c_{m,n}^{\frac{r_0}{(p_0,q_0,r_0)}} \right)^{(p_0,q_0,r_0 )} \right)
\mbox{,}
\end{multline*}
we obtain formula~\eqref{expfor}.
\end{proof}

\begin{rmk}
We note that the same expression as expression~\eqref{expfor} appeared recently (after appearence of the first e-print version of our paper in  arXiv) in~\cite{H}.
\end{rmk}

\begin{lem} \label{lemexp1}
Let $\mathbf{Q} \subset R$. Let $f ,g  \in L^2\bbG_m(R)$ such that
$f = 1- a_{i,j}u^j t^i$, $g = 1- b_{k,l}u^l t^k $, where $i,j,k,l
\in {{\mathbf Z}}$ and $a_{i,j}, b_{k,l} \in R$.
\begin{enumerate}
\item If at least one of the following conditions is satisfied:
\begin{equation}  \label{cond}
 \begin{vmatrix}
j & l \\
i & k
\end{vmatrix}  \ne 0 \mbox{,}
\qquad jl+ik \ge 0  \mbox{,} \qquad jl =0   \mbox{,}
\end{equation}
then $(f,g,t)=S(a_{i,j}, b_{k,l})=1$.
\item If none of conditions in~\eqref{cond} is satisfied, then
$$
(f,g,t)= S(a_{i,j}, b_{k,l})=\left(1 - a_{i,j}^{\left|
\frac{l}{(j,l)}  \right|} b_{k,l}^{\left| \frac{j}{(j,l)}  \right|}
 \right)^{-  \frac{l \cdot \left| (j,l)   \right|  }{ \left| l
 \right|}}  \mbox{,}
$$
where $(j,l)$ is the greatest common divisor of two integers with
any sign.
\end{enumerate}
\end{lem}

\begin{lem} \label{lemexp2}
Let $\mathbf{Q} \subset R$. Let $f ,g  \in L^2\bbG_m(R)$ such that
$f = 1- a_{i,j}u^j t^i$, $g = 1- b_{k,l}u^l t^k $, where $i,j,k,l
\in {{\mathbf Z}}$ and $a_{i,j}, b_{k,l} \in R$.
\begin{enumerate}
\item If at least one of the following conditions is satisfied:
\begin{equation}  \label{cond2}
 \begin{vmatrix}
j & l \\
i & k
\end{vmatrix}  \ne 0 \mbox{,}
\qquad jl+ik \ge 0  \mbox{,} \qquad ik =0   \mbox{,}
\end{equation}
then $(f,g,u)=Q(a_{i,j}, b_{k,l})=1$.
\item If none of conditions~\eqref{cond2} is satisfied, then
$$
(f,g,u)= Q(a_{i,j}, b_{k,l})= \left(1 - a_{i,j}^{\left|
\frac{k}{(i,k)}  \right|} b_{k,l}^{\left| \frac{i}{(i,k)}  \right|}
 \right)^{ \frac{k \cdot \left| (i,k)   \right|  }{ \left| k
 \right|}}  \mbox{,}
$$
where $(i,k)$ is the greatest common divisor of two integers with
any sign.
\end{enumerate}
\end{lem}
Proof of Lemmas~\ref{lemexp1} and~\ref{lemexp2} follows from the
following formulas (see~\eqref{2cexp}):
$$
(f,g,t)= \exp \Res \log f d \log g \wedge \frac{dt}{t}= \exp \Res
\log f \frac{\partial \log g}{\partial u} du \wedge \frac{dt}{t}
\mbox{,}
$$
$$
(f,g,u) = \exp \Res \log f d \log g \wedge \frac{du}{u}= (\exp \Res
\log f \frac{\partial \log g}{\partial t} \frac{du}{u} \wedge
dt)^{-1}
$$
and explicit calculations with formal $\log$-series similarly to the
proof of Lemma~\ref{lll}.

If $\mathbf{Q} \subset R$, $f= 1- a_{i,j} u^j t^i$
and $g,h$ are any from the set $\{u,t \}$, then  it follows from
formula~\eqref{2cexp} that $(f,g,h)=1$.

\medskip

Recall that we introduced in~\S~\ref{dlg} two subgroups of the group
$L^2\bbG_m(R)= R((u))((t))^*$: $\bbP(R)= 1+ \frakp(R)$ and $\bbM(R)
= 1+ \frakm(R)$. We have considered in \S~\ref{nattop}  the topology on $\bbM(R)$ such
that the base of neighbourhoods of $1$ consists of subgroups $U_{i,j}(R)= 1
+\fraku_{i,j}(R)$, where $i,j \in \mathbf{N}$ (see formulas~\eqref{uij} and \eqref{Uij}). We have considered also    $\bbP(R)$ as
a topological group, where the base of neighbourhoods of $1$
consists of subsets $U_{\{ n_i\}}(R)=1 + \fraku_{\{n_i\}}(R)$, where $i \in \mathbf{N}$,
$n_i \in  \mathbf{N}$ (see formulas\eqref{uni} and \eqref{Uni}).
We will speak about infinite products in $\bbM(R)$ or in $\bbP(R)$, which will converge in these topologies.

Note that every element $f \in R((u))((t))$ can canonically be written as
\begin{equation} \label{canon}
f = \mathop{\sum_{(i,j) \in {{\mathbf Z}}^2}}\limits_{(i,j) \ge
(i_f, j_f)} a_{i,j} u^j t^i \mbox{,}
\end{equation}
where $a_{i,j} \in R$, $a_{i_f, j_f} \ne 0$, and we consider on
${{\mathbf Z}}^2$ the following lexicographical order: $(i_1,j_1)
> (i_2,j_2) $ iff either $i_1 >i_2$ or $i_1 =i_2, j_1 > j_2$.

\begin{prop}
\begin{enumerate}
\item Let $R$ be any ring. Then every $f \in \bbM(R)$ can be uniquely decomposed into the following infinite product
\begin{equation} \label{dd1}
f = \mathop{\prod_{(i,j) \in {{\mathbf Z}}^2}}\limits_{(i,j) \ge
(0,1)} (1 - b_{i,j} u^j t^i)  \mbox{, where} \quad b_{i,j} \in R
\mbox{.}
\end{equation}
\item
Let $R \in \B$.
 Then every $g \in \bbP(R)$ can be uniquely decomposed into the following infinite product
\begin{equation}  \label{dd2}
g = \mathop{\prod_{(i,j) \le (0,-1)}}\limits_{i > n_g , j > m_g} (1
-c_{i,j} u^j t^i) \mbox{, where} \quad  c_{i,j} \in \n R \mbox{,}
\quad
 n_g < 0 \mbox{.}
\end{equation}
\end{enumerate}
\end{prop}
\begin{proof}
Uniqueness of both decompositions  can be easily verified.

We explain  now how to obtain decomposition~\eqref{dd1}. Let
$$f_0 = f = 1 + \sum\limits_{(i,j) \ge (0,1)} a_{i,j}u^j t^i \mbox{.}$$ We can
define $h_0 = 1 + \sum\limits_{j \ge 1} a_{0,j} u^j=\prod\limits_{j
\ge 1} (1 - b_{0,j} u^j) $. Then it is clear that
$$f_1= f_0 h_0^{-1} \in 1 + t R((u))[[t]] \mbox{.}$$
Let $f_1 = 1 + \sum\limits_{(i,j) \ge (1, j_{f_1})} d_{i,j} u^j t^i
$. We define $h_1 = \prod\limits_{j \ge j_{f_1} } (1 + d_{1,j} u^jt)
$. Then we have $f_2 = f_1 h_1^{-1} \in 1 + t^2 R((u))[[t]]$.
Repeating this procedure we will obtain that $f_n \to 1$ when $n \to
\infty$. Thus we have decomposition~\eqref{dd1}.

We explain now how to obtain decomposition~\eqref{dd2}. Let
$$g_0=g
= 1 + \mathop{\sum\limits_{(i,j) \le (0,-1)}}\limits_{i > k_g, j \ge
l_g} a_{i,j} u^i t^j \mbox{,}$$ where $a_{i,j} \in \n R$, $k_g < 0$.
We can define $e_0= 1 + \sum\limits_{j \ge l_g}^{j \le -1} a_{0,j}
u^j=\prod\limits_{j \le -1} (1- c_{0,j} u^j)$, where the last
product contains only a finite number of multipliers. Then we have
$g_1= g_0 e_0^{-1} \in 1+ t^{-1} R((u))[t^{-1}]$. Let $g_1 = 1+
\sum\limits_{i > k_{g_1}, j \ge l_{g_1}} d_{i,j} u^j t^i$, where
$d_{i,j} \in \n R$.
 We define
$e_1 = \prod\limits_{j \ge l_{g_1}} (1+ d_{-1,j} u^j t^{-1})$. The
element $e_1$ is a well-defined  element from $R((u))((t))^*$, since $\n R$ is
a nilpotent ideal. We have that $g_2 =g_1 e_1^{-1} \in 1 + t^{-2}
R((u))[t^{-1}]$. Repeating this procedure we will obtain that
$g_{-k_{g_1}} \in 1 + t^{-k_{g_1}} \cdot (\n R)^2 ((u))[t^{-1}]$. Since $\n R$
is a nilpotent ideal, we will obtain decomposition~\eqref{dd2} after
repeating some times all this procedure.
\end{proof}

\vspace{0.3cm}

Now using decompositions~\eqref{decomp2}, \eqref{dd1}-\eqref{dd2}
and Lemmas~\ref{lll}-\ref{lemexp2} we obtain the following
definition.
\begin{dfn}[Explicit formula] \label{defcon2}
Let $R \in \B$. Let $f,g,h \in L^2\bbG_m(R)$. Let
\begin{equation} \label{eq1}
f = f_0 \cdot u^{\nu_2(f)}  t^{\nu_1(f)} \cdot \prod_{(i,j) \in
{{\mathbf Z}}^2 \setminus (0,0)} (1 - a_{i,j} u^j t^i) \mbox{,
where} \quad f_0 \in R^* \mbox{,} \quad a_{i,j} \in R \mbox{,}
\end{equation}
\begin{equation} \label{eq2}
g = g_0 \cdot u^{\nu_2(g)}  t^{\nu_1(g)}  \cdot \prod_{(k,l) \in
{{\mathbf Z}}^2 \setminus (0,0)} (1 - b_{k,l} u^l t^k) \mbox{,
where} \quad g_0 \in R^* \mbox{,} \quad b_{k,l} \in R \mbox{,}
\end{equation}
\begin{equation} \label{eq3}
h = h_0 \cdot u^{\nu_2(h)}  t^{\nu_1(h)}  \cdot \prod_{(m,n) \in
{{\mathbf Z}}^2 \setminus (0,0)}  (1 - c_{m,n} u^n t^m) \mbox{,
where} \quad h_0 \in R^* \mbox{,} \quad c_{m,n} \in R \mbox{,}
\end{equation}
then {\em the two-dimensional Contou-Carr\`{e}re symbol} from
$L^2\bbG_m(R) \times L^2\bbG_m(R) \times L^2\bbG_m(R)$ to $ R^*$ is
given as
\begin{multline}  \label{concar}
(f,g,h) = (-1)^A f_0^{\nu(g,h)} g_0^{\nu(h,f)} h_0^{\nu(f,g)} \cdot
\mathop{\mathop{\prod_{(i,j) \in {{\mathbf Z}}^2 \setminus
(0,0)}}\limits_{(k,l) \in {{\mathbf Z}}^2 \setminus
(0,0)}}\limits_{(m,n) \in {{\mathbf Z}}^2 \setminus (0,0)}
 T(a_{i,j},b_{k,l}, c_{m,n}) \ \times \\
 \times
 \mathop{\mathop{\prod_{(i,j) \in {{\mathbf Z}}^2 \setminus
(0,0)}}\limits_{(k,l) \in {{\mathbf Z}}^2 \setminus (0,0)}}
S(a_{i,j}, b_{k,l})^{\nu_1(h)} \ \cdot \mathop{\mathop{\prod_{(i,j)
\in {{\mathbf Z}}^2 \setminus (0,0)}}\limits_{(k,l) \in {{\mathbf
Z}}^2 \setminus (0,0)}} Q(a_{i,j}, b_{k,l})^{\nu_2(h)} \
\times \\
\times \mathop{\mathop{\prod_{(i,j) \in {{\mathbf Z}}^2 \setminus
(0,0)}}\limits_{(m,n) \in {{\mathbf Z}}^2 \setminus (0,0)}}
S(c_{m,n},a_{i,j})^{\nu_1(g)} \ \cdot \mathop{\mathop{\prod_{(i,j)
\in {{\mathbf Z}}^2 \setminus (0,0)}}\limits_{(m,n) \in {{\mathbf
Z}}^2 \setminus (0,0)}} Q(c_{m,n},a_{i,j})^{\nu_2(g)} \
\times \\
\times \mathop{\mathop{\prod_{(k,l) \in {{\mathbf Z}}^2 \setminus
(0,0)}}\limits_{(m,n) \in {{\mathbf Z}}^2 \setminus (0,0)}}
S(b_{k,l}, c_{m,n})^{\nu_1(f)} \ \cdot \mathop{\mathop{\prod_{(k,l)
\in {{\mathbf Z}}^2 \setminus (0,0)}}\limits_{(m,n) \in {{\mathbf
Z}}^2 \setminus
(0,0)}} Q(b_{k,l}, c_{m,n})^{\nu_2(f)} \ \mbox{,} \\
\end{multline}
where $A \in \mathbb{Z}(R)$ is given by formula~\eqref{sign}.
\end{dfn}
To verify that this definition is well-defined, we have to check
that the infinite products in~\eqref{concar} contain only a finite
number non-equal to $1$ terms. We will explain it, first, when
$\mathbf{Q} \subset R$. It follows from the following general
property of continuity  of expression~\eqref{2cexp}. Let $d, \{ d_i
\}_{i \in \mathbf{Z}}$ and $e, \{ e_i \}_{i \in \mathbf{Z}}$ be the
collections of integers. Then
 there is some open subset
$U_{\{ n_i \}}(R) = 1+ \fraku_{\{ n_i\}} \subset \bbP(R)$ and some open subgroup $U_{i,j}(R)=1+\fraku_{i,j}(R) \subset \bbM(R)$
(which depend on collections $d, \{ d_i \}$ and $e, \{ e_i \}$) such
that for any $f_1 \in U_{\{ n_i \}}(R)$, for any $f_2 \in U_{i,j}(R)$, for any $g =\mathop{\mathop{\sum}\limits_{i >d
}}\limits_{(i,j)
> (i,d_i)} g_{i,j} u^jt^i$, for any
$h =\mathop{\mathop{\sum}\limits_{i >e }}\limits_{(i,j)
> (i,e_i)} h_{i,j} u^jt^i$
 we have
\begin{equation}  \label{cont}
 \Res \, (\log f_i \cdot \frac{dg}{g} \wedge \frac{dh}{h}) =0 \mbox{,} \quad i=1,2 \mbox{.}
\end{equation}
(This formula  is equivalent to the fact that the resulting series
does not contain the non-zero coefficient at $u^{-1}t^{-1} du \wedge
dt$, and therefore it is easy to construct the corresponding open
subset $U_{\{ n_i\}}(R) \subset \bbP(R)$ and the open subgroup $U_{i,j}(R) \subset
\bbM(R)$ such that formula~\eqref{cont}
   is satisfied. To construct $f_2$ one can use also formula~\eqref{explogtop}. To construct $f_1$ it is important that $R \in \B$, see Lemma~\ref{logexp}.) Now using this and
Lemmas~\ref{lll}-\ref{lemexp2} we obtain that only a finite number
of $a_{i,j}$ give non-trivial contributions to products which
contain $T(\cdot, \cdot, \cdot)$, $S(\cdot, \cdot)$ and $Q(\cdot,
\cdot)$ in formula~\eqref{concar}. Now using the obvious
anti-symmetric properties of $T(\cdot, \cdot, \cdot)$, $S(\cdot,
\cdot)$ and $Q(\cdot, \cdot)$, we obtain the same for $b_{k,l}$, and
then for $c_{m,n}$.

Now we verify that Definition~\ref{defcon2} is well-defined in the
general case. We reduce this case to the previous one. We will find
rings $S_1, S_2 \in \B$ and elements $\tilde{f}, \tilde{g},
\tilde{h} \in S_1((u))((t))^*$ such that there is a map of rings
$S_1 \to R$, $S_1 \subset S_2$, $\mathbf{Q} \subset S_2$, and the
elements $\tilde{f} , \tilde{g}, \tilde{h}$ go to the elements
$f,g,h$  under the natural map $S_1((u))((t))^* \to R((u))((t))^*$.
For elements $\tilde{f} , \tilde{g}, \tilde{h}$ we have
decompositions~\eqref{decomp2}, \eqref{dd1}-\eqref{dd2}. Since these
decompositions are uniquely defined, they are functorial and go to
decompositions~\eqref{eq1}-\eqref{eq3}  for $f,g,h$  under the
natural map $S_1((u))((t))^* \to R((u))((t))^*$. For  $\tilde{f} ,
\tilde{g}, \tilde{h}$ the corresponding analogous
formula~\eqref{concar} contains only a finite number of multipliers
by the previous case, because $S_1((u))((t))^* \subset
S_2((u))((t))^*$ and $\mathbf{Q} \subset S_2$. Therefore
formula~\eqref{concar} is well-defined for $f,g,h$. Now the
existence of such $S_1, S_2$ and $\tilde{f}, \tilde{g}, \tilde{h}$
follows from the following lemma.
\begin{lem} \label{lift}
Let $R \in \B$. Let $f_1, \ldots, f_n \in R((u))((t))^* $. Then
there are rings $S_1, S_2 \in \B$ and elements $\tilde{f_1}, \ldots,
\tilde{f_n} \in S_1((u))((t))^*$ such that there is a map of rings
$S_1 \to R$, $S_1 \subset S_2$, $\mathbf{Q} \subset S_2$, and the
elements $\tilde{f_1} , \ldots, \tilde{f_n}$ go to the elements $f_1
,\ldots, f_n$  under the natural map $S_1((u))((t))^* \to
R((u))((t))^*$.
\end{lem}
\begin{proof}
Without loss of generality we can assume that $\nu_1(f_l) \in
\mathbf{Z} \subset \bbZ(R)$, $1 \le l \le n$ and  $\nu_2(f_l) \in
\mathbf{Z} \subset \bbZ(R)$. (Otherwise we can find a ring
decomposition $R= R_1 \oplus \ldots \oplus R_m$ such that the
previous condition  is satisfied for every $R_k$ and
 work then separately with every $R_k$.) Let $f_l = \sum\limits_{(i,j) \in \mathbf{Z}^2}  a_{l,i,j} u^j t^i$,
 where $1 \le l \le n$ and $a_{l,i,j} \in R$. According to formula~\eqref{decomp2}, for any $1 \le l \le n $ we have
 that $a_{l,i,j} \in \n R$ when $(i,j) < (\nu_1(f_l), \nu_2(f_l))$,
 and $a_{l, \nu_1(f_l), \nu_2(f_l)} \in R^*$. Let $(\n R)^m =0$. We
 define the ring
 $$
 S_1 = P^{-1}{\mathbf Z} \, [\, \{A_{l,i,j} \} \, ]/I^m  \mbox{,}
 $$
where the set of variables $\{A_{l,i,j} \}$ depends on the indices
$1 \le l \le n$ and $(i,j) \in \mathbf{Z}^2$ such that $a_{l,i,j}
\ne 0$, $P$ is a set which is  multiplicatively generated by
elements $A_{l, \nu_1(f_l), \nu_2(f_l)}$, where $1 \le l \le n$, and
$I$ is an ideal generated by all $A_{l,i,j}$ such that $(i,j) <
(\nu_1(f_l), \nu_2(f_l))$. We define a map $S_1 \to R$ which is
given on variables as $A_{l,i,j} \mapsto a_{l,i,j}$.     Now for any
$1 \le l \le n $ we define $\tilde{f_l}= \sum\limits_{(i,j) \in
\mathbf{Z}^2} A_{l,i,j} u^j t^i$. It is clear that $\tilde{f_l} \in
S_1((u))((t))^*$. We define the ring
$$
S_2=P^{-1}{\mathbf Q} \, [\, \{A_{l,i,j} \} \, ]/I^m  \mbox{,}
$$
where $\{A_{l,i,j} \}$, $P$ and $I$ are as above. The proof is
finished.
\end{proof}

Now we have the following proposition.
\begin{prop} \label{prcom}
Let $R \in \B$. The two-dimensional Contou-Carr\`{e}re symbol
$(\cdot, \cdot, \cdot)$, constructed by the explicit formula from
Definition~\ref{defcon2}, is an anti-symmetric,
 functorial   with respect to $R$ map. If $\mathbf{Q} \subset R$, then values
 of $(\cdot, \cdot, \cdot)$ calculated by Definition~\ref{cc2}  and Definition~\ref{defcon2}  coincide.
\end{prop}
\begin{proof}
The anti-symmetric property of $(\cdot, \cdot, \cdot)$ easily
follows from
 formula~\eqref{concar}, because  $T(\cdot, \cdot, \cdot)$, $S(\cdot, \cdot)$ and $Q(\cdot, \cdot)$
satisfy the anti-symmetric property.
  The functoriality follows from
the uniqueness properties  of decompositions~\eqref{decomp2} and
\eqref{dd1}-\eqref{dd2}. If $\mathbf{Q} \subset R$, then two
definitions give the same values because of
Lemmas~\eqref{lll}-\eqref{lemexp2}, the anti-symmetric property of
$(\cdot, \cdot, \cdot)$, the tri-multiplicativity of $(\cdot, \cdot,
\cdot)$ given by Definition~\ref{cc2}  and the continuity of
expression~\eqref{2cexp} (see reasonings before
formula~\eqref{cont}).
\end{proof}

\begin{rmk}\label{n-dim CC}
Starting from expression
$$
\exp \Res \, \log ( f_{1} \cdot \frac{df_2}{f_2} \wedge \cdots
\wedge \frac{df_{n+1}}{f_{n+1}}  ) \mbox{,}
$$
 where $f_1, \ldots, f_{n+1} \in R((t_n)) \ldots ((t_1))^*$ and where $\log f_1$ is well-defined,
one could similarly  define and develop the theory of the $n$-dimensional Contou-Carr\`{e}re symbol.
 We restricted ourself to the case $n=2$, since
further we will prove for  the two-dimensional Contou-Carr\`{e}re
symbol the reciprocity laws on algebraic surfaces.
\end{rmk}

\section{Properties of  the two-dimensional Contou-Carr\`{e}re symbol} \label{propert}
\subsection{Case $\mathbf{Q} \subset R$.}
We assume in this subsection that $R$ is a ring such that
$\mathbf{Q} \subset R$. We will need the following lemma.
\begin{lem} \label{lem}
 Let $f,g \in L^2\bbG_m(R)$. Then
\begin{equation} \label{pr}
 \nu(f,g) = \Res  (\frac{df}{f} \wedge \frac{dg}{g}) \mbox{.}
 \end{equation}
\end{lem}
\begin{proof}
It follows by direct calculations using
decomposition~\eqref{decomp2} and bimultiplicativity of both parts
of~\eqref{pr}. We note that if $(\nu_1, \nu_2)(f) = (0,0)$ and
$f_0=1$, then $\Res (\frac{df}{f} \wedge \frac{dg}{g}) = \Res ( d
\log f \wedge \frac{dg}{g}) = \Res \, d (\log f \frac{dg}{g}) =0$.
\end{proof}

We have the following proposition.

\begin{prop}   \label{st}
 The two-dimensional Contou-Carr\`{e}re symbol satisfies the Steinberg properties, i.e. $(f, 1-f, g)=1$ for any $f, 1-f, g \in L^2\bbG_m(R)$ (and other analogous equalities are satisfied  from the anti-symmetric property of $(\cdot, \cdot, \cdot)$).
\end{prop}
\begin{proof} If $R =R_1 \oplus R_2$, where $R_i $ is a ring, $i=1,2$, then we can prove the Steinberg property separately for elements restricted
to  $R_1$ and to $R_2$. Therefore  without loss of generality we can
assume that $\nu_1(f) \in {\mathbf Z} \subset \bbZ(R)$ and $\nu_2(f)
\in {\mathbf Z} \subset \bbZ(R)$ (and similar conditions for $1-f$).
We consider several cases.
\begin{itemize}
\item Let $(\nu_1, \nu_2)(f) > (0,0)$. Then for $f'= 1-f$ we have
$(\nu_1, \nu_2)(f') = (0,0)$ and $1 - f'_0 \in \n R$ (see
decomposition~\eqref{decomp2} for the definition of $f'_0$).
Therefore using anti-symmetric property it is enough to prove $(f',
f, g) =1$. From Lemma~\ref{lem} it is easy to see that we can
correctly  apply formula~\eqref{2cexp}  to calculate $(f',f,g)$. Now
this case follows from
\begin{multline*}
\log(1-f) \frac{df}{f} \wedge \frac{dg}{g} = -(f + \frac{f^2}{2} +
\ldots) \frac{df}{f} \wedge \frac{dg}{g} =
-(1 + \frac{f}{2} + \ldots) df \wedge \frac{dg}{g} = \\
 = -d(f + \frac{f^2}{4} + \ldots ) \wedge \frac{dg}{g}= d(- (f + \frac{f^2}{4} + \ldots) \frac{dg}{g}) \mbox{,}
\end{multline*}
since $\Res \, d(\cdots) =0$.
\item
 Let $(\nu_1, \nu_2)(f) = (0,0)$ , then $(\nu_1, \nu_2)(1-f) \ge (0,0)$. If $(\nu_1, \nu_2)(1-f) > (0,0)$, then interchanging $f$ and $1-f$ we reduce this case to the previous one. So, we have to consider the case when
  \begin{equation} \label{equ}
 (\nu_1, \nu_2)(f) = (\nu_1, \nu_2)(1-f)=(0,0) \mbox{.}
 \end{equation}
 Let $f = f_0 h $, where $f_0 \in R^*$, $h= f_{-1} f_1$ (see decomposition~\eqref{decomp2}). Then
 $$(f_0h, 1-f_0h, g) = (f_0, 1-f_0h, g) (h, 1-f_0h, g) \mbox{.}$$
 We have that $(f_0, 1-f_0h, g) = f_0 ^{\nu(1-f,g)}=1$, since  $(\nu_1, \nu_2)(1-f)=(0,0)$. From~\eqref{equ} we have that $1-f_0 \in R^*$. Let $h = 1-e$. Then
 \begin{multline*}
 (h, 1-f_0h, g) = (1-e, 1-f_0 + f_0 e, g)= \\
 =\exp \Res \, (-e - \frac{e^2}{2} - \ldots) \cdot \frac{f_0}{1-f_0} \cdot \frac{de}{1 + \frac{f_0}{1-f_0}e}
\wedge \frac{dg}{g} =\\ = \exp \Res \, \frac{f_0}{1-f_0} \cdot (-e -
\frac{e^2}{2} - \ldots) \cdot (1 -  \frac{f_0}{1-f_0}e +
(\frac{f_0}{1-f_0}e)^2 - \ldots) \cdot de
\wedge \frac{dg}{g} = \\
\\ = \exp \Res \, \Phi(e) de  \wedge \frac{dg}{g} = \exp \Res \, d \Psi(e) \wedge \frac{dg}{g} = \exp \Res \, d (\Psi(e)  \wedge \frac{dg}{g}) =1 \mbox{,}
\end{multline*}
 where $\Phi$ and $\Psi$ are some formal series from $\mathbf{Q}[[x]]$.

\item
Let $(\nu_1, \nu_2)(f) < (0,0)$. Then $(\nu_1, \nu_2)(f^{-1}) > 0$.
From the tri-multiplicativity we have
$$
(f,1-f,g)= (f,f,g)(f,-1 +f^{-1},g) \mbox{.}
$$

We claim that $(f,f,g) = (-1)^{\nu(f,g)}$ for any $f,g \in
L^2\bbG_m(R)$. Indeed, from the anti-symmetric and
tri-multiplicative property of $(\cdot, \cdot, \cdot)$ it follows
that one has to verify this equality  only when   $f$ and $g$ are
multipliers from decomposition~\eqref{decomp2}. This can be easily
done.

We have also
$$
(f,-1 +f^{-1},g)= (f,-1,g) (f,1-f^{-1},g)= (-1)^{\nu(f,g)}
(f,1-f^{-1},g) \mbox{,}
$$
$$
(f,1-f^{-1},g)= (f^{-1}, 1-f^{-1},  g)^{-1}=1  \mbox{,}
$$
where the last equality follows from the first case applied to
$f^{-1}$.

Thus in this case we obtained
$$
(f,1-f,g)= (-1)^{\nu(f,g)} \cdot (-1)^{\nu(f,g)} =1 \mbox{.}
$$
\end{itemize}
\end{proof}

\begin{rmk}
By means of  tri-multiplicativity, anti-symmetric property and
property $(f,f, g) = (-1,f,g)$ for $f,g \in L^2\bbG_m(R)$ (which
follow from Steinberg relations) one can reduce formula~\eqref{cc4}
to formula~\eqref{cc3}.
\end{rmk}

\subsection{Case $R \in \B$}
We assume in this subsection that $R$ is any ring from   $\B$.

\begin{prop} \label{deform}
Let $k$ be a field.
\begin{enumerate}
\item \label{as1} Let $R =k$. Then the two-dimensional Contou-Carr\`{e}re symbol coincides with the two-dimensional
 tame symbol.
\item  \label{as2} Let $R = k[\epsilon]/ \epsilon^4$. Then for any $f,g,h \in k((u))((t))$
\begin{equation} \label{res}
(1+\epsilon f, 1+\epsilon g, 1+ \epsilon h) = 1 + \epsilon^3 \Res \,
f dg \wedge dh \mbox{.}
\end{equation}
\end{enumerate}
\end{prop}
\begin{proof}
Assertion~\ref{as1} follows from the explicit formula for the
two-dimensional tame symbol, see~\cite{Pa0}, \cite[\S~$4$A]{OsZh}.

Assertion~\ref{as2} follows from the following calculaton when
$\mathbf{Q} \subset R$ (see formula~\eqref{2cexp})
\begin{multline*}
(1+\epsilon f, 1+\epsilon g, 1+ \epsilon h) =
\exp \Res \, \log(1 + \epsilon f) d \log (1+\epsilon g) \wedge d \log (1 + \epsilon h) = \\
= \exp \Res \, (\epsilon f - \frac{\epsilon^2 f^2}{2} + \ldots) d
(\epsilon g - \frac{\epsilon^2 g^2}{2}  + \ldots) \wedge
d (\epsilon h - \frac{\epsilon^2 h^2}{2}  + \ldots) = \\
= \exp ( \epsilon^3 \Res f dg \wedge dh) = 1 + \epsilon^3 \Res \, f
dg \wedge dh  \mod \epsilon^4 \mbox{.}
\end{multline*}

If $\mathbf{Q} \nsubseteq R$, then we use the continuity of left and
right hand sides of~\eqref{res}. Therefore it is enough to verify
this equality for elements $f,g,h$ of type $1 + \epsilon a_{i,j} u^j
t^i$. To achieve this goal we use Lemma~\ref{lll}.
\end{proof}

\begin{prop} \label{stgen}
The two-dimensional Contou-Carr\`{e}re symbol constructed by the
explicit formula from Definition~\ref{defcon2} is a
tri-multiplicative map from $L^2\bbG_m(R) \times L^2\bbG_m(R) \times
L^2\bbG_m(R) $ to $R^* $ and it satisfies the Steinberg relations.
\end{prop}
\begin{proof}
First we explain the tri-multiplicativity. Let $f_1, f_2, g,h \in
L^2\bbG_m(R)$. We want to prove, for example, that $(f_1f_2, g,h) =
(f_1,g,h)(f_2,g,h)$. Using Lemma~\ref{lift} we will find the rings
$S_1 \subset S_2$ such that $\mathbf{Q} \subset S_2$ and elements
$\tilde{f_1}, \tilde{f_2}, \tilde{g}, \tilde{h} \in S_1((u))((t))^*
$ which are mapped to the elements $f_1, f_2, g, h$. Then, by
Definition~\ref{cc2} and Proposition~\ref{prcom},
$(\tilde{f_1} \tilde{f_2}, \tilde{g}, \tilde{h}) =
(\tilde{f_1},\tilde{g},\tilde{h})(\tilde{f_2},\tilde{g},\tilde{h})$,
since $\mathbf{Q} \subset S_2$. Hence we obtain the
tri-multiplicativity.

Now we explain the Steinberg property. Let $f, 1-f, h \in L^2\bbG_m(R)$. We
want to prove that $(f,1-f,h)=1$. Without loss of generality we can
assume that $\nu_1(f) \in \mathbf{Z} \subset \bbZ(R)$ and $\nu_2(f)
\in \mathbf{Z} \subset \bbZ(R)$ (and similar conditions for $1-f$).
We consider several cases.

If $(\nu_1, \nu_2)(f) > (0,0)$, then we consider by Lemma~\ref{lift}
the rings $S_1 \subset S_2$ and an element $\tilde{f} \in
S_1((u))((t))^*$ which is mapped to the element $f$. Since $(\nu_1,
\nu_2)(\tilde{f}) > (0,0)$, we have that $1 - \tilde{f} \in
S_1((u))((t))^*$. Therefore we can apply Proposition~\ref{st} to the
ring $S_2((u))((t))^*$ and elements $\tilde{f}, 1-\tilde{f} $. Hence
this case follows.

If $(\nu_1, \nu_2)(f) = (0,0)$ , then $(\nu_1, \nu_2)(1-f) \ge
(0,0)$. If $(\nu_1, \nu_2)(1-f) > (0,0)$, then interchanging $f$ and
$1-f$ we reduce this case to the previous one. So, we have to
consider the case when $
 (\nu_1, \nu_2)(f) = (\nu_1, \nu_2)(1-f)=(0,0) $. By Lemma~\ref{lift} we have the rings $S_1 \subset S_2$ such that $\mathbf{Q} \subset S_2$, and
 $\tilde{f} \in S_1((u))((t))^*$. Let $T$ be a subset of $S_1$  which is mutiplicatively generated by $1-A_{0,0}$ (see the proof of Lemma~\ref{lift} for the definition of the element  $A_{0,0}$, or more exactly $A_{1,0,0}$.) We define the rings $S_1'= T^{-1}S_1$, $S_2'= T^{-1}S_2$. Then $S_1'$ is mapped to $R$, $S_1' \subset S_2'$ and $\mathbf{Q} \subset S_2'$. Moreover, $\tilde{f}, 1-\tilde{f}  \in S_1'((u))((t))^*$. Therefore we can  apply Proposition~\ref{st} to the ring $S_2'((u))((t))^*$. Hence this case follows.

 If $(\nu_1, \nu_2)(f) < (0,0)$, then this case can be reduced to the first case by the same method as in Proposition~\ref{st}.
  \end{proof}

\begin{rmk}
The tri-multiplicativity of two-dimensional Contou-Carr\`{e}re
symbol  will follow also from Corollary~\ref{cor} (of Theorem~\ref{th-main})
later,
 because we know  the tri-multiplicativity for the commutator (the map $C_3$) in a categorical central extension.
 The Steinberg relations will follow also from Corollary~\ref{Steinb} (of Theorem~\ref{K-th}) later, where these relations will be obtained from the product structure in algebraic $K$-theory.
 In this subsection we used elementary methods
to prove these properties.
\end{rmk}

\section{Categorical central extensions} \label{seccentr}
\subsection{Central extensions}
In~\cite{OsZh}, we defined the notion of a central extension of a
group by a Picard groupoid. We need to extend such formalism to any
topos. This is straightforward if we adapt Grothendieck's point of
view of functor of points. Let us briefly indicate how to do this.

Let $\calT$ be any topos. Let $\calP$ be a sheaf (a.k.a a stack) of
Picard groupoids in $\calT$, then it makes sense to talk about a
$\calP$-torsor (see~\cite{Br1}). Namely, a $\calP$-torsor $\calL$ is
a sheaf (a.k.a a stack) of groupoids with an action given by some
(Cartesian) bifunctor $\calP\times_\calT\calL\to\calL$ which
satisfies certain axioms. In particular, for any $U\in\calT$ it
gives rise to a bifunctor $\calP(U)\times\calL(U)\to\calL(U)$.
Moreover, for any $U\in\calT$, $\calL(U)$ is either empty or a
$\calP(U)$-torsor (see~\cite[\S~2C]{OsZh}), and for any $U$, there
is a covering $V\to U$ such that $\calL(V)$ is a $\calP(V)$-torsor.
All $\calP$-torsors form naturally a Picard $2$-stack in
$2$-groupoids in $\calT$, see~\cite[ch.~8]{Br2}. If $A$ is a sheaf
of abelian groups in $\calT$, and $BA$ is the Picard groupoid of
$A$-torsors, then an $BA$-torsor has another name as an $A$-gerbe.
(More precisely, there is a canonical equivalence between the
$2$-stack of $BA$-torsors and the $2$-stack of abelian $A$-gerbes
in~$\calT$, see~\cite[Prop.~2.14]{Br2}.)

Let $G$ be a group in $\calT$ and $\calP$ be a sheaf of Picard
groupoids in $\calT$. Then a central extension of $G$ by $\calP$ is
a rule to assign to every $U\in\calT$ and to every $g\in G(U)$ a
$\calP_U$-torsor $\calL_g$ over $\calT/U$ such that: 1) for any $V
\to U$ the $\calP_V$-torsor $\calL_g |_V$ corresponds to $g |_V \in
G(V)$, 2) for any $U\in\calT$ the $\calP_U$-torsors $\calL_g$
satisfy the properties as in \cite[\S~2E]{OsZh} which are compatible
with restrictions. If $A$ is a sheaf of abelian groups in $\calT$,
and $BA$ is the Picard groupoid of $A$-torsors, then a central
extension of $G$ by $BA$
 was described in~\cite[\S~5.5]{Del2}.

\medskip

Now let $\calL$ be a central extension of $G$ by $\calP$, and assume
that $G$ is abelian. Recall that in \cite{OsZh}, we constructed
certain maps $C_2$ and $C_3$ (generalized commutators), which have obvious generalization to
the sheaf theoretical contents. It means that we have a
bimultiplicative and anti-symmetric morphism
\begin{equation}  \label{gencom2}
C_2^{\calL}:G\times G\to\calP
\end{equation}
and a tri-multiplicative and anti-symmetric morphism
\begin{equation}  \label{gencom3}
C_3^{\calL}:G\times G\times G\to\pi_1(\calP) \mbox{.}
\end{equation}

\subsection{The central extension of $\GL_{\infty,\infty}$}
Now we specify the central extensions that we are considering in the
paper.

(i) Let $\Pic^\bbZ$ be the Picard groupoid of graded lines. It means
that for any commutative ring $R$, $\Pic^\bbZ(R)$ is the category of
graded line bundles over $\spec R$, together with the natural Picard
structure. It forms a sheaf (with respect to the flat topology) of
Picard groupoids over $\Aff$. Proposition~3.6 of~\cite{OsZh} remains
true in this sheaf version.

(ii) Let $\bbV$ be a 2-Tate vector space over a field $k$.
In~\cite{OsZh} we studied the central extension of $\GL(\bbV)$ by
$\Pic^\bbZ$. In the current setting, we need a sheaf version of this
construction.

The first goal to endow the group $\GL(\bbV)$ a structure as a sheaf
of groups over $\Aff$. This would be clear if we can make sense of
family of 2-Tate vector spaces (or 2-Tate $R$-modules). I.e., for any $k$-algebra $R$, some certain ``complete" tensor product $\bbV\hat\otimes_k R$ should be a 2-Tate $R$-module and $\GL(\bbV)(R)$ should be the group of automorphisms of this 2-Tate $R$-module.
However, to keep the size of the paper, we will not try to develop a full theory of 2-Tate $R$-modules here. Instead, let us specialize to the case $\bbV=k((u))((t))$. Then $\bbV\hat\otimes_kR$ should be just $R((u))((t))$, and
we follow the idea of \cite{FZ} to define a group that acts on $R((u))((t))$. More precisely, let
\[\GL_{\infty,\infty}(R)=\gl_\infty(\gl_\infty(R))^*.\]
Here for any ring $A$ (not necessarily commutative nor unital),
$\gl_\infty(A)$ is the algebra of continuous endomorphisms of
$A((t))$ as a \emph{right} $A$-module, where we consider $A$ as a
discrete topological space.

The group $\GL_{\infty,\infty}$ is a sheaf of groups over $\Aff$,
which acts on $R((u))((t))$. Explicitly, the action can be described
as follows.

If we give $A((t))$ the topological basis $\{t^i\}$, then elements
in $\gl_\infty(A)$ could be regarded as
$\infty\times\infty$-matrices $X=(X_{ij})_{i,j\in {\mathbf Z}}$
which act on $A((t))$ by the formula
$$ X(t^j)=\sum\limits_{i\in {\mathbf Z}}X_{ij}t^i. $$
It is easy to see that
\[\gl_\infty(A)=\left\{\begin{array}{l}(X_{ij})_{i,j\in {\mathbf Z}}, \,
X_{ij} \in A \ \mid \ \forall
m\in {\mathbf Z}, \, \exists \ n\in {\mathbf Z}, \\
\mbox{ such that whenever } i<m, j>n, \, X_{ij}=0
\end{array}\right\}.\]
Therefore,
\[\gl_{\infty,\infty}(R)=\left\{\begin{array}{l}
(X_{ij})_{i,j\in {\mathbf Z}}, \, X_{ij}\in\gl_\infty(R) \ \mid \
\forall
m\in {\mathbf Z}, \exists \ n\in {\mathbf Z}, \\
\mbox{ such that whenever } i<m, j>n, \, X_{ij}=0
\end{array}\right\}.\]

From this presentation, it is clear that $\gl_{\infty,\infty}(R)$
acts on $R((u))((t))$ by the following formula. If we represent an
element in $\gl_{\infty,\infty}(R)$ by $X=(X_{ij})_{i,j\in {\mathbf
Z}}$ and $X_{ij}=(X_{ij,mn})_{m,n\in {\mathbf Z}}$. Then
\[X(u^nt^j)=\sum_{m,n\in {\mathbf Z}}X_{ij,mn}u^mt^i.\]
Observe that $\GL_{\infty,\infty}(R)$ acts on $R((u))((t))$ by the
same formula as above.

\begin{rmk} \label{inv} It is possible to give a more invariant definition of the $R$-ring $\gl_{\infty,\infty}(R)$
and the group $\GL_{\infty,\infty}(R)= \gl_{\infty,\infty}(R)^*$.

We define for any integer $n$ an $R$-submodule of $R((u))((t))$
$$\oo_n = t^n  R((u))[[t]] \mbox{.}$$
If $m < n$, then the  $R$-module
 $\oo_m / \oo_n$
is a free  $R((t))$-module, and therefore it is a topological
$R$-module  with the topology induced by open subspaces
$$ E_l = u^l t^m R[[u,t]]  / u^l t^n R[[u,t]]  \subset  \oo_m / \oo_n \mbox{.}
$$

We say that an $R$-linear map $F : R((u))((t)) \to R((u))((t))$ belongs to $\gl_{\infty,\infty}(R)$, if the following conditions hold
\begin{enumerate}
\item \label{i1}
for any integer $n$ there exists an integer $m$ such that $F \oo_{n}
\subset \oo_{m} $,
\item \label{i2}
for any integer
 $m$
there exists an integer $n$ such that $F \oo_n \subset \oo_m $,
\item \label{i3}
for any integer $n_1 < n_2 $ and $m_1 < m_2$ such that $F \oo_{n_1}
\subset \oo_{m_1}$ and $F \oo_{n_2} \subset \oo_{m_2}$ we have that
the induced $R$-linear map
$$   \bar{F} \quad : \quad \oo_{n_1} / \oo_{n_2} \lrto
\oo_{m_1} / \oo_{m_2}
$$
belongs to $\Hom_{cont} (\oo_{n_1} / \oo_{n_2} , \oo_{m_1} /
\oo_{m_2}) \mbox{.} $
\end{enumerate}

If $R=k$ is a field, then $\gl_{\infty,\infty}(R)$ and
$\GL_{\infty,\infty}(R)$ were studied also in~\cite{O3}
and~\cite{O1}.
\end{rmk}

\begin{lem}Let $R((u))((t))^*$ act on $R((u))((t))$ by the multiplication. Then this action induces an embedding $L^2\bbG_m\subset\GL_{\infty,\infty}$.
\end{lem}

We recall that there is a natural linear topology on the $R$-module $R((u))((t))$,
see formula~\eqref{top}.

\begin{lem} \label{conti} The natural action of $\GL_{\infty,\infty}(R)$ on $R((u))((t))$ is continuous.
\end{lem}
Proof is the same as \cite[Lemma~2]{O3}.

\smallskip

However, Remark~\ref{inv} indicates that $R((u))((t))$ has some more delicate structures than just a topological $R$-module, and $\GL_{\infty,\infty}(R)$ preserves this more delicate structure. To make it precise, we first recall that there is a notion of Tate $R$-modules as in \cite[\S 3]{Dr}.

\begin{dfn}(Drinfeld)
Let $R$ be a commutative ring. An elementary Tate $R$-module is a topological $R$-module which is isomorphic to $P\oplus Q^*$, where $P,Q$ are discrete projective $R$-modules, and $Q^*=\Hom(Q,R)$ with the natural topology\footnote{The basis of open neighborhoods of $0 \in Q^*$  consists of annihilators of finite subsets in $Q$.}. A Tate $R$-module is a topological $R$-module which is a topological direct summand of an elementary Tate $R$-module.
\end{dfn}

\begin{dfn} \label{Tate}
A lattice $\bbL\subset R((u))((t))$ is an $R$-submodule such
that there exists some $N\gg 0$ with the properties:  $t^NR((u))[[t]]\subset \bbL\subset
t^{-N}R((u))[[t]]$ and $t^{-N}R((u))[[t]]/\bbL$ (with the induced topology) is a Tate
$R$-module.
\end{dfn}

Note that since a Tate $R$-module is always Hausdorff, a lattice $\bbL$ is a closed $R$-submodule in $R((u))((t))$.
Examples of lattices in $R((u))((t))$ include $t^NR((u))[[t]]$.
If $\bbL$ is a lattice, then for any $N$ such that $\bbL\subset
t^{-N}R((u))[[t]]$, $t^{-N}R((u))[[t]]/\bbL$ is a Tate $R$-module (as it follows from the following lemma).

\begin{lem}\label{split}
Any open continuous surjection\footnote{We note that $M_1 \to M_2$ is an open continuous surjection if and only if the quotient module $M_2$ is endowed with the quotient topology.} $ f : M_1 \to M_2$ between Tate $R$-modules is splittable when $M_1$ has a countable basis of open neighborhoods of $0$.
\end{lem}
\begin{proof}
By definition, we can assume that $M_2$ is a topological direct summand of a topological $R$-module $N=N_1 \oplus N_2 $,
where $N_1 = \oplus_{i \in I} R e_i$ is a discrete free $R$-module, and $N_2 = \prod_{j \in J} R e_j$ is endowed with product topology of discrete $R$-modules. Let $p : N \to M_2$ be the projection. It is enough to construct a continuous $R$-module map $g: N \to M_1 $
such that $p = fg$.  Moreover, it is enough to consider two cases: to construct the map $g$ restricted to $N_1$ and restricted to $N_2$.
The first case is obvious since $N_1$ is a discrete free $R$-module.
For the second case we note that $f$ is an open map. Now let $U_1 \supset U_2 \supset  \ldots  $
be a countable  basis of open neighborhoods of $0$ from $M_1$. For any $l \in {\mathbf N}$ there is the minimal finite (or empty) subset
$K_l$ of the set $J$ such that $p(e_j) \subset f(U_l)$ for any $j \in J \setminus K_l$.
For any $j \in K_{l+1} \setminus K_l$ we define an element $g(e_j) \subset U_l$ with the property $fg(e_j)= p(e_j)$.
For any $j \in J \setminus \bigcup_l K_l$ we have that $p(e_j) $ belongs to $\bigcap_l f(U_l)$ which is a zero submodule, since
$f(U_l)$, where $l $ runs over ${\mathbf N}$, is a  basis of open neighborhoods of $0$ in $M_2$. Therefore we put $g(e_j)=0$.
We note that $R$-modules which we consider are complete with the topology. Hence we obtained a well-defined continuous $R$-module map
$g : N_2 \to M_1$ such that $fg = p$.
\end{proof}

We note that from Lemma~\ref{split} we have that if $\bbL$ is a lattice in $ R((u))((t))$, then \linebreak $\bbL/t^m R((u))[[t]]$ is a Tate $R$-module
for any integer $m$ such that $t^m R((u))[[t]]  \subset \bbL$.

\begin{prop}  \label{lat} Let $\bbL$ be a lattice in $R((u))((t))$ and $g \in \GL_{\infty,\infty}(R)$. Then
$g\bbL$ is also a lattice in $R((u))((t))$.
\end{prop}
\begin{proof}
As before, for simplicity, we denote by $\calO_n=t^nR((u))[[t]]$. Let us also denote $R((u))((t))$ by $K$.
By Remark~\ref{inv}, there is some integer $N > 0$ such that
\[\calO_N\subset g \bbL \subset \calO_{-N}.\]
We have the following
exact sequence of $R$-modules:
\begin{equation} \label{seq}
0 \lrto  \frac{g \bbL}{\calO_N}  \lrto
\frac{\calO_{-N}}{\calO_N}  \lrto \frac{\calO_{-N}}{g\bbL} \lrto 0
\end{equation}
Clearly, $\frac{\calO_{-N}}{\calO_N}$ with the induced topology is an elementary Tate $R$-module and therefore it is enough to show that there is a splitting of \eqref{seq} as topological $R$-modules. We consider the following Cartesian square
\[\begin{CD}
\frac{\calO_{-N}}{\calO_N}  @>>> \frac{\calO_{-N}}{g\bbL}\\
@VVV@VVV\\
\frac{K}{\calO_N}  @>>> \frac{K}{g\bbL}
\end{CD}\]
Therefore, it is enough to show the bottom row is splittable as topological $R$-modules. But this will follow if we can show that $K\to \frac{K}{g\bbL}$ admits a splitting $\frac{K}{g\bbL}\to K$ (then $\frac{K}{g\bbL}\to K\to \frac{K}{\calO_N}$ splits of the bottom row). Finally, to see that $K\to \frac{K}{g\bbL}$ admits a splitting, we note that, by Lemma~\ref{conti}, $g$ is a continuous automorphism of $K$ and therefore it is enough to show that $K\to \frac{K}{\bbL}$ admits a splitting.

Now we reverse the above reasoning. We choose an integer $M >  0$ such that $\calO_M\subset \bbL\subset \calO_{-M}$. Then there is an obvious splitting of $K\to \frac{K}{\calO_M}$. Therefore it is enough to find a splitting of $ \frac{K}{\calO_M}\to  \frac{K}{\bbL}$. As $\bbL$ is a lattice, we can find a splitting of $\frac{\calO_{-M}}{\calO_M}\to  \frac{\calO_{-M}}{\bbL}$ by Lemma~\ref{split}, which can be obviously extended to a splitting of $ \frac{K}{\calO_M}\to  \frac{K}{\bbL}$.
\end{proof}

\vspace{0.5cm}

Now if $\bbL,\bbL'$ are two lattices in $R((u))((t))$, we can define
as usual the $\calP ic_R^\bbZ$-torsor $\calD et(\bbL|\bbL')$   over
$\spec R$ (locally in the Nisnevich topology), and therefore we
obtain a central extension of $\GL_{\infty,\infty}$ by $\calP
ic^\bbZ$ similarly to the central extension from~\cite[\S~4C]{OsZh}.
To do this, let $m$ be minimal such that
$$t^{m}R((u))[[t]]\subset \bbL  \qquad \mbox{ and} \qquad t^{m}R((u))[[t]]\subset \bbL' \mbox{.}$$ Then $\bbL/t^{m}R((u))[[t]]$ and
$\bbL'/t^mR((u))[[t]]$ are Tate $R$-modules. For a Tate $R$-module
$M$, let $\calD et(M)$ be the $\calP ic_R^\bbZ$-torsor of
all graded-determinantal theories\footnote{Like in our previous
paper~\cite{OsZh}, we use  the notion of ``a graded-determinantal
theory" and notation $\calD et(M)$ for the $\calP ic_R^\bbZ$-torsor of all graded-determinant theories on $M$. In~\cite[\S~5.2]{Dr} this notion was under the name ``a
determinant theory" and the notation  for  the $\calP ic_R^\bbZ$-torsor of all determinant theories  on $M$ was $\calD et_M $.} on $M$. (By the theorem of
Drinfeld, see~\cite[Th.~3.4]{Dr} and \cite[\S 2.12]{BBE}, there is a Nisnevich covering  $\spec R'
\to \spec R$  such that the category  $\calD et(M) (R')$ is not
empty.) Then we define\footnote{As in~\cite{OsZh} we use the additive notation for the multiplication of two  $\calP ic_R^\bbZ$-torsors.}
\[\calD et(\bbL|\bbL')=\calD et({ \bbL' /  t^{m}R((u))[[t]]}) - \calD et({ \bbL /t^{m}R((u))[[t]]}) \mbox{.}\]

Therefore, we obtain a (categorical) central extension of
$\GL_{\infty,\infty}$ by $\calP ic^\bbZ$:
$$g \mapsto \calD et(\bbL|g\bbL) \mbox{.}$$
(We need to fix a lattice $\bbL$ to construct
such an extension. For example $\bbL = \oo_n$ for some integer $n$.)
 By restriction, we have a
central extension of
 $L^2\bbG_m$ by $\calP ic^\bbZ$.

\subsection{The generalized commutator and the change of local parameters} \label{main}
We constructed a central extension of
 $L^2\bbG_m$ by $\calP ic^\bbZ$ (after fixing some lattice $\bbL$). For such a central extension  there are  generalized commutators
$C_2$ and $C_3$, see formulas~\eqref{gencom2} and~\eqref{gencom3}. Therefore, we have a map
\begin{equation} \label{c3}
C_3 \  :  \ L^2\bbG_m\times L^2\bbG_m\times L^2\bbG_m \lrto \bbG_m
\mbox{.}
\end{equation}

Note that the map $C_3$ in formula~\eqref{c3} does not depend on
the choice of a lattice $\bbL$, which we used to construct a central
extension. Indeed, if $\bbL'$ is another lattice in $R((u))((t))$,
then any object of $\calD et(\bbL| \bbL') (R')$ (which exists after
some Nisnevich covering  $\spec R' \to \spec R$) gives an
isomorphism (over $\spec R'$) of categorical central extensions
constructed by $\bbL$ and by $\bbL'$. Therefore the corresponding
maps $C_3$ coincide, see~\cite[Cor.~2.19]{OsZh}.

\medskip

We recall the definition of the group functor $\Aut$\footnote{The central extension of this group is usually called the algebraic Virasoro group.}. This is the
group functor which associates with every ring $R$ the group of
continuous automorphisms of the $R$-algebra $R((t))$. Then
\[\Aut(R)=\{ t' = \sum a_it^i\in R((t))^*\mid a_1\in R^*, a_i \in \n R \mbox{ if } i<0\},\] and
the action of $\Aut(R)$ on $R((t))$ is given by $t'\in \Aut(R)\mapsto
\phi_{t'}: R((t))\to R((t))$, where
$$\phi_{t'}(\sum a_i t^i) = \sum a_i {t'}^i.$$ This is called  the change of a
local parameter in $R((t))$.

Now we return to the two-dimensional story. There is a question:
what is a well-defined change of local parameters in $R((u))((t))$?
In this case of a two-dimensional local field
$k((u))((t))$, where $k$ is a field, we can define the change of
local parameters as a map: $t \mapsto t'$, $u \mapsto u'$,
$$\phi_{t', u'} \; : \; \sum a_{i,j} u^jt^i \mapsto \sum a_{i,j}
{u'}^j{t'}^i \mbox{,} $$ where $u',t' \in L^2\bbG_m(R) $,
$\nu_1(t')=1$, $(\nu_1, \nu_2)(u')= (0,1)$, $t' \in  tR[[t]]$, $u'
\in \frakm(R)$. By induction, similarly to the case of
two-dimensional local fields (see also the proof of Lemma~\ref{clp}
below), one can prove that $\phi_{t', u'}$ is a well-defined
automorphism of the $R$-algebra $R((u))((t))$. Moreover, it easy to
see, that the set of these automorphisms is a subgroup in the group
of continuous automorphisms of the $R$-algebra  $R((u))((t))$, i.e,
the composition of two change of local parameters of above kind and
the inverse will be again the change of local parameters of above
kind. But this is only the analogy with two-dimensional local
fields, and we did not use that a ring $R$ may contain nilpotent
elements. Therefore we have the following lemma.
\begin{lem} \label{clp}
Let $u',t' \in L^2\bbG_m(R) $, $\nu_1(t')=1$, $(\nu_1, \nu_2)(u')=
(0,1)$, $u' \in  R((u))[[t]]$. Then a map
$$
\phi_{t', u'} \; : \; \sum a_{i,j} u^jt^i \mapsto \sum a_{i,j}
{u'}^j{t'}^i
$$
is a well-defined continuous automorphism of the $R$-algebra
$R((u))((t))$.
\end{lem}
\begin{proof}
We can change, first, the local parameter $u \mapsto u'$. To see
that for any $f \in R((u))((t))$ the series $\phi_{t,u'}(f)$ is
well-defined, i.e, it converges in the topology of $R((u))((t))$, we
use a Taylor  formula:
 \begin{equation}  \label{Tayl}
 g(\tilde{u} + \delta,t) = \sum_{k \ge 0 } (D_k(g))(\tilde{u},t)) \delta^k  \mbox{,}
 \end{equation}
 where $D_0 \equiv \id $, and for $k \ge 1$ we have  on monomials
 $$D_k (u^j t^i) = \frac{1}{k!}\frac{d^k(u^j)}{d u ^k }t^i= {j \choose k} u^{j-k}t^i
 \mbox{,} \qquad
  {j \choose k}= \frac{j \cdot (j-1) \cdot \ldots \cdot (j-k+1)}{1\cdot 2 \cdot \ldots \cdot k} \mbox{.}$$
   The map $D_k$ is extended to series $g$ from $R((u))((t))$ as an $R$-linear continuous map. We suppose that the series $D_k(g))(\tilde{u},t)$, $k \ge 0$ are already well-defined.  We apply formula~\eqref{Tayl} successively in the following cases: 1) if  $g=D_k(f) $, $\tilde{u}= cu$, $c \in R^*$, $\delta \in u \cdot \frakm (R)= u^2R[[u]] + t R((u))[[t]]$, then the right hand side of~\eqref{Tayl} is a converge series in topology of $R((u))((t))$, and 2) if $g=f$, $\tilde{u} = cu + a $, $c \in R^*$,
  $a \in  u \cdot \frakm (R)$, $\delta = (u' - \tilde{u}) \in \n R((u))((t))$, then the right hand side of~\eqref{Tayl} is a finite sum. It gives us that the series $\phi_{t,u'}(f)$ is well-defined.

  To prove that $\phi_{t,u'}$ is an automorphism of the algebra $R((u))((t))$, we find a series $v \in R((u))[[t]]^*$
  such that the substitution of
  $v$ in $u'$ instead of $u$ is equal $u$, i.e.
  $u'(v,t) =u$. Using the corresponding one-dimensional result for $R((u))$, we can find $v_1 \in R((u))^*$ such that
  $u'(v_1,t) \equiv u \mod t R((u))[[t]]$. Now it is easy to find an element $v_2 = u + b_1(u)t$, $b_1(u) \in R((u))$ such that
  $u'(v_1(v_2,t),t)  \equiv u \mod t^2 R((u))[[t]]$. After some steps we find $v_k = u +b_{k-1}(u)t^k$, $b_{k-1}(u) \in R((u))$ such that
  $u'(v_1(v_2(\ldots v_k)),t) \equiv u \mod t^k R((u))[[t]]$. Then the sequence $w_k=v_1(v_2(\ldots v_k))$ tends to $v$ when $k \to +\infty$.

  Now $\phi_{t',u'}=  \phi_{t,u'} \, \phi_{t'(v,t),u}  $, where $u'(v,t) =v(u',t)=u$. But $\phi_{t'(v,t),u}$ is a well-defined automorphism of the algebra $R((u))((t))$ by the one-dimensional result applied to the ring $B((t))$, where $B= R((u))$. From Remark~\ref{inv} it is easy to see that $\phi_{t'(v,t), u}  $ and
  $\phi_{t,u'}$ are from
$\GL_{\infty,\infty}(R)$. Therefore, by Lemma~\ref{conti} they are
continuous automorphisms of $R((u))((t))$.
  \end{proof}
\begin{rmk} \label{vir}
The composition of two automorphisms of $R((u))((t))$ as in
Lemma~\ref{clp} is not always the automorphism of the same type (as
in Lemma~\ref{clp}), since we demanded $u' \in  R((u))[[t]]$. It
would be natural to withdraw this condition. Then as in the proof of
Lemma~\ref{clp}, we can obtain that the operator $\phi_{t',u'}$ is a
well-defined continuous homomorphism from the $R$-algebra
$R((u))((t))$ to itself. But there is the problem to prove that
$\phi_{t,u'}$ is an automorphism, because there can be infinite many
nilpotent elements in the series for $u'$ (compare with the last
paragraph of Section~\ref{tdt}). If $R \in \B$ (see
Definition~\ref{catB}), then by the similar method as in the end
of~\cite[\S~1]{Mo} it is possible to prove that there is $v \in
L^2\bbG_m(R)$ such that $u'(v,t) =u$. Therefore, in this case,
$\phi_{t',u'}$ is an automorphism of $R((u))((t))$ for any
 $u',t' \in L^2\bbG_m(R) $ with $\nu_1(t')=1$, $(\nu_1, \nu_2)(u')= (0,1)$.
  \end{rmk}

\begin{prop} \label{invar}
The map $C_3$ is invariant under the change of the local parameters
$t$ and $u$. Namely, in conditions of Lemma~\ref{clp}, for any $f ,
g, h \in L^2\bbG_m(R)$ we have that
$$C_3(f,g,h) =
C_3(\phi_{t',u'}(f), \, \phi_{t',u'}(g), \, \phi_{t',u'}(h))
\mbox{.}$$
\end{prop}
\begin{proof}
From Remark~\ref{inv} it is easy to see that $\phi_{t', u'} \in
\GL_{\infty,\infty}(R)$. Therefore, by Proposition~\ref{lat}, if
$\bbL$ is a lattice in $R((u))((t))$, then an $R$-submodule
$\phi_{t', u'} (\bbL)$ is again a lattice in $R((u))((t))$.  Now we
apply the fact that the maps $C_3$ coincide even if they are
constructed by various lattices in $R((u))((t))$.
\end{proof}

\medskip

We will need some lemmas that the property to be invariance under
the change of local parameters can uniquely define some tri-linear
and anti-symmetric map over a field. Let $k$ be a field of
characteristic zero in the following two lemmas. We recall the
following well-known lemma.
\begin{lem} \label{calcu1}
Let $K=k((t))$. Let $\langle\cdot, \cdot\rangle$ be a  pairing
$K\times K \to k$ such that it is continuous\footnote{Here and in Lemma~\ref{calcu2} a continuous map means a map which is continuous in each argument, i.e. when we fix other arguments of the map.}, bilinear, anti-symmetric
and invariant under the change of local parameter in $K$: $t \mapsto
t'$. Then $\langle f,g \rangle= c \cdot \res(f dg)$, where $c \in k$
is fixed, and $f,g$ are any elements from $K$.
\end{lem}
\begin{proof}
It is enough to prove that $\langle t^{-n}, t^n \rangle =cn$,
$\langle t^n, t^m \rangle = 0$ if $n+m \ne 0$. If $n+m \ne 0$, then
we take $t'=2t$. Now from $\langle t^n, t^m \rangle = \langle{t'}^n,
{t'}^m \rangle = 2^{n+m}\langle t^n, t^m \rangle $ it follows that
$\langle t^n, t^m \rangle =0$. Hence, $\langle t^{-n}, t \rangle =0$
when $n \ge 2$. Let $t' = t + t^n$ be a new local parameter. Then
${t'}^{-n}= t^{-n}- nt^{-1} + \ldots$. Now we have
$$
0 = \langle t^{-n}, t \rangle = \langle{t'}^{-n}, t' \rangle =
\langle t^{-n} - nt^{-1} + \ldots, \, t + t^n \rangle = \langle
t^{-n}, t^n \rangle -n \langle t^{-1}, t \rangle  \mbox{.}
$$
(We can consider only the finite sum of $\langle \cdot, \cdot
\rangle $ in the above expression due to the continuous property of
$\langle \cdot, \cdot \rangle $.) Thus, $\langle t^{-n}, t^{n}
\rangle = n \langle t^{-1}, t \rangle = nc$.
\end{proof}

Now we consider the case of a two-dimensional local field.
\begin{lem} \label{calcu2}
Let $K = k((u))((t))$. Let $\langle \cdot, \cdot , \cdot \rangle $
be a map $K \times K \times K \to k$ such that it is continuous,
tri-linear, anti-symmetric and
 invariant under the change of local parameters in $K$: $t \mapsto t'$, $u \mapsto u'$. Then $\langle f,g,h \rangle = c \cdot \Res (f dg \wedge dh)$, where
$c \in k$ is fixed, and $f,g,h$ are any elements from $k$.
\end{lem}
\begin{proof}
We will consider several cases and reduce them to the case
\begin{equation} \label{f1}
\langle u,t,u^{-1}t^{-1} \rangle =c \mbox{.}
\end{equation}

We note that
\begin{equation}  \label{twist}
\langle u^jt^i, u^lt^k, u^nt^m \rangle =0 \qquad \mbox{if} \qquad
j+l+n \ne 0 \qquad  \mbox{or} \qquad i+k+m \ne 0 \mbox{.}
\end{equation}
(The proof is by twisting $t'=2 t$, $u'= 3u$ as in the
$1$-dimensional case.)

We have to prove that $X = \langle u^jt^i, u^lt^k, u^{-j-l}t^{-i-k}
\rangle = c\cdot (jk-il)$.

We will calculate the following easy case
\begin{equation} \label{trc}
\langle u^m, u^n, u^l \rangle =0 \mbox{.}
\end{equation}
 We consider a new local parameter $t'=t+t^2$. Then ${t'}^{-1} =
t^{-1} -1 + t^2 + \ldots$. We have
\begin{eqnarray*}
0 = \langle t^{-1}u^m, u^n, u^l \rangle = \langle{t'}^{-1}u^m, u^n,
u^l \rangle=
\langle t^{-1}u^m -u^m + tu^m + \ldots,  u^n, u^l\rangle=\\
=\langle t^{-1}u^m, u^n, u^l\rangle - \langle u^m,u^n, u^l\rangle +
\langle tu^m, u^n, u^l\rangle + \ldots \mbox{.}
\end{eqnarray*}
Hence, using the continuous property of $ \langle \cdot, \cdot,
\cdot\rangle$, we have
 $ \langle u^m, u^n, u^l\rangle=0$, since by~\eqref{twist}, $ \langle t^qu^m, u^n,u^l\rangle=0$ for any $q \ne 0 $.

Now we will explain, why
\begin{equation}  \label{f2}
 \langle u^jt^i, t, u^{-j}t^{-i-1}\rangle = j \cdot c
\end{equation}
 when $(i,j) > (0,1)$ (with respect to the lexicographical order
in ${{\mathbf Z}}^2$, see~\eqref{canon}). Then $u'= u+ u^jt^i$ is a
well-defined change of the local parameter $u$. We have
$$0 =  \langle u,t,u^{-j}t^{-i-1}\rangle=  \langle u',t, {u'}^{-j}t^{-i-1}\rangle \mbox{.}$$
We compute ${u'}^{-j}= (u + u^jt^i)^{-j} = u^{-j} -ju^{-1}t^i + d
u^{j-2}t^{2i} + \ldots$, where $d \in {\mathbf Z}$. Therefore, we
have
\begin{eqnarray*}
0=  \langle u + u^jt^i, t, u^{-j}t^{-i-1} - j u^{-1}t^{-1} + du^{j-2}t^{i-1} + \ldots\rangle= \\
=-j  \langle u,t,u^{-1}t^{-1}\rangle +  \langle u^jt^i, t,
u^{-j}t^{-i-1}\rangle = -j\cdot c +  \langle u^jt^i, t,
u^{-j}t^{-i-1}\rangle \mbox{.}
\end{eqnarray*}
Thus, we explained this case.

Analogously, we can obtain that
\begin{equation}  \label{f3}
 \langle u, u^lt^k, u^{-l-1}t^{-k}\rangle =k \cdot c
\end{equation}
 when $k \ge 0$. The case $k=0$ we have proved above.
If $k=1$, then we consider $t'= u^{-l}t $. We obtain
$$
 \langle u, u^lt, u^{-l-1}t^{-1}\rangle =  \langle u, u^l{t'}, u^{-l-1}{t'}^{-1}\rangle=  \langle u, t,
u^{-1}t^{-1}\rangle=c \mbox{.}
$$
Now if $k > 1$, then we consider a well-defined change of the local
parameter $t \mapsto t'= t + u^lt^k$. We have ${t'}^{-k}= t^{-k} - k
u^lt^{-1} + e u^{2l}t^{k-2}$ + \ldots,  where $e \in {\mathbf Z}$.
We insert this expression into the following equality: $ 0=
 \langle u,t,u^{-l-1}t^{-k}\rangle=  \langle u,t', u^{-l-1} {t'}^{-k}\rangle \mbox{.} $ Thus,
using~\eqref{twist}, we obtained this case.

Now we consider the case when $(k,l) \ge (0,1)$ and $(i,j) \ge
(0,1)$. Without loss of generality (using the anti-symmetric
property of $ \langle \cdot, \cdot, \cdot\rangle$) we can assume
that $(k,l) \ge (i,j)$. If $(i,j) = (0,1)$, then this is the
previous case, see formula~\eqref{f3}.
 Therefore we assume that $(i,j) > (0,1)$. If $k=0$,
this was also calculated above. If $k=1$,  then let $t'= u^{-l}t $
be a new local parameter. We have
$$
 \langle u^j t^i, u^lt, u^{-j-l} t^{-i-1}\rangle=  \langle u^j {t'}^i, u^lt', u^{-j-l}
{t'}^{-i-1}\rangle=  \langle u^{j-li}t^i, t, u^{li-j}t^{-i-1}\rangle
\mbox{.}
$$
Since $(i,j-li) >(0,1)$, we have that the last expression was  also
calculated above. Therefore we can assume that $k>1$. There is a
well-defined change of the local parameter $t \mapsto t'$, where  $t'= t +
u^lt^k$.
We have
\begin{eqnarray}
0=  \langle u^j t^i,t,u^{-j-l}t^{-i-k}\rangle=  \langle u^j{t'}^i, t',  {u}^{-j-l} {t'}^{-i-k}\rangle =   \nonumber\\
\label{ntr} =  \langle  u^j (t+u^lt^k)^i, t+u^lt^k , u^{-j-l}( t+ u^lt^k)^{-i-k}\rangle  \mbox{.}
\end{eqnarray}
We calculate now
$(t+u^lt^k)^i = t^i +i u^l t^{i+k-1} + a u^{2l}t^{i+2k-2} + \ldots$, where $a \in
{\mathbf Z}$, and calculate $( t+ u^lt^k)^{-i-k}= t^{-i-k} + (-i-k)u^lt^{-i-1}+b u^{2l} t^{k-i-2} + \ldots$, where $b \in {\mathbf Z}$.
We substitute these expressions into formula~\eqref{ntr}. We obtain
\begin{eqnarray}
0 = \langle  u^jt^i +i u^{l+j} t^{i+k-1} + \ldots, t + u^lt^k, u^{-j-l}t^{-i-k} + (-i-k)u^{-j}t^{-i-1}
+ \ldots   \rangle = \nonumber \\
=  \langle u^jt^i, u^lt^k, u^{-j-l}t^{-i-k}  \rangle  +(-i-k) \langle  u^jt^i, t, u^{-j}t^{-i-1}  \rangle
+i \langle u^{l+j} t^{i+k-1}, t, u^{-j-l}t^{-i-k}   \rangle
 \mbox{.}   \nonumber
\end{eqnarray}
Hence and using formula~\eqref{f2} we obtain
$$
0 = X + (-i-k) \cdot j \cdot c + i \cdot (l+j) \cdot c \mbox{.}
$$
Therefore  we calculated $X = (jk-il)\cdot c$. We finished the proof of case $(k,l)
\ge (0,1)$, $(i,j) \ge (0,1)$.

We calculate now
$$
Z =  \langle u^j, u^lt^k, u^{-j-l}t^{-k}\rangle  \mbox{,}
$$
where $k <  0$. Then $u' = u+ u^{-j-l}t^{-k}$ is a well-defined
change of the local parameter $u$, since $-k >0$. We have $0 =
\langle u^j, u^lt^k, u\rangle = \langle {u'}^j, {u'}^lt^k, u'
\rangle$. We compute ${u'}^j= u^j + j u^{-l-1}t^{-k}+ \ldots$ and
${u'}^l= u^l + l u^{-j-1}t^{-k} + \ldots$. Hence,
\begin{eqnarray*}
0 =   \langle u^j + j u^{-l-1}t^{-k} + \ldots, u^lt^k + l u^{-j-1} + \ldots, u + u^{-j-l}t^{-k}\rangle = \\
= Z + l  \langle u^j, u^{-j-1}, u\rangle + j \langle u^{-l-1}t^{-k},
u^lt^k, u\rangle = Z + j \cdot (-k) \cdot c \mbox{,}
\end{eqnarray*}
where we used formulas~\eqref{trc}, \eqref{f3} and the
anti-symmetric property of $  \langle \cdot, \cdot, \cdot\rangle$.
Hence $Z =j  k c $.

We calculate now $Y=  \langle u^jt^i, u^lt^k,
u^{-j-l}t^{-i-k}\rangle$ when $(i,j)< (0,0)$ and $(k,l)<(0,0)$. Then
$(-i-k, -j-l)>(0,0)$. If $i=0$ or $k=0$, then this is the previous
case, where we calculated $Z$. (If $i=k=0$, then it follows from
formula~\eqref{trc}.) Therefore, we assume that $-k \ge 1$ and $-i
\ge 1$. We consider a well-defined change of local parameters: $t'=
t+ t^{-i-k}$. We have $0 =   \langle u^jt^i, u^lt^k,
u^{-j-l}t\rangle=   \langle u^j {t'}^i, u^l{t'}^k,
u^{-j-l}t'\rangle$. We compute ${t'}^i= t^i + it^{-k-1} + \ldots$
and ${t'}^k= t^k + kt^{-i-1} + \ldots$. Therefore, we have
\begin{eqnarray*}
0=   \langle u^jt^i + iu^jt^{-k-1} + \ldots, u^lt^k + ku^lt^{-i-1}+ \ldots, u^{-j-l}t + u^{-j-l}t^{-i-k}\rangle = \\
= Y + k   \langle u^jt^i, u^lt^{-i-1}, u^{-j-l}t\rangle + i \langle
u^jt^{-k-1}, u^lt^k, u^{-j-l}t\rangle \mbox{.}
\end{eqnarray*}
Since $-k -1\ge 0$ and $-i-1 \ge 0$, we can use the previous cases
and the anti-symmetric property of $  \langle \cdot, \cdot,
\cdot\rangle$. We obtain $Y = -k(j(i-1)-il)c -i(jk-l(-k-1))c=
(kj-li)c$.

Now, using the anti-symmetric property of $  \langle \cdot, \cdot,
\cdot \rangle$, it is easy to see that we have considered all
possible cases.
\end{proof}

\subsection{The generalized commutator and the symbol}  \label{gencat}
We recall that in Sections~\ref{tdt} and~\ref{explicform}  we have explicitly  defined  the
two-dimensional Contou-Carr\`{e}re symbol $L^2\bbG_m(R)  \times
L^2\bbG_m(R) \times L^2\bbG_m(R)  \lrto \bbG_m(R)$ for a ring $R$
  when $\mathbf{Q} \subset R$ or $R \in \B$. Lemma~\ref{unique}, the following important theorem and
  the corollary from this theorem
  show that we can define the two-dimensional Contou-Carr\`{e}re symbol
  for any ring $R$ by means of the map $C_3$. Besides, this theorem greatly generalizes Theorem~4.11 from~\cite{OsZh}.
  We recall that the map $C_3$ is tri-multiplicative and
  anti-symmetric.

\begin{thm} \label{th-main}
Let $R$ be any $\mathbf{Q}$-algebra. Then the map
 $$ C_3 \  :  \ L^2\bbG_m(R)  \times L^2\bbG_m(R)  \times L^2\bbG_m(R) \lrto \bbG_m(R)
 $$
 satisfies properties~\eqref{2cexp}-\eqref{sign}.
\end{thm}
\begin{proof}
During the proof we assume that all the schemes (and ind-schemes) are defined over the field $\mathbf{Q}$. For simplicity of notations, we will omit indication on the field $\mathbf{Q}$.

We will use the following fact. Let $V$ be an (elementary, for simplicity) Tate $R$-module,
and let $\GL(V)$ be the group of automorphisms of $V$ as a Tate $R$-module (i.e. $\GL(V)(R')$ is the group of continuous automorphisms of the Tate $R'$-module $V \hat{\otimes}_R R'$), which is a sheaf
of groups over $\mathbf{Aff}/R$. Then we have a canonical
homomorphism (explicitly defined for an elementary Tate $R'$-module $V \hat{\otimes}_R R'$ by the choice of a coprojective
lattice\footnote{We recall that by definition from~\cite{Dr}, a lattice $L$ in a Tate $R$-module $V$ is an open $R$-submodule
of $V$ such that $L/U$ is finitely generated for any open $R$-module $U \subset L$. A lattice $L$ is called coprojective if $V/L$ is a projective $R$-module.}, after a Nisnevich
covering $\spec R' \to \spec R$), see also~\cite[\S~5.6]{Dr}:
\[\calD et_V: \GL(V)\to\calP ic^\bbZ_R \mbox{.}\]
Let $Z_2\subset\GL(V)$ be the subsheaf of commuting elements, then
we obtain the usual commutator map for $\calD et_V$  (see
Section~$2$D of~\cite{OsZh}):
\[\on{Comm}: Z_2\to \bbG_m.\]
Now let $f,g,h\in\GL_{\infty,\infty}(R)$ be commuting with each
other elements. Assume that $\bbL$ is a lattice in $R((u))((t))$
fixed by $f$ and $g$, and assume that $h\bbL\subset\bbL$. Then $f,g$
induces $\pi_h(f),\pi_h(g)\in \GL(V)$ where $V=\bbL/h\bbL$ and
\begin{equation} \label{comm}
C_3(f,g,h)=\on{Comm}(\pi_h(f),\pi_h(g))^{-1} \mbox{.}
\end{equation}
 Formula~\eqref{comm} can be proved in exactly the same way as in \cite{OsZh} (see the proof of
Theorem~4.11 in \emph{loc. cit.}, in particular Lemma~4.12 and the commutative
diagram~(4-8)).

\medskip

We recall (see formula~\eqref{schdec}) that for the group ind-scheme $L^2\bbG_m$ we have decomposition
\[L^2\bbG_m=\bbP\times\bbZ^2\times \bbG_m \times \bbM.\]

We will use also the following fact.
 Let $\bbS$ and $\bbT$ be any connected group (ind)-subschemes in the ind-scheme $L^2\bbG_m$.
 We show that
 \begin{equation}  \label{groupsch}
 C_3(f,g,h)=1 \mbox{,}
 \end{equation}
 where $f \in \bbS(R)$, $g \in \bbT(R)$, $h \in \bbG_m(R)$.
 Indeed, by varying
 $f\in\bbS(R)$ and  $g\in\bbT(R)$ for any ring $R$, we may regard
$C_3$ as a morphism of  (ind)-schemes
\[\bbS\times \bbT\to \underline{\Hom}_{\mbox{gr}}(\bbG_m,\bbG_m) \simeq\bbZ \mbox{.}\]
The (ind)-scheme $\bbS \times\bbT$ is connected. Therefore the claim is clear, since $C_3(0,g,h)=1$ for any $g$ and $h$ from $L^2\bbG_m(R)$.
We note that by the same reason formula~\eqref{groupsch} remains true if we put in this formula $g= u$ or $g =t$ and $f, h$ as above.

The schemes $\bbP$,  $\bbG_m$ and  $\bbM$ are connected. Indeed, it is clear that $\bbG_m$ and $\bbM$ are connected schemes. Besides, the topological space (over $\spec \mathbf{Q} $) of the scheme $\bbP$ is equal to one point, because if
$f \in L_m \n_n(R)$ (see formula~\eqref{nilind}) then $a_i^{k_{i,n}}=0$, where $i \ge m$  and $k_{i,n}$ depends on $i$ and $n$.)

\medskip

We will use also that if $f$, $g$ and $h$ are elements from $L^2\bbG_m(R)$ such that they preserve a lattice in $R((u))((t))$,
then $C_3(f,g,h)=1$. Indeed, we have directly from the construction that  the categorical central extension restricted to the subgroup generated by elements $f$, $g$ and $h$
is trivial.

\medskip

The proof of the theorem is based on several case by case
inspections. We will use the tri-multiplicative and anti-symmetric property
of the map $C_3$.

\medskip

We will check  property~\eqref{cc4}  for the map $C_3$. Since $u$ and $t$ belong to $\mathbf{Q}((u))((t))^*$, then this property directly  follows
from~\cite[Th.~4.11]{OsZh}.

\medskip

We will check now  property~\eqref{cc3} for the map $C_3$. We will consider several cases.

Let $a \in \bbG_m(R)$, $g$ and $h$ be any elements from $(\bbP \times \bbM \times \bbG_m )(R)$ for any ring $R$. Then
$$C_3(a,g,h)=C_3(g,h,a)=1$$ by formula~\eqref{groupsch}. It satisfies formula~\eqref{cc3}.

Let $a \in \bbG_m(R)$, $g$ be any element from $(\bbP \times \bbM \times \bbG_m )(R)$ for any ring $R$.
Then
$$
C_3(a,g,t)=C_3(g,t,a)=1
$$
by an analog of formula~\eqref{groupsch}. It satisfies formula~\eqref{cc3}.

The case $C_3(a,g,u)=1$ when   $a \in \bbG_m(R)$ and  $g$ is any element from $(\bbP \times \bbM \times \bbG_m )(R)$ for any ring $R$
is analogous to the previous case. It satisfies formula~\eqref{cc3}.

The case $C_3(a,t,t)=C_3(a,u,u)=1$ when  $a \in \bbG_m(R)$ follow from the fact $\Hom(\bbG_m,\bbG_m)\simeq\bbZ $. Therefore we check these equalities on $a \in \mathbf{Q}^* \subset \bbG_m(R)$ and this checking follows from~\cite[Th.~4.11]{OsZh}. It satisfies formula~\eqref{cc3}.

The last case $C_3(a,u,t)=a$ when   $a \in \bbG_m(R)$ follows from formula~\eqref{comm}
for $V=R((u))$ and $h=t$. It satisfies also formula~\eqref{cc3}. We have finally  checked   property~\eqref{cc3} for the map $C_3$.

\medskip

We will check now  property~\eqref{2cexp} for the map $C_3$. We will consider again several cases.

Let $f \in (\bbP \times \bbM) (R) $, $g \in \bbG_m(R)$ and $h $ be any from $L^2\bbG_m(R)$, where $R$ is any ring.
Using the anti-symmetric property of the map $C_3$ we reduce this case to the cases above when we checked the property~\eqref{cc3}.

Now let $f$, $g$ and $h$ be any elements from  $(\bbP \times \bbM) (R) $. Using the tri-multiplicativity of the map $C_3$, we will assume that any of these elements belongs either to $\bbP(R)$ or to $\bbM(R)$.
We recall  that over $\mathbf{Q}$, there are isomorphisms of group ind-schemes (see formulas~\eqref{explog1} and~\eqref{explog}):
\[\exp:\frakm\simeq \bbM, \quad \exp: \frakp\simeq \bbP \mbox{.}\]
Our goal is to prove that
\begin{equation} \label{goal}
C_3(f,g,h)= (f,g,h) \mbox{,}
\end{equation}
 where $(\cdot, \cdot, \cdot)$ is an expression given by formula~\eqref{2cexp}.
By Proposition~\ref{weak version}, it is enough to check formula~\eqref{goal}
on the following elements:
    \begin{equation} \label{variant}
    f=\exp(au^jt^i) \mbox{,} \quad g=\exp(bu^lt^k) \mbox{,} \quad h=\exp(cu^nt^m) \mbox{,}
    \end{equation}
    where $a$, $b$ and $c$ are from $R$.

We note that if $(i,j) > (0,0)$, then $f \in \bbM(R)$. If $(i,j) < (0,0)$, then $f \in \bbP(R)$ and $a \in \n R$. For the sequel it will be convenient for us
to include also variant when $(i,j)= (0,0)$ and $a \in \n R$ (although in this case $f$ is not from $(\bbP \times \bbM)(R)$).
The same is true for $g$ with $(k,l)$ and $h$ with $(m,n)$ (including the cases when $(k,l)=0$ and $(m,n)=0$).

Let us consider, for example,  the variant when indices $(i,j) \le 0$, $(k,l)>0$ and $(m,n) >0$.
We consider a map $\widehat{\bbG}_a\to\bbP \times (1+ \n)$ which is given as  $a\mapsto \exp(au^jt^i)$,
where $\widehat{\bbG}_a=\on{Spf} \mathbf{Q}[[T]]= \n$  is the additive
formal group, the group ind-schemes $1+ \n \stackrel{\exp}{\longleftarrow} \widehat{\bbG}_a$  are isomorphic (the group ind-scheme $1+\n$ corresponds to the index
$(i,j)=(0,0)$). We consider also two maps $\bbG_a\to\bbM$ which are given as $ b\mapsto
\exp(bu^lt^k)$ and $c\mapsto \exp(cu^nt^m)$. By composing these maps
with $C_3$ and $(\cdot, \cdot, \cdot)$, we obtain two
tri-multiplicative morphisms
$$
 \widehat{\bbG}_a\times\bbG_a\times\bbG_a \longrightarrow \bbG_m  \mbox{.}
$$
Analogously, when we consider other variants for indices $(i,j)$, $(k,l)$ and $(m,n)$ from formula~\eqref{variant} we will also obtain
two tri-multiplicative morphisms of the following type:
\begin{equation} \label{trim}
\bbH_1 \times \bbH_2  \times \bbH_3  \longrightarrow \bbG_m  \mbox{,}
\end{equation}
where   the group ind-scheme $\bbH_i$ $(1 \le i \le 3)$ is equal either to $\bbG_a$ or to $\widehat{\bbG}_a$.

\begin{lem} \label{hom}
Let $\phi$ be a tri-multiplicative morphism from $\bbH_1 \times \bbH_2  \times \bbH_3$ to
$\bbG_m$, where the group ind-scheme $\bbH_i$ $(1 \le i \le 3)$ is equal either to $\bbG_a$ or to $\widehat{\bbG}_a$ (over $\mathbf{Q}$).
If $\bbH_i = \bbG_a$ for all $i$, then $\phi=1$. In other cases the set of all such tri-multiplicative morphisms  is isomorphic to the set
$\bbG_a(\mathbf{Q})$ such that if $d \in \mathbf{Q}$ then the morphism $\phi_d$ is given by a formula:
\begin{equation} \label{diff}
\phi_d(a,b,c)= \exp(dabc)  \mbox{,}
\end{equation}
where $a \in \bbH_1(R)$, $b \in \bbH_2(R)$ and $c \in \bbH_3(R)$  (for any ring $R$).
\end{lem}
\begin{proof}
We consider, first, the case when all $\bbH_i = \widehat{\bbG}_a$.
Let $\on{\bf Tri}$ be  the set of all tri-multiplicative morphisms:
$\widehat{\bbG}_a\times \widehat{\bbG}_a\times \widehat{\bbG}_a \longrightarrow \bbG_m$.
We note that $\widehat{\bbG}_a\times \widehat{\bbG}_a\times \widehat{\bbG} = \on{Spf} \mathbf{Q}[[x_1, x_2, x_3]]$
and $\bbG_m = \spec \mathbf{Q}[T,T^{-1}]$. Then any morphism $\phi$ from the ind-scheme  $\on{Spf} \mathbf{Q}[[x_1, x_2, x_3]]$
to the scheme $\bbG_m$ is given by an invertible series
$$p=\phi^*(T) \in \mathbf{Q}[[x_1, x_2, x_3]] \mbox{.}$$
We note that the group law $\widehat{\bbG}_a \times \widehat{\bbG}_a \to \widehat{\bbG}_a $ is given by the map of rings
$\mathbf{Q}[[y]] \to \mathbf{Q}[[x, x']]$, where $y \mapsto x+ x'$. The group law $\bbG_m \times \bbG_m \to \bbG_m$
is given  by the maps of rings $\mathbf{Q}[S,S^{-1}] \to \mathbf{Q}[T,T^{-1}, T', T'^{-1}]$, where $S \mapsto TT'$.
Therefore, using the tri-multiplicativity condition, we obtain that
 the set $\on{\bf Tri}$ is described in the
following way:
\[ \small
\on{\bf Tri}=\left\{p(x_1,x_2, x_3)\in
\mathbf{Q}[[x_1,x_2,x_3]]  \ \left|\ \begin{array}{c}p\in \mathbf{Q}[[x_1,x_2,x_3]]^*\\
p(x_1+x'_1,x_2,x_3)=p(x_1,x_2,x_3)p(x'_1,x_2,x_3)\\
p(x_1,x_2+x'_2,x_3)=p(x_1,x_2,x_3)p(x_1,x'_2,x_3) \\
p(x_1,x_2, x_3 +x'_3) = p(x_1,x_2,x_3) p(x_1,x_2,x'_3)
.\end{array}\right.\right\}\]
We note that from this description  we obtain that the series $p(x_1,x_2,x_3)$ has constant coefficient equal to $1$.
By applying the $\log$-map to the
conditions describing $p(x_1,x_2,x_3)$, one can see\footnote{Indeed, it is enough to prove that if $F(x) \in S[[x]]$ and $F(x+y)=F(x)+F(y)$ for a $\mathbf{Q}$-algebra $S$,
 then $F(x)= c x$ for some $c \in S$. We consider the Taylor formula $F(x+y) = F(x) + F'(x)y + (1/2)F''(x)y^2 + \ldots$. Hence we have that
 $F(y)= F'(x)y + (1/2)F''(x)y^2 + \ldots$. Therefore $F'(x)= c$ for some $c \in S$. Now, to apply induction we consider the ring $S[[x_3]]$,
 where $S=\mathbf{Q}[[x_1,x_2]]$.} by
induction on number of variables that
$$ p(x_1,x_2,x_3)=\exp(dx_1x_2x_3)
$$
 for some
 $d\in {\mathbf Q}$.
 The last formula implies formula~\eqref{diff}.

 The other variants for the group ind-schemes $\bbH_i$ can be done analogously using that
 $\bbG_a \times \widehat{\bbG}_a \times \widehat{\bbG}_a=\on{Spf} \mathbf{Q}[x_1] [[x_2, x_3]]
 $ and $\bbG_a \times \bbG_a \times \widehat{\bbG}_a=\on{Spf} \mathbf{Q}[x_1, x_2][[ x_3]]$. For the case when all $\bbH_i = \bbG_a$
 we note that $\bbG_a \times \bbG_a \times \bbG_a=\spec \mathbf{Q}[x_1, x_2,x_3]$. Therefore the expression $\exp(dx_1x_2x_3)$ belongs   to the ring $\mathbf{Q}[x_1, x_2,x_3]$ if and only if $d=0$.
 \end{proof}

From this lemma we have the following formula for any tri-multiplicative morphism $\phi : \bbU \times \bbU \times \bbU \to \bbG_m$, where $ \bbU =\bbP \times \bbM \times (1+ \n) $:
\begin{equation} \label{coinc}
\phi(\exp(au^jt^i), \exp(bu^lt^k),\exp(cu^nt^m))= \exp (d_{(i,j),(k,l),(m,n)} abc)  \mbox{,}
\end{equation}
where the group ind-schemes $(1+ \n)  \stackrel{\exp}{\longleftarrow} \widehat{\bbG}_a$ are isomorphic, indices  $(i,j)$, $(k,l)$, $(m,n)$ are from $\mathbf{Z} \times \mathbf{Z}$, $a$, $b$ and $c$ are any elements from any ring $R$ (and  $a$ or $b$ or $c$ is a nilpotent element if the corresponding index $(i,j)$ or $(k,l)$
or $(m,n)$  is less or equal to $(0,0)$). Besides, the elements $d_{(i,j),(k,l),(m,n)} \in \mathbf{Q}$ depend only on $\phi $ and indices $(i,j)$, $(k,l)$, $(m,n)$.

Therefore, to prove that two tri-multiplicative morphisms $C_3$ and $(\cdot, \cdot, \cdot)$ composed with $\exp$-maps coincide as morphisms $\bbH_1 \times\bbH_2\times\bbH_3\to\bbG_m$, it is enough to prove that the corresponding elements $d_{(i,j),(k,l),(m,n)} \in {\mathbf Q}$ in formula~\eqref{coinc} coincide for both morphisms. This can be done by the choice of a particular ring $R$ and particular elements $a$, $b$ and $c$
in this ring. We have a lemma.
\begin{lem}\label{sim} Let
$R=\mathbf{Q}[\varepsilon_1,\varepsilon_2,\varepsilon_3]/(\varepsilon_1^2,\varepsilon_2^2,\varepsilon_3^2)$.
Let $f=\exp(\varepsilon_1 u^jt^i)\in L^2\bbG_m(R), g=\exp(\varepsilon_2
u^lt^k)\in   L^2\bbG_m(R)$ and $h=\exp (\varepsilon_3 u^nt^m) \in  L^2\bbG_m(R)$,
then
\begin{equation} \label{eqc3cc}
C_3(f,g,h)=(f,g,h) =
\exp((lm-nk)\delta_{i+k+m,0}\delta_{j+l+n,0}\varepsilon_1
\varepsilon_2 \varepsilon_3)  \mbox{.}
\end{equation}
\end{lem}
\begin{proof}
The second equality in~\eqref{eqc3cc} follows from
formula~\eqref{2cexp} by easy calculations. We note that we have a well
defined map $\langle\cdot, \cdot, \cdot\rangle: K \times K \times K
\to {\mathbf Q}$, where $K ={\mathbf Q}((u))((t))$, which is
defined by  the following
equality (this equality is a consequence of  formula~\eqref{coinc} and the continuous property of the map $C_3$ which follows from Proposition~\ref{lemma-main}):
 \begin{equation} \label{trilin}
 1+\langle p,q,r\rangle\varepsilon_1  \varepsilon_2 \varepsilon_3 =
C_3(1+  p\varepsilon_1, 1+q \varepsilon_2, 1+ r
\varepsilon_3) \mbox{,}
\end{equation}
 where $p,q,r$ are any elements from $K$\footnote{We note that $K$ is the Lie algebra of the group ind-scheme $ \bbU =\bbP \times \bbM \times (1+\n) $. Then by any tri-multiplicative morphism $\phi: \bbU \times \bbU \times \bbU \to \bbG_m$ we can construct a continuous tri-linear map
 $\langle\cdot, \cdot, \cdot\rangle: K \times K \times K
\to {\mathbf Q}$ by a formula analogous to formula~\eqref{trilin}. From the above reasonings we have the fact:
  the morphism $\phi$
 is uniquely defined by the map $\langle\cdot, \cdot, \cdot\rangle$.}.  Since $C_3$ is a tri-multiplicative, continuous, anti-symmetric and invariant under
 the change of local parameters map (see Proposition~\ref{invar}), the map $\langle\cdot, \cdot, \cdot\rangle$ is  tri-linear, continuous, anti-symmetric and invariant
 under the change of
 local parameters.
Now the first formula in~\eqref{eqc3cc} coincides with the last
formula in~\eqref{eqc3cc} by Lemma~\ref{calcu2} (from which it follows that $\langle p,q,r\rangle = c \Res( p dq \wedge dr) $, where $c \in \mathbf{Q}$) and from the
following easy direct calculation of $\langle \cdot , \cdot , \cdot\rangle $ on particular elements $p$, $q$ and $r$:
\begin{eqnarray*}
C_3(1+  u\varepsilon_1, 1+ t \varepsilon_2, 1+  u^{-1}t^{-1}
\varepsilon_3)= C_3(t+u^{-1} \varepsilon_3, 1+  u\varepsilon_1, 1+ t
\varepsilon_2 ) = \\ = \on{Comm}(\pi_{f'}(1+  u\varepsilon_1),
\pi_{f'}( 1+ t \varepsilon_2))= 1 + \varepsilon_1  \varepsilon_2
\varepsilon_3 \mbox{,}
\end{eqnarray*}
where $f'=t+u^{-1} \varepsilon_3 $ and we used formula~\eqref{comm}.
\end{proof}

From Lemma~\ref{sim} we have that
\begin{equation} \label{immed}
d_{(i,j),(k,l),(m,n)} = (lm-nk)\delta_{i+k+m,0}\delta_{j+l+n,0}
\end{equation}
in formula~\eqref{coinc} for both
morphisms $C_3$ and $(\cdot, \cdot, \cdot)$ composed with $\exp$-maps.
Therefore we have   checked\footnote{We note that from formula~\eqref{immed} we immediately obtain, for example,
that $C_3(f,g,h)=1$ if all elements $f$, $g$ and $h$ belong to $\bbP(R)$ or all elements $f$, $g$ and $h$ belong to $\bbM(R)$} the property~\eqref{2cexp}  for elements $f \in (\bbP \times \bbM)(R)$, $g \in (\bbP \times \bbM)(R)$ and $h \in (\bbP \times \bbM)(R)$.

\smallskip

Now we will check property~\eqref{2cexp} when $f$, $g$ are any elements from  $(\bbP \times \bbM) (R) $, and $h$ is fixed and equal either to $t$ or
to $u$. As in the previous case, it will be convenient  for us to consider elements $f$ and $g$  from the bigger group ind-scheme, i.e. $(\bbP \times \bbM \times (1 +\n)) (R) $, where the group ind-schemes $1+\n \stackrel{\exp}{\longleftarrow} \widehat{\bbG}_a$ are isomorphic.
As in the previous case, using Proposition~\ref{weak version}, we can reduce the proof to the analysis of
the following case:
 $f=\exp(au^jt^i)$ and
$g=\exp(bu^lt^k)$,  where $a$, $b$ from the ring $R$, and $a \in \n R$ if $(i,j) \le 0$, $b \in \n R$ if $(k,l) \le 0$ since in this case they belong to $(\bbP \times \widehat{\bbG}_a)(R)$.  We consider the map: $\widehat{\bbG}_a \to \bbP \times (1+\n)$,
$c\mapsto \exp(cu^mt^n)$ if $(n,m) \le 0$,  and the map $\bbG_a\to\bbM$: $c\mapsto
\exp(cu^mt^n)$ if $(n,m) > 0$. Using these maps and restricting the maps $C_3$ and $(\cdot,
\cdot, \cdot)$ to elements $f$ and $g$ as above (under fixed indices $(i,j)$ and $(k,l)$ from the set $\mathbf{Z} \times \mathbf{Z}$) and $h=u$ we obtain two
bimultiplicative maps
\begin{equation} \label{rp}
{\bbH}_1 \times \bbH_2  \longrightarrow \bbG_m \mbox{,}
\end{equation}
where the group ind-scheme $\bbH_i$ ($1 \le i \le 2$) is isomorphic either to $\bbG_a$ or to $\widehat{\bbG}_a$.

Now we have the full analog of Lemma~\ref{hom} with the analogous
proof. If $\bbH_1 =\bbH_2 = \bbG_a$, then any bimultiplicative
morphism of type~\eqref{rp} is equal to $1$. In other cases of the
group ind-schemes $\bbH_i$ the set all such bimultiplicative
morphisms is isomorphic to the set $\bbG_a(\mathbf{Q})$: by an
element $d \in \mathbf{Q}$ we construct the morphism $\phi_d$ of
type~\eqref{rp} with an explicit formula
\begin{equation} \label{two}
\phi_d(a,b)= \exp(dab) \mbox{,}
\end{equation}
where $a \in \bbH_1(R)$ and $b \in \bbH_2(R)$ for any ring $R$.

Now to calculate
that  $d$ coincides for the maps $C_3$ and $(\cdot, \cdot, \cdot)$ composed with $\exp$-maps as above, we
consider the ring $R= \mathbf{Q}[\varepsilon_1, \varepsilon_2]/(\varepsilon_1^2,
\varepsilon_2^2) $ and  $R$-points of group ind-schemes as in~\eqref{rp}.
 We need a lemma.
\begin{lem}\label{sim2} Let
$R=\mathbf{Q}[\varepsilon_1,\varepsilon_2]/(\varepsilon_1^2,\varepsilon_2^2)$.
Let $f=\exp(\varepsilon_1 u^jt^i)\in L^2\bbG_m(R)$, $g=\exp(\varepsilon_2
u^lt^k)\in L^2\bbG_m(R)$, then
\begin{equation} \label{eqc3cc2}
C_3(f,g,u)=(f,g,u) = \exp(i
\delta_{i+k,0}\delta_{j+l,0}\varepsilon_1  \varepsilon_2)  \mbox{.}
\end{equation}
\end{lem}
\begin{proof}
The second equality in~\eqref{eqc3cc2} follows from
formula~\eqref{2cexp}   by easy direct calculations.
 Now it is enough to calculate
$ C_3(1+u^jt^i \ve_1, 1+ u^lt^k \ve_2,u) $. By above reasonings with
bimultiplicative maps this expression has to be equal to $1 + d
\ve_1 \ve_2$, where $d \in {\mathbf Q}$ depends on $i,j,k,l$. We
will compute  this $d$. We note that for any $a, b, c \in {\mathbf
Q}$ it is true an equality (by formula~\eqref{cc3} which we already checked):
$$
C_3(1+bu^jt^i \ve_1, 1+ cu^lt^k \ve_2,a)=1 \mbox{.}
$$
We consider the change of local parameters: $t'= 2t$, $u'=3u$. Since
$C_3$ is invariant under the change of local parameters (by Proposition~\ref{invar}) and $C_3$ is a
tri-multiplicative map, we have
\begin{eqnarray*}
1+ d\ve_1 \ve_2 =  C_3(1+u^jt^i \ve_1, 1+ u^lt^k \ve_2,u)=
C_3(1+{u'}^j{t'}^i \ve_1, 1+ {u'}^l{t'}^k \ve_2,u') = \\ =
C_3((1+u^jt^i \ve_1)^{2^j3^i},  (1+ u^lt^k \ve_2)^{2^l3^k}  , u)
=1+2^{j+l}3^{i+k}d \ve_1 \ve_2 \mbox{.}
\end{eqnarray*}
Hence we have that $d=0$ if $j+l \ne 0$ or $i+k \ne 0$. It means
that we have to compute now $ C_3(1+u^jt^i \ve_1, 1+ u^{-j}t^{-i}
\ve_2,u) $. Without loss of generality we assume that $i \le 0$. If $i=0$, then the last
expression is equal to $1$, since elements $u, 1+u^j, 1+u^{-j}$
preserve the lattice $R((u))[[t]]$. If $i=-1$, then we consider the
change of the local parameter $t' = u^jt$. The map $C_3$ is
invariant under this change. Therefore we obtain
\begin{eqnarray*}
C_3(1+u^jt^{-1} \ve_1, 1+ u^{-j}t \ve_2, u) =
C_3(1+u^j{t'}^{-1} \ve_1, 1+ u^{-j}{t'} \ve_2, u) = \\
= C_3(1+t^{-1} \ve_1, 1+ t \ve_2, u) = C_3(t+ \ve_1, 1+t \ve_2, u)=
1 -\ve_1 \ve_2  \mbox{.}
\end{eqnarray*}
 (Here the last equality is easily calculated in the ring $R((u))[[t]]/ (t+\ve_1)R((u))[[t]]$ by formula~\eqref{comm}.)

Therefore we suppose that $i<-1$. Then $t' = t + u^{-j}t^{-i}$ is a
well-defined change of the local parameter $t$. We have
\begin{eqnarray*}
1= C_3( 1+ u^{j}t^{i} \ve_1,
 1+ t \ve_2,
 u)=C_3(
  1+ {u}^{j}{t'}^{i} \ve_1,
 1+ t' \ve_2,
 u) = \\
= C_3(
  1+ ( u^{j}t^{i} +i t^{-1} + \ldots )\ve_1,
1+ (t+u^{-j}t^{-i})\ve_2,
 u) \mbox{.}
\end{eqnarray*}
Hence, using the continuous property of $C_3$ (see Proposition~\ref{lemma-main}), we obtain that
\begin{eqnarray*}
C_3(
 1+ u^{j}t^{i} \ve_1,
1+u^{-j}t^{-i} \ve_2, u)= C_3(
 1+ t^{-1} \ve_1,
 1+t \ve_2,
u)^{-i} = \\
= C_3(t+ \ve_1, 1+t \ve_2, u)^{-i}    =1 +i \ve_1 \ve_2 \mbox{.}
\end{eqnarray*}
\end{proof}
Now from formula~\eqref{eqc3cc2} we see that $d = i
\delta_{i+k,0}\delta_{j+l,0}$ for both maps $C_3$ and $( \cdot , \cdot ,
\cdot)$ composed with $\exp$-maps on indices $(i,j)$, $(k,l)$ from the set $\mathbf{Z} \times  \mathbf{Z}$. Thus, by above reasonings, we have checked formula~\eqref{2cexp}
when elements $f$ and $g$ are from $(\bbP \times \bbM)(R)$ and $h=u$.

Formula~\eqref{2cexp} with elements $f$ and $g$  from $(\bbP \times \bbM)(R)$ and $h=t$ follows from the same arguments and a lemma.
\begin{lem}\label{sim3} Let
$R=\mathbf{Q}[\varepsilon_1,\varepsilon_2]/(\varepsilon_1^2,\varepsilon_2^2)$.
Let $f=\exp(\varepsilon_1 u^jt^i)\in L^2\bbG_m(R)$, $g=\exp(\varepsilon_2
u^lt^k) \in L^2\bbG_m(R)$, then
$$
C_3(f,g,t)=(f,g,t) = \exp(l
\delta_{i+k,0}\delta_{j+l,0}\varepsilon_1  \varepsilon_2)  \mbox{.}
$$
\end{lem}
\begin{proof}
The second equality follows from formula~\eqref{2cexp}   by easy
direct calculations. Now it is enough to calculate $ C_3(1+u^jt^i
\ve_1, 1+ u^lt^k \ve_2,t) $. By above reasonings with
bimultiplicative maps this expression has to be equal to $1 + d
\ve_1 \ve_2$, where $d \in {\mathbf Q}$ depends on $i,j,k,l$. We
will compute  this $d$. We note that for any $a, b, c \in {\mathbf
Q}$ it is true an equality:
$$
C_3(1+bu^jt^i \ve_1, 1+ cu^lt^k \ve_2,a)=1 \mbox{.}
$$
We consider the change of local parameters: $t'= 2t$, $u'=3u$. Since
$C_3$ is an invariant under the change of local parameters and
tri-multiplicative map, we have
\begin{eqnarray*}
1+ d\ve_1 \ve_2 =  C_3(1+u^jt^i \ve_1, 1+ u^lt^k \ve_2,t)=
C_3(1+{u'}^j{t'}^i \ve_1, 1+ {u'}^l{t'}^k \ve_2,t') = \\ =
C_3((1+u^jt^i \ve_1)^{2^j3^i},  (1+ u^lt^k \ve_2)^{2^l3^k}  , t)
=1+2^{j+l}3^{i+k}d \ve_1 \ve_2 \mbox{.}
\end{eqnarray*}
Hence we have that $d=0$ if $j+l \ne 0$ or $i+k \ne 0$. It means
that we have to compute now $ C_3(1+u^jt^i \ve_1, 1+ u^{-j}t^{-i}
\ve_2,t) $.
 Omitting trivial cases (when by formula~\eqref{comm} the computation is done in $R((u))$), we can suppose that
$(-i,-j) > (0,1)$. Then $u' = u + u^{-j}t^{-i}$ is a well-defined
change of the local parameter $u$. We have
\begin{eqnarray*}
1= C_3( 1+ u^{j}t^{i} \ve_1,
 1+ u \ve_2,
 t)=C_3(
  1+ {u'}^{j}t^{i} \ve_1,
 1+ u' \ve_2,
 t) = \\
= C_3(
  1+ ( u^{j}t^{i} +j u^{-1} + \ldots )\ve_1,
1+ (u+u^{-j}t^{-i})\ve_2,
 t) \mbox{.}
\end{eqnarray*}
Hence, using the continuous property of the map $C_3$, we obtain that
$$
C_3(
 1+ u^{j}t^{i} \ve_1,
1+u^{-j}t^{-i} \ve_2, t)= C_3(
 1+ u^{-1} \ve_1,
 1+u \ve_2,
t)^{-j} = 1 -j \ve_1 \ve_2 \mbox{.}
$$
\end{proof}

At the end we have to verify that $C_3 (f,t,t)=C_3(f,u,u)= C_3(f,u,t)=1$ when $f \in (\bbP \times \bbM)(R)$.
It can be done, for example, in the following way. We fix $g$ and $h$ from the two-element set $\{t,u \}$. Arguing as above
and using the facts $\underline{\Hom}_{\mbox{gr}}(\bbG_a, \bbG_m)= \widehat{\bbG}_a$ and $\underline{\Hom}_{\mbox{gr}}(\widehat{\bbG}_a, \bbG_m)= \bbG_a$
we see that it is enough to prove an equality:
$$
C_3(1 + u^j t^i \ve, g,h)=1 \mbox{.}
$$
when the ring $R=\mathbf{Q}[\ve]/\ve^2$ and $(i,j) \in \mathbf{Z} \times \mathbf{Z} \setminus (0,0)$. If $(i,j) >0$, then this equality follows from the fact $\widehat{\bbG}_a(\mathbf{Q})=1$. If $(i,j) <0$, then  we have that $C_3(1 + u^j t^i \ve, g,h)=1+ d\ve  $ where $d \in \bbG_a(\mathbf{Q})$. We consider a change of local parameters: $u'=2u$,
$t'=3t$. Then we obtain
$$
1+d\ve=C_3(1+ u^jt^i \ve, g,h)= C_3(1+ {u'}^j{t'}^i \ve, g,h)= C_3(1+2^j3^i u^jt^i \ve, g, h)= 1+ 2^j3^i d \ve \mbox{.}
$$
Hence we have $d=0$. We have finished to verify the last case of the theorem.
\end{proof}

\begin{cor} \label{cor} Let $R \in \B$.
The two-dimensional   Contou-Carr\`{e}re  symbol $(\cdot, \cdot, \cdot)$ given by the explicit formula
from Definition~\ref{defcon2} and the map $C_3$ coincide as the maps from $({L^2\bbG_m})^3(R)$ to $\bbG_m(R)$.
\end{cor}
\begin{proof}
The map $C_3$ is functorial with respect  to the ring $R$. By
Proposition~\ref{prcom}, the two-dimensional   Contou-Carr\`{e}re
symbol $(\cdot, \cdot , \cdot)$ given by an explicit formula from
Definition~\ref{defcon2} is also functorial with respect to the ring
$R$. Therefore using Lemma~\ref{lift} we can reduce the proof to the
case of a $\mathbf{Q}$-algebra $R$ from $\B$. Now we apply
Theorem~\ref{th-main}  and the end of Proposition~\ref{prcom} to
show that these two maps coincide.
\end{proof}

\section{Reciprocity laws}
In this section we fix a perfect field $k$ and a local finite $k$-algebra $R$.

Let $V$ be a Tate $R$-module. We note that any projective $R$-module is a free.
Since the ring $R$ is an Artinian ring,  any connected Nisnevich covering of $\spec R$ is $\spec R$ itself.
Hence, using Drinfeld's theorem (see~\cite[Th.~3.4]{Dr} and also an explicit exposition in~\cite[\S 2.12]{BBE}), we obtain that $V$ is an elementary Tate $R$-module, i.e. $V= P \oplus Q^*$,
where $P$ and $Q$ are discrete free $R$-modules.

We will use also the following remark.
\begin{rmk} \label{sum}
Let $\phi: M \to N$ be an open continuous surjection between Tate $R$-modules
such that these topological $R$-modules have countable bases of open neighborhoods of $0$. Then by Lemma~\ref{split}  we have topological decomposition
 $M = \Ker \phi \oplus N $ and $\Ker \phi $ is a Tate $R$-module.
We have also  a canonical isomorphism of $\mathcal{P}ic_R^{\bbZ}$-torsors:
$$
{\mathcal D}et({\Ker \phi}) + {\mathcal D}et({N})  \lrto {\mathcal D}et({M})   \mbox{,}
$$
because for any coprojective lattices $L_1$ from $\Ker \phi$  and $L_2$ from  $N$ we have that $L_1 \oplus L_2 $ is a coprojective lattice in $M$ (in other words, we can find a coprojective lattice in $M$ such that its intersection  with $\Ker \phi$ and its image in $N$ will be  coprojective lattices in $\Ker \phi$ and in $N$ correspondingly).
\end{rmk}

\medskip
Now we introduce the norm map which generalizes the usual norm map for extensions of fields.
Let $L \supset K \supset k $ will be  finite extensions of fields. We recall that $R$ is a finite local $k$-algebra and $k$ is a perfect field.  We {\em define} the norm map
\begin{equation}  \label{defnorm}
\Nm_{L/K} \; : \;  (L \otimes_k R)^* \lrto (K \otimes_k R)^*   \mbox{,}  \quad a \mapsto \prod_{i=1}^n \sigma_i(a)  \mbox{,}
\end{equation}
where the $\sigma_i$   are all the isomorphisms of $L$  into the algebraic closure $\bar{K}$   of $K$  fixing the elements of $K$,
and $\sigma_i$ is extended to the ring homomorphism $L \otimes_k R  \to \bar{K} \otimes_k R $ by the natural rule
$b \otimes x \mapsto  \sigma_i(b) \otimes x$ where  $b \in L$, $x \in R$. It is clear that the map $\Nm_{L/K}$ is well-defined\footnote{To see that $\Nm_{L/K} \subset (K \otimes R)^* $ we note that $(\bar{K} \otimes_K R)^{\Gal{\bar{K}/K}} = K \otimes R$.} and this map is a multiplicative map.
Besides, from the corresponding property of isomorphisms of fields we have that for any finite extension $M \supset L$ of fields the  maps
$\Nm_{L/K} \circ \Nm_{M/L}$ and $\Nm_{M/K}$ from $(M \otimes_k R)^*$ to  $(K \otimes R)^*$ coincide.

\smallskip

Let $X$ be a smooth connected algebraic surface over $k$. For any
closed point $x \in X$ let $\ho_x$ be a completion of the local ring
$\oo_x$ at the point $x$. Let $K_x$ be the localization of the ring
$\ho_x$ with respect to the multiplicative system $\oo_x \setminus
0$. For any irreducible curve $C$ on $X$ (in other words, for any
integral one-dimensional subscheme $C$ of $X$) let $K_C$ be a field
which is the completion of the field $k(X)$ with respect to the
discrete valuation given by the curve $C$. By any pair $x \in C$,
where $x$ is a closed point and $C \subset X$ is an irreducible
curve (which contains $x$) we will canonically construct the ring
$K_{x,C}$ (see also details in a survey~\cite{O4}). We consider the
decomposition
 \begin{equation}  \label{ext}
 C \mid_{\spec \ho_x} = \bigcup_{i=1}^s {\mathbf C}_i  \mbox{,}
 \end{equation}
 where every ${\mathbf C}_i$ is an integral one-dimensional subscheme in $\spec \ho_x$. We {\em define}
 \begin{equation}  \label{xC}
 K_{x,C} = \prod_{i=1}^s K_i \mbox{,}
 \end{equation}
 where the field $K_i=K_{x, \mathbf{C}_i}$ is the completion of the field $\Frac \ho_x$ with respect to the discrete valuation given by ${\mathbf C}_i $.

 We consider any one-dimensional integral subscheme ${\mathbf C}  \subset \spec \ho_x$.
We consider
 the field $K= K_{x, \mathbf{C}}$ which is the completion of the field $\Frac \ho_x$ with respect to the discrete valuation given by ${\mathbf C} $.
Let $M$ be the residue field of the discrete valuation field   $K$. Since $\ho_x = k(x)[[v,w]]$ where $k(x)$ is the residue field of the point $x$, by the Weierstrass preparation theorem we have that the field $M$ is a finite extension  of at least one  of the fields: $k(x)((v))$ or  $k(x)((w))$.
Therefore $M$ is a complete discrete valuation field with the residue field $k'$ which is a finite extension of the field $k(x)$.
  Hence
  \begin{equation}  \label{istop}
  K = M((t))=k'((u))((t))
  \end{equation}
  for some $u$ and $t$ from $\Frac \ho_x$, i.e $K$ is a two-dimensional local field. (We used that on $K$ there is the natural topology of inductive and projective limits which extends the topology on $k(x)[[u,v]]$. This topology comes to the  topology of two-dimensional local field under isomorphisms~\eqref{istop}.)

  Let $F $ be a field such that
$k \subset F \subset k(x)$. Let $L$ be a field such that $L \supset k'$ and $L \supset F$ is a finite Galois extension. We have that
$$
k' \otimes_F L = \prod_{l=1}^m L_{\sigma_l}  \mbox{,}
$$
where $L_{\sigma_l}=L$  and the product is taken over all the isomorphisms $\sigma_l$ of $k'$ into $L$  fixing the elements of $F$,
and  $m= [ k' : k]$.  We consider a scheme
$$Y = \spec \ho_x \times_{\spec F} \spec L$$ with a canonical morphism $p$ from this scheme to $\spec \ho_x$. Let $\mathbf T$ be the set of all pairs $y \in D$
where $D$ is a one-dimensional integral subscheme of $Y$ and $y$ is a closed point on $D$  such that $p(y)=x$ and $p(D)= {\mathbf C}$. From  the properties  of complete discrete valuation fields and extensions of valuations we have isomorphisms of rings
\begin{equation}  \label{norm}
\prod_{l=1}^m L_{\sigma_l} ((u))((t))= k'((u))((t)) \otimes_F L = K \otimes_F L = \prod_{\{ y \in D \} \in {\mathbf T}} K_{y,D}  \mbox{,}
\end{equation}
where any $K_{y,D}$ is a two-dimensional local field and it is equal to $L_{\sigma_l}((u))((t))$ for some $l$. We have that the set $\mathbf T$ consists of $m$ elements
and the group $\Gal(L/F)$ acts  by permutations of direct summands in~\eqref{norm} such that this group acts on $L_{\sigma_l}$ as
$\sigma (e_{\sigma_l})=\sigma (e)_{\sigma \sigma_l}$ for $\sigma \in \Gal(L/F)$ and $e_{\sigma_l} \in L_{\sigma_l}$. Besides, we have
that in view of isomorphisms~\eqref{norm} the embedding $K \to K \otimes_F L  : x \mapsto x \otimes 1$ is given as
\begin{equation}  \label{evid}
\sum_{p,q} a_{p,q}u^q t^p   \longmapsto \prod_{l=1}^m  \sum_{p,q} \sigma_l(a_{p,q}) u^q t^p   \quad \mbox{where} \; a_{p,q} \in k'  \mbox{.}
\end{equation}

\smallskip

For any pair $x \in C$ as above, we have in formula~\eqref{xC}  that $K_i=k_i((u_i))((t_i))$, where the field $k_i$ is a finite extension of the field $k(x)$. We consider
$$K_{x,C} \otimes_k R = \prod_{i=1}^s (K_i \otimes_k R) =  \prod_i^{s} (k_i \otimes_k R)((u_i))((t_i))  \mbox{.}$$
Let the map
$(\cdot, \cdot, \cdot)_i :  ((k_i \otimes_k R)((u_i))((t_i))^*)^3   \to (k_i \otimes_k R)^*  $
be the two-dimensional Contou-Carr\`{e}re symbol.
 We {\em define} a map:
$$
(\cdot, \cdot, \cdot )_{x,C} \quad : \quad
(K_{x,C}  \otimes_{k} R)^*  \times (K_{x,C}  \otimes_{k} R)^* \times (K_{x,C}  \otimes_{k} R)^*  \lrto (k(x) \otimes_k R)^*
$$
in the following way\footnote{We note that from \S~\ref{seccentr}  we have that $(\cdot, \cdot, \cdot)$ does not depend on the choice of local parameters $u_i$ and $t_i$.}
\begin{equation} \label{omit}
(f,g,h)_{x,C} = \prod_{i=1}^s  \Nm_{k_i /k(x)} (f_i, g_i, h_i)_i
\end{equation}
where $f= \prod_{i=1}^s f_i$, $f_i \in (K_i \otimes_k R)^*$, and the same notations we take for $g$ and $h$.

We note that for any irreducible curve $C \subset X$ we have the canonical embedding $K_C \hookrightarrow K_{x,C}$ for any closed point $x \in C$.
This embedding induces a map $K_C \otimes_k R \to K_{x,C} \otimes_k R$.

For any closed point $x \in X$ we have the canonical embedding $K_x \hookrightarrow K_{x,C}$ for any irreducible curve $C \subset X$ which contains the point $x$. This embedding induces a map $ K_x \otimes_k  R \to K_{x,C} \otimes_k R  $.

\begin{thm}[Reciprocity laws for the two-dimensional Contou-Carr\`{e}re symbol]  \label{res-law}
Let $X$ be a smooth connected algebraic surface over a perfect field $k$. Let $R$ be a local finite $k$-algebra. The following reciprocity laws are satisfied.
\begin{enumerate}
\item  \label{part1} Let $x$ be a closed point on $X$. Then for any $f$, $g$ and $h$ from $(K_x \otimes_k R)^*$ we have
\begin{equation}  \label{point}
\prod_{C \ni x} (f,g,h)_{x,C} =1  \mbox{,}
\end{equation}
where this product is taken over all irreducible curves $C$ containing the point $x$ on $X$ and  in the product only finitely many  terms are distinct from $1$.
\item \label{part2}
Let $C$ be a projective irreducible curve on $X$. Then for any $f$, $g$ and $h$ from $(K_C \otimes_k R)^*$ we have
\begin{equation}  \label{curve}
\prod_{x \in C} \Nm_{k(x)/k} (f,g,h)_{x,C} =1  \mbox{,}
\end{equation}
where this product is taken over all closed points $x$ on $C$ and  in the product only finitely many  terms are distinct from $1$.
\end{enumerate}
\end{thm}
\begin{proof}
Before to prove  parts~\eqref{part1} and~\eqref{part2} of the theorem we will make some general remarks which will be useful for the proof of both parts.

We note that  formula~\eqref{point} depends only on the two-dimensional local regular ring $\oo_x$.

From formula~\eqref{norm} and the above description of the field $K_{x, \mathbf{C}}$ we have that if an integral one-dimensional subscheme $\mathbf{C} \subset \spec \ho_x$   (for some point $x$)  splits in  $\spec (\ho_x \hat{\otimes}_F \bar{k}) $ (where $k \subset F \subset k(x)$ and the field
$\bar{k}$ is an algebraic closure of $k$), then it splits on the same irreducible  components   over some finite Galois extension $L \supset k$.
Conversely, if $\bar{C} \subset X \otimes_{\spec k} \spec{\bar{k}}$ is an irreducible curve, then $\bar{C}$ is defined over some finite extension $L \supset k$.
 Now every $\mathbf{\bar{C}}_i$ (see formula~\eqref{ext}) which is  defined over $L$ comes from some $\mathbf{\bar{C}} \subset \spec \ho_x$ after the base change given by the field extension from $k(x)$ to $L$.

Using formula~\eqref{evid}, the definition of the norm map  given by formula~\eqref{defnorm} and the functoriality of $(\cdot, \cdot, \cdot)$  we reduce at once the proof of this theorem to the case of an algebraically closed ground  field  if we consider the scheme
$\spec \oo_x \times_{\spec k(x)} \spec{\bar k} $ instead of $\spec \oo_x$ for the proof of formula~\eqref{point} and the scheme
$X \times_{\spec k}  \spec(\bar{k})$ instead of $X$ for the proof of formula~\eqref{curve}. So, we assume that $k= \bar{k}$ is an algebraically closed field. Besides, we can omit the norm maps in formulas~\eqref{omit}   and~\eqref{curve}.

We can assume that $X$ is connected.

By Corollary~\ref{cor} of Theorem~\ref{th-main} we know that $(\cdot, \cdot, \cdot) =C_3$. Therefore we will prove the reciprocity laws~\eqref{point}
and~\eqref{curve} for the map $C_3$. Our strategy to prove these reciprocity laws is to repeat the proof of Theorem~5.3 from~\cite{OsZh},
but to change the scheme $X$ to the scheme $X_R = X \otimes_{\spec k} \spec R$  and consider  local rings and adelic complexes on the scheme $X_R$
which has the same topological space as the scheme $X$.

\smallskip

\eqref{part1}
We will give a sketch of the proof of formula~\eqref{point}.
We fix a point $x \in X$.
 Similarly to~\cite[\S~5B]{OsZh} we look at the scheme
$U_{x,R}= (\spec {\ho_x \otimes_{k} R}) \setminus x$.  For any $f
\in  (\Frac(\ho_x) \otimes_k R)^*$ we consider the coherent subsheaf
$f \cdot (\ho_x\otimes_{k} R)$ of the constant sheaf $\Frac(\ho_x)
\otimes_k R$ on the scheme $\spec {\ho_x \otimes_{k} R}$. The adelic
complex $\ad_{X,x,R}(f \cdot (\ho_x\otimes_{k} R))$ of the
restriction of the sheaf $f \cdot(\ho_x\otimes_{k} R)$ to the scheme
$U_{x,R}$ looks as follows\footnote{For various $f$ the
corresponding adelic complexes $\ad_{X,x,R}(f \cdot
(\ho_x\otimes_{k} R))$  are isomorphic because the sheaves are
isomorphic. We are interested in the position of this adelic complex
inside of the adelic complex of the restriction  of the constant
sheaf $\Frac(\ho_x) \otimes_k R$ to the scheme $U_{x,R}$}:
$$
\da_{X,x,0,R} (f \cdot(\ho_x\otimes_{k} R)) \oplus \da_{X,x,1,R} (f \cdot( \ho_x\otimes_{k} R))
 \lrto \da_{X,x,01,R} (f \cdot(\ho_x\otimes_{k} R))
$$
and this complex is isomorphic to the subcomplex $f \cdot
(\ad_{X,x}( \ho_x) \otimes_k R)$ of the complex $\ad_{X,x}(\Frac
\ho_x) \otimes_k R$, where $\ad_{X,x}(\cdot)$ is the corresponding
adelic complex when $R=k$. Therefore we have
\begin{eqnarray*}
\da_{X,x,0,R} (f \cdot(\ho_x\otimes_{k} R)) =\Frac(\ho_x) \otimes_k
R\\
\da_{X,x,1,R} (f \cdot( \ho_x\otimes_{k} R)) = \prod_{\mathbf{C} \ni
x} f \cdot ( \hat{\oo}_{x, \mathbf{C}} \otimes_k R) \\
\da_{X,x,01,R} (f \cdot( \ho_x\otimes_{k} R)) = {\prod_{\mathbf{C}
\ni x}}' K_{x, \mathbf{C}} \otimes_k R  \mbox{,}
\end{eqnarray*}
where $\mathbf{C}$ runs over all one-dimensional integral subschemes
in $\spec \ho_x$, a ring $\hat{\oo}_{x, \mathbf{C}}$ is the discrete
valuation ring in the two-dimensional local field $K_{x,
\mathbf{C}}$, an expression $f \cdot ( \hat{\oo}_{x, \mathbf{C}}
\otimes_k R)$ is considered inside of the ring $K_{x, \mathbf{C}}
\otimes_k R$, and $\prod'$ is the restricted product with respect to
the rings $\hat{\oo}_{x, \mathbf{C}} \otimes_k R$. Besides,
$$H^0(\ad_{X,x,R}(f \cdot (\ho_x\otimes_{k} R))) = f \cdot
(\ho_x\otimes_{k} R)$$ is a compact Tate $R$-module (i.e. it is dual
to the free discrete $R$-module), and $ H^1(\ad_{X,x,R}(f \cdot
(\ho_x\otimes_{k} R))) $ is isomorphic to $ H^1(\ad_{X,x}( \ho_x ))
\otimes_k R $ which is a discrete Tate $R$-module.

Let $f$ and $g$ are from $(\Frac(\ho_x) \otimes_k R)^*$. For any
$\mathbf{C}$ we have that $f$ and $g$ belong to $(K_{x, \mathbf{C}}
\otimes_k R)^*$. Therefore for any $\mathbf{C}$ there is
$n_{\mathbf{C}} \in \mathbf{N}$ such that inside the ring $K_{x,C}
\otimes_k R$ we have
$$t_{\mathbf{C}}^{n_{\mathbf C}} \ho_{x, \mathbf{C} }
\otimes_k R \subset f \cdot (\ho_{x, \mathbf{C}} \otimes_k R) \qquad
\mbox{and} \qquad t_{\mathbf{C}}^{n_{\mathbf C}} \ho_{x, \mathbf{C}
} \otimes_k R \subset g \cdot (\ho_{x, \mathbf{C}} \otimes_k R)
\mbox{,}
 $$
 where $t_{\mathbf{C}} \in \ho_{x}$ gives an equation $t_{\mathbf{C}}=0$ of
 $\mathbf{C}$ in $\spec \ho_x $. Since we can take $n_{\mathbf{C}}
 =0$ for almost all $\mathbf{C}$, an element $h = \prod\limits_{\mathbf{C} \ni x} t_{\mathbf{C}}^{n_{\mathbf{C}}}$
 from $\ho_x$ is well-defined. Thus, we have constructed the element $h \in (\Frac(\ho_x) \otimes_k
 R)^*$ such that we have the following embeddings of adelic
 complexes
 \begin{eqnarray*}
\ad_{X,x,R}( h \cdot (\ho_x\otimes_{k} R)) \subset \ad_{X,x,R}(f
\cdot (\ho_x\otimes_{k} R)) \\
\ad_{X,x,R}(h \cdot (\ho_x\otimes_{k} R))  \subset \ad_{X,x,R}(g
\cdot (\ho_x\otimes_{k} R))  \mbox{.}
 \end{eqnarray*}
 We note that among such constructed elements $h$ there is a "minimal" element when all the integers $n_{\mathbf{C}}$ are minimal.
We have that
$$\da_{X,x,1,R} (f_1 \cdot( \ho_x\otimes_{k} R)) / \da_{X,x,1,R} (f_2 \cdot( \ho_x\otimes_{k} R))  =
\bigoplus_{C \ni x} (f_1 \cdot (\ho_{x, \mathbf{C}} \otimes_k R)) / (f_2 \cdot (\ho_{x, \mathbf{C}} \otimes_k R))
\mbox{,} $$
where $f_1$ and $f_2$ are from $(\Frac(\ho_x) \otimes_k R)^*$ such that $f_1 \cdot( \ho_x\otimes_{k} R) \supset f_2 \cdot( \ho_x\otimes_{k} R)$
is a Tate $R$-module, because for almost all $\mathbf{C} \ni x$ we have
$$f_1 \cdot (\ho_{x, \mathbf{C}} \otimes_k R) =
f_2 \cdot (\ho_{x, \mathbf{C}} \otimes_k R) = \ho_{x, \mathbf{C}} \otimes_k R \mbox{.}$$

We have that
\begin{multline*}
\mathcal{D}et(\da_{X,x,1,R} (f\cdot( \ho_x\otimes_{k} R))  \mid \da_{X,x,1,R} (g\cdot( \ho_x\otimes_{k} R))) = \\ =
\mathcal{D}et(\da_{X,x,1,R} (g\cdot( \ho_x\otimes_{k} R))  / \da_{X,x,1,R} (h\cdot( \ho_x\otimes_{k} R))) - \\ -
\mathcal{D}et(\da_{X,x,1,R} (f\cdot( \ho_x\otimes_{k} R))  / \da_{X,x,1,R} (h\cdot( \ho_x\otimes_{k} R)))
\end{multline*}
is a well-defined  $\mathcal{P}ic_{R}^{\mathbb{Z}}$-torsor. Using the explicit description of the cohomology groups of the adelic complex given above, we have  that for any $f \in (\Frac(\ho_x) \otimes_k
 R)^*$
\begin{multline*}
\mathcal{D}et(H^*(\ad_{X,x,R}(f \cdot (\ho_x\otimes_{k} R)) )) = \mathcal{D}et(H^0(\ad_{X,x,R}(f \cdot (\ho_x\otimes_{k} R))) ) - \\- \mathcal{D}et(H^1(\ad_{X,x,R}(f \cdot (\ho_x\otimes_{k} R))) )
\end{multline*}
is a well-defined  $\mathcal{P}ic_{R}^{\mathbb{Z}}$-torsor. Using Remark~\ref{sum}  and decompose long exact cohomological sequences into the split-exact short sequences of Tate $R$-modules (when $f\cdot( \ho_x\otimes_{k} R) \subset g\cdot( \ho_x\otimes_{k} R) )$ we obtain that
\begin{multline*}
\mathcal{D}et(\da_{X,x,1,R} (f\cdot( \ho_x\otimes_{k} R))  \mid \da_{X,x,1,R} (g\cdot( \ho_x\otimes_{k} R))) = \\ =
\mathcal{D}et(H^*(\ad_{X,x,R}(g \cdot (\ho_x\otimes_{k} R)) )) - \mathcal{D}et(H^*(\ad_{X,x,R}(f \cdot (\ho_x\otimes_{k} R)) )) \mbox{.}
\end{multline*}
Hence we have a trivialization (given by multiplications on the elements of the group $(\Frac(\ho_x) \otimes_k
 R)^*$) of a categorical central extension
 $$
 f \mapsto  \mathcal{D}et(\da_{X,x,1,R} ( \ho_x\otimes_{k} R)  \mid \da_{X,x,1,R} (f\cdot( \ho_x\otimes_{k} R)))
 $$
over the group $(\Frac(\ho_x) \otimes_k
 R)^*$. From this fact, as in the proof of Theorem~5.3 from~\cite{OsZh} we obtain the reciprocity law around the point $x$ for any elements $f$, $g$ and $h$ from
 the group $(\Frac(\ho_x) \otimes_k R)^*$ when we take in formula~\eqref{point}  the product over all one-dimensional integral subschemes $\mathbf{C}$
 of the scheme $\spec \ho_x$. This product contain only finitely many terms distinct from $1$, because for almost all $\mathbf{C}$ we have that
 the elements $f$, $g$ and $h$ preserve the lattice $\ho_{x, \mathbf{C}} \otimes_k R$ in $K_{x, \mathbf{C}} \otimes_k R$. Therefore the coresponding $C_3(f,g,h)$ is equal to $1$ for such $\mathbf{C}$. To obtain formula~\eqref{point} itself, i.e. when we take elements $f$, $g$ and $h$ from the ring $K_x \otimes_k R$ and when the product in this formula is taken over all irreducible curves $C \subset X$ such that $C \ni x$, we note that
 $$f  \cdot (\ho_{x, \mathbf{C}} \otimes_k R) = g  \cdot (\ho_{x, \mathbf{C}} \otimes_k R) = h  \cdot (\ho_{x, \mathbf{C}} \otimes_k R)=
 \ho_{x, \mathbf{C}} \otimes_k R$$
 when $\mathbf{C}$ is not a formal branch of some irreducible curve $C \subset X$. Therefore $C_3(f,g,h)=1$ for such $\mathbf{C}$. Thus formula~\eqref{point} follows from the previous product formula for elements from the group $(\Frac(\ho_x) \otimes_k R)^*$ and all $\mathbf{C}$.

\smallskip

\eqref{part2}
We will give a sketch of the proof of formula~\eqref{curve}. We fix an irreducible projective curve $C$ on $X$.
At first, we prove formula~\eqref{curve} when elements $f$, $g$ and $h$ are from the group $(k(X) \otimes_k R)^*$.
On the scheme $X_R$ we consider the following set $E$ of invertible subsheaves of the constant sheaf $k(X) \otimes_k R$  on the scheme $X_R$:
$$
\ff \in E     \qquad \mbox{iff} \qquad  \ff = g \cdot \oo_{X_R}(D)
$$
for some $g \in (k(X) \otimes_k R)^*$ and a divisor $D$ on $X$. It is clear that for any subsheaves $\ff$ and $\g$ from the set $E$ there is a sheaf
$\oo_{X_R}(D)$ for some divisor $D$ on $X$ such that inside the sheaf $k(X) \otimes_k R$ we have
$$
\ff \subset \oo_{X_R}(D)  \qquad \mbox{,}  \qquad \g \subset \oo_{X_R}(D) \mbox{.}
$$
Moreover, we can find a "minimal" sheaf with the above property. Using it, we can find also a divisor $G$ on $X$ such that
\begin{equation}  \label{sub}
\ff \supset \oo_{X_R}(G)  \qquad \mbox{,}  \qquad \g \supset \oo_{X_R}(G) \mbox{.}
\end{equation}

Let $J_C$ be the ideal sheaf of the curve $C$ on $X$. Let $J_{C, R}= J_C \otimes_k R$.
For any sheaf $\ff \in E$ we consider a complex
$$
\ad_{X, C, R}(\ff)= \mathop{\lim\limits_{{\lrto}}}_n   \mathop{\lim\limits_{{\longleftarrow}}}_{m >n}  \ad_{(C_R, \oo_{x,R} /J_{C,R}^{m-n})}(
\ff   \otimes_{\oo_{X,R}} J_{C,R}^n/J_{C,R}^m)   \mbox{,}
$$
where $(C_R, \oo_{x,R} /J_{C,R}^{m-n})$ is the scheme with the topological space as the topological space of the scheme $C_R= C \otimes_{\spec k} \spec R$  and the structure sheaf
as the sheaf $\oo_{x,R} /J_{C,R}^{m-n}$, and $\ad_{(C_R, \oo_{x,R} /J_{C,R}^{m-n})}(\cdot)$ is the functor of the adelic complex on this scheme applied to coherent sheaves on it.   We have that the cohomology groups $H^*(\ad_{X, C, R}(\ff))$ are isomorphic
to $\mathop{\lim\limits_{{\lrto}}}\limits_n   \mathop{\lim\limits_{{\longleftarrow}}}\limits_{m >n} H^*(X, \oo(D) \otimes_{\oo_X} J_C^n/ J_C^m) \otimes_k R$ for some divisor $D$ on $X$. Using the case when $R=k$ (see~\cite[\S~5B]{OsZh} and the proof of Proposition~12 from~\cite{O3}),    we have\footnote{It is important that $C$ is a projective curve. Therefore for any $m > n$ we have $\dim_k H^i(X, \oo(D)
\otimes_{\oo_X} J_C^n/ J_C^m) < \infty$ where $i$ is equal to  $1$ or to $2$.} that $H^0(\ad_{X, C, R}(\ff))$  is a Tate $R$-module,
and $\tilde{H}^1(\ad_{X, C, R}(\ff))$ which is the quotient space of ${H}^1(\ad_{X, C, R}(\ff))$ by the closure of $0$ is a Tate $R$-module.
We introduce a $\mathcal{P}ic^{\mathbb{Z}}_R$-module
$$\mathcal{D}et (H^*(\ad_{X, C, R}(\ff))) = \mathcal{D}et (H^0(\ad_{X, C, R}(\ff)))  -  \mathcal{D}et (\tilde{H}^1(\ad_{X, C, R}(\ff)))   \mbox{.}
$$
The adelic complex $\ad_{X, C, R}(\ff)$ looks as follows
$$
\da_{X,C, 0, R }(\ff)  \oplus   \da_{X, C, 1, R}(\ff)  \lrto \da_{X,C, 01, R}(\ff)  \mbox{,}
$$
where
$$
\da_{X,C, 0, R }(\ff) = K_C \otimes_k R \qquad \mbox{,}  \qquad
\da_{X,C, 01, R}(\ff)= \da_C((t_C)) \otimes_k R
$$
 and $t_C$ is a local parameter of the curve $C$
on some open affine subset of $X$, $\da_C$ is the  ring of adeles on the curve $C$. Besides,
$$\da_{X,C, 1, R}(\ff)= (\prod_{x \in C} ((B_x \otimes_{k} R) \otimes_{\oo_{X_R}} \ff)) \cap (\da_C((t_C)) \otimes_k R)
  \mbox{,}$$
where\footnote{Here and later we use the following notation. The
ring $B_x \otimes_k R$ is an $\oo_{x, R}= \oo_x \otimes_k R$-module.
We mean $(B_x \otimes_{k} R) \otimes_{\oo_{X_R}} \ff =(B_x
\otimes_{k} R) \otimes_{\oo_{x,R}} \ff_x $, where $\ff_x$ is a stalk
of the sheaf $\ff$ at the point $x \in X_R$, i.e. $\ff_x$ is an
$\oo_{x,R}$-module.} the intersection is taken inside of the ring
$\prod\limits_{x \in C } K_{x,C} \otimes_k R$, and $B_x$ is the
subring of the ring $K_{x,C}$ given as
$\mathop{\lim\limits_{\lrto}}\limits_{n > 0} s_C^{-n} \ho_x$ for
$s_C \in \oo_x$ which defines the curve $C$ on some local affine set
on $X$ containing the point $x$ (clearly, the ring $B_x$ does not
depend on the choice of $s_C$).

We fix two sheaves $\ff \supset \h$ from the set $E$.
For almost all points $x \in C$ we have
$$((B_x \otimes_{k} R) \otimes_{\oo_{X_R}} \ff) /
((B_x \otimes_{k} R) \otimes_{\oo_{X_R}} \h) =0  \mbox{.}$$
Therefore
$$\da_{X,C, 1, R}(\ff) / \da_{X,C, 1, R}(\h)   = \bigoplus_{x \in C} ((B_x \otimes_{k} R) \otimes_{\oo_{X_R}} \ff) /
((B_x \otimes_{k} R) \otimes_{\oo_{X_R}} \h)$$ is a Tate $R$-module, because for any point  $x \in C$ we have that
$$((B_x \otimes_{k} R) \otimes_{\oo_{X_R}} \ff) /
((B_x \otimes_{k} R) \otimes_{\oo_{X_R}} \h)$$ is a Tate $R$-module as it follows from a lemma.

\begin{lem}
Let $g \in (\Frac \ho_x \otimes_k R)^*$ such that $g \cdot (\ho_x \otimes_k R) \subset \ho_x \otimes_k R$. Then we have that
$(B_x \otimes_k R) / (g \cdot (B_x \otimes_k R)) $ is a Tate $R$-module.
\end{lem}
\begin{proof}
The beginning of the proof is similar to the proof of Proposition~\ref{lat}. Let $K = K_{x,C} \otimes_k R$ and $B_R= B_x \otimes_k R$.
Then the map $K \to K/B_R$ is splittable (i.e. admits a continuous splitting), because the map $K_{x,C} \to K_{x,C} /B_x$ is splittable. The multiplication by $g$ is a continuous automorphism of $K$. Therefore the map $K \to K/gB_R$ is splittable. The restriction of the last splitting gives the splitting of the map $B_R \to B_R/gB_R$.
Thus we have a topological decomposition: $$B_R = (B_R / g B_R) \oplus g B_R  \mbox{.}$$
Since $R$ is a local finite $k$-algebra, there is an element $f \in (\Frac \ho_x)^*$ such that $g \cdot (\ho_x \otimes_k R) \supset
f \cdot (\ho_x \otimes_k R)$ (it is easy to see it directly by localizing the regular ring $\ho_x$ at all prime ideals of height $1$, or one can look at the adelic complex $\ad_{X,x,R}(\cdot)$ considered above and its zero cohomology group). Therefore $g B_R  \supset f B_R$. As before, we prove that
there is a topological decomposition: $g B_R = (g B_R / f B_R) \oplus f B_R$. Therefore we have
$$
B_R = (B_R / g B_R) \oplus (g B_R / f B_R) \oplus f B_R \mbox{.}
$$
Hence we have that the $R$-module $B_R/ gB_R$ is a topological direct summand of the $R$-module $B_R/ f B_R$. Since $f \in (\Frac \ho_x)^*$,
we obtain $B_R /f B_R = (B_x / f B_x) \otimes_k R$. Now $B_x / fB_x$ is a Tate $k$-vector space. Therefore $B_R/f B_R$ is a Tate $R$-module.
Hence, $B_R/ gB_R$ is a Tate $R$-module.
\end{proof}

Now, using Remark~\ref{sum}  (for long cohomological sequences of Tate $R$-modules) we have for any sheaves $\ff$ and $\g$ from $E$
\begin{multline} \label{cohom}
\mathcal{D}et (H^*(\ad_{X, C, R}(\ff))) - \mathcal{D}et (H^*(\ad_{X, C, R}(\g)))  = \\ =
\mathcal{D}et(\da_{X,C, 1, R}(\g) \mid \da_{X,C, 1, R}(\ff))  \mbox{,}
\end{multline}
 where
 \begin{multline*}
 \mathcal{D}et(\da_{X,C, 1, R}(\g) \mid \da_{X,C, 1, R}(\ff)) = \\ =
 \mathcal{D}et(\da_{X,C, 1, R}(\ff) / \da_{X,C, 1, R}(\oo_{X_R}(G))) -
 \mathcal{D}et(\da_{X,C, 1, R}(\g) / \da_{X,C, 1, R}(\oo_{X_R}(G)))
 \end{multline*}
 and the divisor $G$ on $X$ is chosen by formula~\eqref{sub}  with the  "maximality" condition for the sheaf $\oo_{X_R}(G)$.

 From formula~\eqref{cohom} we have that the categorical central extension
 $$
 g \longmapsto \mathcal{D}et(\da_{X,C, 1, R}(\oo_{X_R}) \mid \da_{X,C, 1, R}(g\oo_{X_R}))   \mbox{,}  \quad g \in (k(X) \otimes_k R)^*
 $$
is isomorphic to the trivial categorical central extension over the group $(k(X) \otimes_k R)^*$.
It gives the reciprocity law for any elements $f$, $g$ and $h$ from the group $(k(X) \otimes_k R)^*$ along the curve $C$ on $X$, but for any point
$x \in C$ the local generalized commutator (depending on $3$ commuting elements) has to be constructed from the following categorical cental extension
\begin{equation}  \label{cat1}
d \longmapsto \mathcal{D}et  \left(  (B_x \otimes_{k} R)       \mid       ((B_x \otimes_{k} R) \otimes_{\oo_{X_R}} (d \cdot \oo_{X_R}))  \right)
\mbox{,}  \quad d \in (k(X) \otimes_k R )^* \mbox{.}
\end{equation}
In \S~\ref{gencat} we proved that the two-dimensional Contou-Carr\`{e}re symbol coincides with the generalized  commutator $C_3$ when this commutator
is calculated
from another categorical central extension\footnote{More exactly, in \S~\ref{gencat} we have considered the ring $R((u))((t))$, and the ring $K_{x,C} \otimes_k R$ is a finite direct product of these rings. Therefore the generalized commutator will be the finite product of two-dimensional Contou-Carr\`{e}re symbols calculated for every formal branch of the curve $C$ at the point $x$.}:
\begin{equation}  \label{cat2}
d \longmapsto \mathcal{D}et  \left( (\ho_{x,C} \otimes_k R) \mid  (g  \cdot (\ho_{x,C} \otimes_k R))  \right)  \mbox{,}  \quad d \in (k(X) \otimes_k R )^*  \mbox{,}
\end{equation}
where $\hat{\oo}_{x,C} \otimes_k R = \prod\limits_{i=1}^{s} (\hat{\oo}_{x, \mathbf{C}_i} \otimes_k R)$ (see formula~\eqref{xC})  is a subring in the ring $K_{x,C} \otimes_k R$.

We will show that  categorical central extension~\eqref{cat1} is inverse to  categorical central extension~\eqref{cat2}. From this fact we have that generalized commutator constructed  by central extension~\eqref{cat1} is equal to the minus one power  of the generalized commutator  constructed  by central extension~\eqref{cat2}. Thus we will prove  formula~\eqref{curve} for any elements $f$, $g$ and $h$ from the group $(k(X) \otimes_k R)^*$.

For any point $x \in C$  and any element $d \in (\Frac{\ho_{x}} \otimes_k R)^*$ we consider a
complex $\ad_{X,C,x, R}(d \cdot ( \ho_{x} \otimes_k R))$:
$$
(B_x \otimes_{\ho_x} (d \cdot (\ho_{x} \otimes_k R) ))  \, \oplus \,   (d \cdot (\ho_{x,C} \otimes_k R) )   \lrto K_{x,C} \otimes_k R  \mbox{.}
$$
We have that $H^i(\ad_{X,C,x, R}(d \cdot ( \ho_{x} \otimes_k R))) =
H^i(U_{x,R}, \, d \cdot
 ( \ho_{x} \otimes_k R)  \mid_{U_{x,R}} ) \otimes_k R$ (for $i$ equal to $1$ or to $2$)
is a Tate $R$-module, where (we recall) the scheme $U_{x,R} = \spec
(\ho_x \otimes_k R) \setminus x $. Inside of the ring $\Frac \ho_x
\otimes_k R$ we have an equality for any $(d \in k(X) \otimes_k
R)^*$:
$$B_x \otimes_{\ho_x} (d \cdot (\ho_{x} \otimes_k R) ) = (B_x \otimes_{k} R) \otimes_{\oo_{X_R}} (d \cdot \oo_{X_R})  \mbox{.}$$
Therefore from formula
\begin{multline*}
\hspace{-0.33cm} \mathcal{D}et \left( (B_x \otimes_k R) \mid  (B_x \otimes_{\ho_x} (d \cdot (\ho_{x} \otimes_k R) )) \right) +
\mathcal{D}et \left( (\ho_{x,C} \otimes_k R)  \mid (d \cdot (\ho_{x,C} \otimes_k R) ) \right) \\ =
\mathcal{D}et (H^*(\ad_{X,C,x, R}(d \cdot ( \ho_{x} \otimes_k R)))) - \mathcal{D}et (H^*(\ad_{X,C,x, R}( \ho_{x} \otimes_k R))
\end{multline*}
we obtain that categorical central extension~\eqref{cat1} is inverse to  categorical central extension~\eqref{cat2}.

Now we will obtain formula~\eqref{curve} for any elements $f$, $g$
and $h$ from the group $(K_C \otimes_k R)^*$. We note that $K_C=
k(C)((t_C))$, where $k(C)$ is the field of rational functions on the
curve $C$, and $t_C=0$ is an equation of the curve $C$ on some open
affine subset of $X$. Using that $R$ is an Artinian ring and $k(X)
\otimes_k R$ is dense in $K(C)((t_C)) \otimes_k R$ (when we consider
the discrete topology on the field $k(C)$) it is easy to construct
for any $n \in \mathbf{N}$ and $d \in (K_C \otimes_k R)^*$ elements
$d_1 \in (k(X) \otimes_k R)^*$ and $d_2 \in 1 + t_C^n \cdot (k(C)
\otimes_k R)[[t_C]]$ such that $d = d_1 d_2$. The two-dimensional
Contou-Carr\`{e}re symbol is a tri-multiplicative map. Besides, for
any  elements $f$  and $g$ from the group $(K_C \otimes_k R)^*$
there is\footnote{This kind of continuous property for all points $x
\in C$ is obvious when $\mathbf{Q} \subset R$ from
formula~\eqref{2cexp}. The general case follows  by means of the
lift to the case $\mathbf{Q} \subset R$, see the proof of
Lemma~\ref{lift}.} $n \in \mathbf{N}$ such that $(f,g,e)=1$ for any
element $e \in 1 + t_C^n \cdot (k(C) \otimes_k R)[[t_C]]$ and any
point $x \in C$. Hence we obtain formula~\eqref{curve} in general
case from the case when elements $f$, $g$  and $h$ are from  the
group $(K_C \otimes_k R)^*$.
\end{proof}

\section{Contou-Carr\`{e}re symbols via algebraic $K$-theory}\label{Kapp}
This section briefly discuss the $K$-theoretical approach to the usual and the higher-dimensional Contou-Carr\`{e}re symbols.
This approach develops some ideas suggested to us by one of the editors.

In Remark \ref{n-dim CC}, we indicated that there exists the $n$-dimensional Contou-Carr\`{e}re symbol given by an explicit formula
(when $\mathbf{Q} \subset R $).
It could also be obtained via algebraic $K$-theory for any (commutative) ring $R$ in the following way.

For a (commutative) ring $A$ and an integer $i \ge 0$, let $K_i(A)$ denote its $i$th algebraic $K$-group, as defined by D.~Quillen.
Recall that there is a canonical decomposition
 $K_1(A)= A^* \times SK_1(A)$, and there is the product structure in algebraic $K$-theory (see, e.g.,~\cite[\S~2]{S}):
$$K_1(A) \times \cdots \times K_1(A) \lrto K_{n+1}(A).$$
In addition, for any integer $m \ge 1$ there is the following canonical homomorphism
\begin{equation}  \label{Kato-map}
\partial_m:  K_m(A((t))) \lrto K_{m-1}(A)
\end{equation}
 which was constructed by K.~Kato in~\cite[\S~2.1]{K1}. We briefly review the construction. Let $H$ be the exact category of  $A[[t]]$-modules that are annihilated by some power of $t$ and that admit a resolution of length $1$ by finitely generated projective $A[[t]]$-modules.
Then the
``localization theorem for projective modules" (or localization theory of algebraic $K$-theory for singular varieties, see~\cite{Gr}, \cite[\S~9]{S}) produces a canonical homomorphism
\begin{equation}  \label{tilde}
\tilde{\partial}_m \; : \; K_m(A ((t)))  \lrto K_{m-1}(H)  \mbox{,}
\end{equation}
We claim that any $A[[t]]$-module  from the category $H$ is a finitely generated projective $A$-module (see Propositon~\ref{Mod} below).
 Thus we have an exact functor from  the category $H$ to the category of finitely generated projective $A$-modules, and therefore a homomorphism $K_{m-1}(H) \to K_{m-1}(A)$. The composition of the  homomorphism $\tilde{\partial}_m$ with the last homomorphism   gives the homomorphism
$\partial_m$.

\begin{prop}[Compare with~{\cite[\S~3.3]{G}}] \label{Mod}
Let  $\alpha : M \hookrightarrow N$ be an embedding of finitely generated projective $A[[t]]$-modules such that $\alpha$
becomes an isomorphism after inverting of the element $t$. Then $N/M$ is a finitely generated projective $A$-module.
\end{prop}
\begin{proof}
 By adding a finitely generated projective $A[[t]]$-module $T$ to $N$ and to $M$
we can assume that $N = A[[t]]^{s}$ for some $s \in  {\mathbf N}$. For some $l \in {\mathbf N}$ we have
$$
t^l N \subset M \subset N   \mbox{.}
$$
Let $P= N/M$. Then we have an embedding $P \hookrightarrow t^{-l}M/M $. We claim that this embedding splits as a map of $A$-modules. From this claim we obtain that $P$ is a finitely generated projective $A$-module, since $t^{-l}M/M$ is a finitely generated projective $A$-module (because we can add to $M$ an $A[[t]]$-module $Q$ such that $M \oplus Q = A[[t]]^r$ for some $r \in {\mathbf N}$). To prove the required splitting, it is enough to show
that an embedding $N \hookrightarrow t^{-l}M$ splits as a map of $A$-modules. The last splitting will follow if we will show that the composed embedding
$$
N \hookrightarrow t^{-l}M  \hookrightarrow t^{-l}N
$$
splits, again as a map of $A$-modules, but this is clear.
\end{proof}

Combining the above facts, one can produce for any (commutative) ring $R$ a map
\begin{multline}\label{n-dim CC K}
(R((t_n))\cdots((t_1))^*)^{n+1}\lrto (K_1(R((t_n))\cdots((t_{1}))))^{n+1} \lrto \\\lrto K_{n+1}(R((t_n))\cdots((t_1)))\stackrel{\partial_2\cdots\partial_{n+1}}{\longrightarrow} K_1(R)\to R^*  \mbox{.}
\end{multline}
By construction and the properties of the product structure in algebraic $K$-theory,  this map is an $n$-multiplicative, anti-symmetric and functorial with respect to $R$ map.
Moreover, map~\eqref{n-dim CC K} satisfies the Steinberg relation, since this relation is satisfied for the product structure in $K$-theory
(see~\cite[\S1 -\S2]{S}).

We believe that the map defined in~\eqref{n-dim CC K} coincides with the one given in Remark~\eqref{n-dim CC} when $\mathbf{Q} \subset R$. We outline the proof of this fact in the case $n=1$ and $n=2$. From the structures of group ind-schemes  $L\mathbb{G}_m$ and $L^2\mathbb{G}_m$ described in Section~\ref{first} and Sections~\ref{dlg} and~\ref{nattop}, the map \eqref{n-dim CC K} for $n=1$ and $2$ will coincide with the usual and the two-dimensional Contou-Carr\`{e}re symbols respectively.
(Note that the case $n=1$ answers a question of \cite[Remark 4.3.7]{KV}.)

\begin{thm}  \label{K-th} Let $R$ be any ring.
For $n=1$ (resp. $n=2$), the map constructed via $K$-theory by expression~\eqref{n-dim CC K} coincides with the Contou-Carr\`{e}re
symbol defined via Lemma-Definition~\ref{ld} (resp. via the  map $C_3$ by formula~\eqref{c3}).
\end{thm}
\begin{proof}
Since  the $K$-theoretical definition by formula~\eqref{n-dim CC K} is functorial with respect to $R$, we have   by
 Lemma-Definition~\ref{ld}   and by
 Lemma~\ref{unique} that it is enough to prove the proposition when $R$ is a $\mathbf{Q}$-algebra. To proceed, we note the following obvious lemma that $\partial_m$ is invariant under some change of local parameter $t$ (compare also
with~\cite[\S~2.1, Lemma~2]{K1}).
\begin{lem} \label{inv loc par}
Let  $f_*: K_m(A((t)))\to K_m(A((t)))$ be the map induced by the automorphism $f: A((t))\to A((t))$ given either by $f(t)=\sum_{i\geq 1} a_it^i$,
where
 $a_1\in A^*$, $a_i \in A \; (i> 1)$, or by $f(t)= a_0 + t$, where $a_0 \in \n A$.  Then $\partial_mf_*=\partial_m: K_m(A((t)))\to K_m(A)$.
\end{lem}

Now, we prove that when $n=1$, the map given by formula~\eqref{n-dim CC K} coincides with the usual Contou-Carr\`{e}re symbol. If $R=k$ is a field, then both definitions equal to the usual tame symbol, because the tame symbol is the boundary map in Milnor $K$-theory and this boundary map coincides
 with the map $\partial_2$ (see, e.g.,~\cite[\S~2.4, Cor.~1]{K1}).
  Formula~\eqref{n-dim CC K} defines  the bimultiplicative and anti-symmetric morphism from $L\mathbb{G}_m \times L\mathbb{G}_m$ to $\mathbb{G}_m$. We denote this morphism as
  $(\cdot , \cdot  )_{K{\rm-th}}$.
  Using the fact that $L\mathbb{G}_m=\widehat{\mathbb{W}}\times\mathbb{Z}\times\mathbb{G}_m\times\mathbb{W}$, and by  similar (but easier) arguments as in the proof of
 Theorem~\ref{th-main}, it remains to calculate the following pairing
$$(\widehat{\mathbb{W}} \times \mathbb{W})  \times    (\widehat{\mathbb{W}}  \times \mathbb{W}) \lrto\mathbb{G}_m$$
  induced by
the morphism $(\cdot , \cdot )_{K{\rm-th}}$. In the sequel, we change $R((t_1))$ to $R((t))$ for simplicity. By an easy analog of part~\ref{pa1} of Proposition~\ref{lemma-main} (applied to the affine group scheme given by a functor
$R \mapsto tR[[t]]$) and by an easy analog of  Lemma~\ref{hom} (see also formula~\eqref{two}), we can assume that $R=\mathbf{Q}[\varepsilon_1,\varepsilon_2]/(\varepsilon_1^2,\varepsilon_2^2)$. But by  Lemma~\ref{inv loc par}, by Lemma~\ref{calcu1} and  similar (but easier) arguments as in the proof of Lemma~\ref{sim}, it remains to show that
\begin{equation}  \label{K-expr}
(1-\varepsilon_1t,1-\varepsilon_2t^{-1})_{K{\rm-th}}=1-\varepsilon_1\varepsilon_2 \mbox{.}
\end{equation}

To check  expression~\eqref{K-expr}  we need the following observation. We recall a fact from~\cite[\S~2.4, Prop.~5]{K1}. For
$a \in K_*(A((t)))$ and  $t \in K_1(A((t)))$, let $\{a,t \}\in K_{*+1}(A((t)))$ denote the product.
\begin{lem} \label{Kato-lemma}
The composite
$$
\alpha \; : \; K_*(A[[t]]) \lrto K_*(A((t)))  \stackrel{a \mapsto \{a,t\}}{\lrto} K_{*+1}(A((t)))  \stackrel{\bigoplus\nolimits_m\tilde{\partial}_m}{\lrto} K_*(H)
$$
coincides with the homomorphism $\beta : K_*(A[[t]])  \to K_*(H)$ induced by the exact functor from the exact category of finitely generated projective $A[[t]]$-modules to the exact category $H$ given as $M \mapsto M/tM$.
\end{lem}
Therefore the composition of $\alpha$ with the homomorphism $K_*(H)  \to K_*(A)$ is the homomorphism $K_*(A[[t]]) \to K_*(A)$ induced by the functor $M \mapsto M /tM$. As a consequence we obtain that
for any $m \ge 1$, for any elements  $a_1, \ldots, a_m\in A[[t]]^*$:
\begin{equation}  \label{reduce}
\partial_{m+1} \{a_1, \ldots, a_m, t   \}= \{\overline{a_1}, \ldots, \overline{a_m} \}   \mbox{,}
\end{equation}
where for any $b \in A[[t]]^*$ we put $\overline{b} \in A^*$ under the homomorphism $A[[t]] \to A$.

Since $ t \mapsto at$, where $a \in A[[t]]^*$, is a well-defined change of local parameter in $A((t))$, from Lemma~\ref{inv loc par}, Formula~\eqref{reduce}
 and using $n$-multiplicativity we obtain also for any elements  $a_1, \ldots, a_{m+1}$ from the group $A[[t]]^*$:
 \begin{equation}  \label{reduce1}
 \partial_{m+1} \{a_1, \ldots, a_{m+1} \} =1 \mbox{.}
 \end{equation}
 (We note that formula~\eqref{reduce1} follows also from the localizing exact sequence (for singular varieties), since the composition
 $K_{m+1}(A[[t]])  \to K_{m+1}(A((t)))  \stackrel{\tilde{\partial}_{m+1}}{\lrto} K_{m}(H) $ is the zero map.)

From formula~\eqref{reduce} we have
$$
(1-\varepsilon_1t,1-\varepsilon_2t^{-1})_{K{\rm-th}}= (1 - \varepsilon_1 t, t)_{K{\rm-th}}^{-1} (1- \varepsilon_1t , t - \varepsilon_2)_{K{\rm-th}}=
(1- \varepsilon_1t , t - \varepsilon_2)_{K{\rm-th}}  \mbox{.}
$$
Using an automorphism of $R[[t]]$ given as  $t \mapsto t + \varepsilon_2$ (which is the change of the local parameter)
and Lemma~\ref{inv loc par} we have
$$
(1- \varepsilon_1t , t - \varepsilon_2)_{K{\rm-th}} = (1 - \varepsilon_1 \varepsilon_2 - \varepsilon_1t, t)_{K{\rm-th}}= 1 - \varepsilon_1 \varepsilon_2  \mbox{.}
$$
 Thus we have checked expression~\eqref{K-expr}. This finishes the proof of the case $n=1$.

\medskip

Now we consider the case $n=2$. Similarly and from Lemma~\ref{inv loc par}, the homomorphism
\begin{equation} \label{partial}
\partial_{m-1}\partial_m   \; : \; K_m(A((u))((t)))\longrightarrow K_{m-2}(A)
 \end{equation}
 is invariant under the change of local parameters (at least) of the following type:
\begin{multline}
\qquad \qquad \qquad \qquad
t\longmapsto \sum_{i\geq 1}a_i t^i \mbox{,} \quad a_1 \in A((u))^* \mbox{,}  \quad
 a_i \in A((u))  \; (i  >1) \mbox{;} \nonumber \\
 u\longmapsto \sum_{i\geq 1} b_iu^i+ gt \mbox{,} \quad b_1 \in A^* \mbox{,} \quad b_i \in A \; (i > 1) \mbox{,} \quad  g \in A((u))[[t]]
 \mbox{.}  \qquad \qquad  \quad
 \end{multline}
We note that from Lemma~\ref{inv loc par} we have that  homomorphism \eqref{partial} is invariant under the following change of the local parameter $t$: $t  \mapsto t + a$, where $a \in \n(A((u)))$. (We will use the last change of local parameter later in an explict calculation.)

The map given by formula~\eqref{n-dim CC K} defines for $n=2$ the  morphism from the ind-scheme $(L^2\bbG_m)^3$, where
 \[L^2\bbG_m=\bbP\times\bbZ^2\times \bbG_m \times \bbM\]
to the scheme $\bbG_m $.
We denote this tri-multiplicative and anti-symmetric morphism as $(\cdot , \cdot , \cdot )_{K{\rm-th}}$, and we will use the notation  $R((u))((t))$ instead of $R((t_2))((t_1))$ in formula~\eqref{n-dim CC K}.
If $R=k$ is a field, then the map $(\cdot , \cdot , \cdot )_{K{\rm-th}}$ coincides with the two-dimensional tame symbol, because the two-dimensional tame symbol
is the composition of two boundary maps in Milnor $K$-theory (see, e.g,~\cite[\S~4A]{OsZh}) and $\partial_m$ restricted to Milnor $K$-theory coincides with the boundary map there (see formula~\eqref{reduce} or~\cite[\S~2.4, Cor.~1]{K1}).

 From Lemma~\ref{unique} and Theorem~\ref{th-main} we see that to check the case $n=2$ it is enough to prove that the map
 $(\cdot , \cdot , \cdot )_{K{\rm-th}}$ satisfies properties~\eqref{2cexp}-\eqref{sign}   for a $\mathbf Q$-algebra $R$.
 Now, using the invariance of the map $(\cdot , \cdot , \cdot )_{K{\rm-th}}$ under the change of local parameters $u$ and $t$  and formulas~\eqref{reduce} and~\eqref{reduce1} we can repeat the proof of Theorem~~\ref{th-main} for the map $(\cdot , \cdot , \cdot )_{K{\rm-th}}$.
 Thus we reduce the proof to calculate some particular cases.

 As in the proof of Lemma~\ref{sim} we have to calculate for the ring \linebreak $R = \mathbf{Q}[\varepsilon_1,\varepsilon_2, \varepsilon_3]/(\varepsilon_1^2,\varepsilon_2^2,
 \varepsilon_3^2)$ the following expression $(1+ u \varepsilon_1, 1 + t \varepsilon_2, 1 + u^{-1}t^{-1} \varepsilon_3 )_{K{\rm-th}}$. We have
 \begin{multline}
 (1+ u \varepsilon_1, 1 + t \varepsilon_2, 1 + u^{-1}t^{-1} \varepsilon_3 )_{K{\rm-th}}=  \nonumber \\ =
 (1+ u \varepsilon_1, 1 + t \varepsilon_2, t )^{-1}_{K{\rm-th}}
 (1+ u \varepsilon_1, 1 + t \varepsilon_2, t + u^{-1} \varepsilon_3 )_{K{\rm-th}}
 \end{multline}
 From formula~\eqref{reduce} we have $(1+ u \varepsilon_1, 1 + t \varepsilon_2, t )_{K{\rm-th}} = 1$. Using an automorphism of the ring $R((u))((t))$ obtained by the change of the local parameter $t$: $t \mapsto t - u^{-1} \varepsilon_3$, and using again formula~\eqref{reduce} we obtain
 \begin{multline}
 (1+ u \varepsilon_1, 1 + t \varepsilon_2, t + u^{-1} \varepsilon_3 )_{K{\rm-th}}=
 (1+ u \varepsilon_1, 1 - u^{-1}\varepsilon_3 \varepsilon_2  + t \varepsilon_2, t  )_{K{\rm-th}} =  \nonumber \\ =
 (1+ u \varepsilon_1, 1 - u^{-1}\varepsilon_3 \varepsilon_2   )_{K{\rm-th}} \mbox{.}
 \end{multline}
 Now using the case $n=1$ of this theorem and an explicit formula~\eqref{multform}   we have that
 $$(1+ u \varepsilon_1, 1 - u^{-1}\varepsilon_3 \varepsilon_2   )_{K{\rm-th}}= 1 +\varepsilon_1 \varepsilon_2 \varepsilon_3 \mbox{.} $$
 Thus, we have checked that $(1+ u \varepsilon_1, 1 + t \varepsilon_2, 1 + u^{-1}t^{-1} \varepsilon_3 )_{K{\rm-th}}=
 1 +\varepsilon_1 \varepsilon_2 \varepsilon_3$.

 As in the proof of Lemma~\ref{sim2} we have the following calculation
 for the ring $R = \mathbf{Q}[\varepsilon_1,\varepsilon_2]/(\varepsilon_1^2,\varepsilon_2^2)$:
 \begin{multline}
 (t+ \ve_1, 1+t \ve_2, u)_{K{\rm-th}}= (1+t \ve_2, u, t+ \ve_1)_{K{\rm-th}}=(1+t \ve_2, u, t+ \ve_1)_{K{\rm-th}}=   \nonumber \\ =
 (1 - \ve_1\ve_2 + t \ve_2,  u, t)_{K{\rm-th}}= (1 - \ve_1\ve_2 ,  u)_{K{\rm-th}}= 1 - \ve_1\ve_2  \mbox{,}
 \end{multline}
 where we used the change of the local parameter $t \mapsto t - \ve_1$ and formula~\eqref{reduce}.

As in the proof of Lemma~\ref{sim3} we have the following calculation
 for the ring $R = \mathbf{Q}[\varepsilon_1,\varepsilon_2]/(\varepsilon_1^2,\varepsilon_2^2)$:
$$
(1 + u^{-1}\ve_1, 1 + u \ve_2, t)_{K{\rm-th}} = (1 + u^{-1}\ve_1, 1 + u \ve_2)_{K{\rm-th}} = 1 + \ve_1 \ve_2  \mbox{.}
$$

The above explicit calculations are the only calculations needed to repeat the proof of Theorem~\ref{th-main} for the case
of the map $(\cdot , \cdot , \cdot )_{K{\rm-th}}$. Thus we have checked the case $n=2$ of Theorem~\ref{K-th}. This finishes the proof of this theorem.
\end{proof}

The following corollary follows from the corresponding property of the product structure in algebraic $K$-theory.
\begin{cor} \label{Steinb}
For any ring $R$ the two-dimensional Contou-Carr\`{e}re symbol defined via the map $C_3$ by formula~\eqref{c3} satisfies the Steinberg relations.
\end{cor}

\begin{rmk}  \label{K-th-res}
Let $C$ be  a smooth projective curve  over a perfect field $k$, $R$ is a commutative $k$-algebra and $m \ge 1$ is an integer. There is the following reciprocity law. For a closed point $p$ of $C$ with residue field $k(p)$, let $t_p$ be a local coordinate around $p$. Using the ring homomorphism $k(C)\otimes_kR\to (k(p)\otimes_kR)((t_p))$ we obtain a homomorphism
\begin{equation}  \label{rec}
s_p \, :  \, K_m(k(C)\otimes_kR)\lrto K_m (k(p)\otimes_kR)((t_p))\stackrel{\partial_{p,m}}{\lrto} K_{m-1}(k(p)  \otimes_k R)  \lrto K_{m-1}(R) \mbox{,}
\end{equation}
where the map
 ${\partial}_{p,m }$ is  homomorphism~\eqref{Kato-map}  applied to $K_m((k(p) \otimes_k R)((t_{p})) )$, and the last arrow in~\eqref{res} denotes the tranfer (or pushforward) map.
Then the reciprocity law is: for any $x\in K_m(k(C)\otimes_kR)$ we have that  $s_p(x)$ is nonzero for only finitely many points $p$, and $\sum_{p\in C} s_p(x)=0$.

We briefly explain its proof.
Since $K_m(\cdot)$ commutes with filtered direct limits of rings (see, e.g.,~\cite[Lemma~5.9]{S})
we can assume that $x$ comes from $K_m(U\otimes_k R')$, where $U$ is an affine open subset of $C$ and $R'\subset R$ is a \emph{Noetherian} subring. In the sequel, we will write $R$ instead of  $R'$ for simplicity (so we can suppose  that $R$ is a Noetherian ring).
We consider $C \setminus U = \bigcup_{i=1}^q p_i$, where $p_i$ is a point on $C$.
    Let $j : U \otimes_k R   \hookrightarrow C \otimes_k R$ be the corresponding open embedding.
    We write  the localizing exact sequence for singular varieties
 (see~\cite[Th.~9.1]{S}):
 \begin{equation} \label{exseq}
 \ldots  \lrto K_m(U \otimes_k R)  \stackrel{\partial}{\lrto} K_{m-1} ({\mathcal H}) \stackrel{\alpha}{\lrto} K_{m-1}(C \otimes_k R) \lrto \ldots \mbox{,}
 \end{equation}
 where $\mathcal H $ is an exact category of coherent $\oo_{C \otimes_k R}$-modules ${\mathcal F}$ such that $j^* {\mathcal F} =0$, and
 ${\mathcal F }$ has a resolution of length $1$ by locally free $\oo_{C \otimes_k R}$-modules of finite rank.

 Let $f : C \otimes_k R  \to {\mathop{\rm Spec}} R$ be the projection. Since $f$ is a proper flat morphism, there is a well-defined map (see~\cite[Prop.~5.12]{S})
 $$f_*  \; : \;  K_{m-1} (C \otimes_k R) \to K_{m-1} (R)  \mbox{.}$$
 We consider a commutative diagram
 \begin{equation}  \label{cd}
 \begin{CD}
 K_m(U \otimes_k R)  @>{\partial}>>
K_{m-1}(\h)  @>{\alpha}>>  K_{m-1}(C \otimes_k R) \\
@VVV
@VVV@VV{f_*}V\\
\bigoplus\limits_{i=1}^q K_{m}( (k(p_i) \otimes_k R)((t_{p_i})))
@>{\bigoplus\limits_{i=1}^q  \tilde{\partial}_{p_i,m }}>>
  \bigoplus\limits_{i=1}^q K_{m-1} (H_{p_i}) @>{\beta}>>   K_{m-1}(R)  \mbox{,}
\end{CD}
\end{equation}
where for any $1 \le i \le q $, $H_{p_i}$ is the exact category attached to $(k(p_i) \otimes_k R)((t_{p_i}))$ and constructed as $H$ at the beginning of Section~\ref{Kapp}, and the map
 $\tilde{\partial}_{p_i,m }$ is  homomorphism~\eqref{tilde}  applied to $K_m((k(p_i) \otimes_k R)((t_{p_i})) )$.
 The map $\beta$ in this diagram is given as sum over $i$ of composition of maps $K_{m-1} (H_{p_i}) \to K_{m - 1} (k(p_i) \otimes_k R)$ and $K_{m - 1} (k(p_i) \otimes_k R) \to K_{m-1}(R)$.

Using $\alpha \circ \partial =0$ in sequence~\eqref{exseq} and commutative diagram~\eqref{cd} we obtain the reciprocity law for elements from
the group $K_m(U \otimes_k R)$.

When $m=2$, using the product structure in algebraic $K$-theory and Theorem~\ref{K-th}, from this reciprocity law we derive the reciprocity law on the curve $C$ over a perfect field $k$ for the usual Contou-Carr\`{e}re symbol over a $k$-algebra $R$ (compare with Theorem~\ref{thres1}).
\end{rmk}

\begin{rmk}
Let $R$ be a finite $k$-algebra.
We  now give a short explanation how from the above reciprocity law  on a projective curve (see Remark~\ref{K-th-res}) and Theorem~\ref{K-th} it is possible to obtain the reciprocity law
 for the two-dimensional   Contou-Carr\`{e}re symbol along a projective curve $C$ on a smooth algebraic surface $X$ over a perfect field $k$
 (see formula~\eqref{curve}). Previously, we proved this reciprocity law in Theorem~\eqref{res-law} by means of categorical central extensions and semilocal adelic complexes on $X$.

 We recall that the field $K_C = k(C)((t_C))$. Therefore $K_C \otimes_k R= (k(C) \otimes_k R) ((t_C))$.
 We suppose for simplicity that $C$ is a smooth curve. (If $C$ is not smooth then one has to work with the normalization of $C$.) Then for any
 point $x \in C$ we have that $K_{x,C}=k(x)((u))((t_C))$ is a two-dimensional local field, where $k(x)((u))$ is the completion of the field $k(C)$
 with respect to the discrete  valuation given by the point $x$.  Besides, $K_{x,C} \otimes_k R =(k(x)((u))\otimes_k R)((t_C))$   We have a commutative diagram
  $$
  \begin{CD}
K_3((k(C) \otimes_k R) ((t_C)))  @>{\partial_{C,3}}>>  K_{2}(k(C) \otimes_k R) \\
@VVV@VVV\\
K_3((k(x)((u)) \otimes_k R )((t_C)) ) @>{\partial_{x,C,3}}>>   K_{2}(k(x)((u)) \otimes_k R )  \mbox{.}
\end{CD}
$$
This diagram and Theorem~\ref{K-th} allow us to reduce the reciprocity law along $C$ on $X$
for the two-dimensional   Contou-Carr\`{e}re symbol
for any elements $f$, $g$ and $h$ from $(K_C \otimes_k R)^*$
to the reciprocity law on $C$ (which we considered in Remark~\ref{K-th-res}) for the element  $\partial_{C,3}  \{ f,g,h \}$ in $K_{2}(k(C) \otimes_k R)$. It is useful also to note that the homomorphism $(k(x) \otimes_k R)^*  \to R^*$ obtained from the transfer homomorphism
$K_1(k(x) \otimes_k R ) \to K_1(R)  $ coincides with the norm map.

We note that it is not clear for us  how it would be possible to apply the localizing exact sequence for singular varieties
to obtain via algebraic $K$-theory the reciprocity law   for the two-dimensional   Contou-Carr\`{e}re symbol around a point on a smooth algebraic surface (see formula~\eqref{point}), which is another reciprocity law obtained in Theorem~\ref{res-law}!
\end{rmk}

\section{The two-dimensional Contou-Carr\`{e}re symbol and two-dimensional class field theory} \label{cft}
Two-dimensional class field theory was developed by A.N.~Parshin,
K.~Kato and others, see~\cite{Pa0, Pa4, Pa} and~\cite{K1, KaS, Ka}.

By Proposition~\ref{deform}, the two-dimensional Contou-Carr\`{e}re symbol coincides with the two-dimensional tame symbol when the ground ring $R$ is equal to a field $k$. The two-dimensional tame symbol
was used in the local two-dimensional class field theory for the field $\df_q((u))((t))$ to describe the generalization of the Kummer duality and, consequently, Kummer extensions of the field  $\df_q((u))((t))$ where $\df_q$ is a finite field and $q=p^n$ for some prime $p$, see~\cite[\S~3.1]{Pa}.

We will construct some  one-parametric deformation of the two-dimensional tame symbol. It will be given as the  two-dimensional Contou-Carr\`{e}re symbol over some Artinian ring. From this deformation we will obtain the local symbol\footnote{The relation between one-dimensional Contou-Carr\`{e}re symbol and the Witt symbol for a usual one-dimensional local field $k((t))$ was noticed in~\cite[\S~4.3]{AP}.} which was  used by Parshin in~\cite[\S~3.1-3.2]{Pa}   to obtain the generalization of the
Artin-Schreier-Witt duality for the two-dimensional local field $\df_q((u))((t))$. The generalization of the Artin-Schreier-Witt duality describes abelian extensions of exponent $p^m$ of the field
$\df_q((u))((t))$.

\medskip

Let $R$ be any commutative ring.
Let $S$ be the set of positive integers which is  closed under passage to divisors. We denote by $W_S(R)$ the ring of (big) Witt
vectors\footnote{This is a ring scheme.},
i.e. $W_S(R)= \{ (x_i)_{i \in S} \}$ where  $x_i \in R$. We will need the ghost (or auxiliary) coordinates which are defined  as
$x (i) = \sum\limits_{d | i} d x_d^{i/d}$ where $i \in S$. The addition and multiplication in ghost coordinates of Witt vectors are coordinate-wise,
but in usual coordinates (i.e. in coordinates $x_i$) the addition and multiplication are given by some universal polynomials with integer coefficients in the variables $x_i$.    We have an equality in the ring ${\mathbf Q}[[s]]$:
\begin{equation}  \label{log-for}
-\log \prod_{i=1}^{\infty} (1 - x_i s^i)= \sum_{l=1}^{\infty} x(l) s^l/l  \mbox{.}
\end{equation}
For any positive integer $n$ we denote by $W_n(R)$ the truncated Witt vectors, i.e. in our previous notation  $W_n(R) = W_{\{1, \ldots, n\}}(R)$.
The additive group of the ring $W_n(R)$ is isomorphic to the group of invertible elements of the following kind $\{1+ r_1s + \ldots + r_i s^i + \ldots +
r_ns^n \}$
in the ring $R[s]/s^{n+1}$ by means of the map
\begin{equation} \label{wittmap}
(x_i)_{1 \le i \le n} \longmapsto \prod_{i=1}^n (1-x_is^i) \mod s^{n+1}  \mbox{.}
\end{equation}

\begin{rmk}
An obvious generalization of the map~\eqref{wittmap} gives an isomorphism between the additive group
of the ring $\mathop{\mathop{\lim}\limits_{\longleftarrow}}\limits_{n \ge 1} W_n(R)$ and the group ${\mathbb{W}}(R)$
introduced in~\S~\ref{first}.
\end{rmk}

We denote $W^p_n (R) = W_{\{1,p, \ldots, p^{n-1} \}}(R)$. The ring scheme $W^p_n$ is a natural quotient of the ring scheme $W_{p^{n-1}}$.
We denote\footnote{We note that usually when $p$ is fixed, the rings $W^p_n(R)$ and $W^p(R)$ are denoted as $W_n(R)$ and $W(R)$.} $W^p(R) = \mathop{\mathop{\lim}\limits_{\longleftarrow}}\limits_{n \ge 1} W_n^p(R)$.
\medskip

We fix a positive integer $n$ and a perfect field $k$. Let $R = k[s]/s^{n+1}$. Let $(\cdot, \cdot, \cdot )$ be the two-dimensional
Contou-Carr\`{e}re symbol: $(R((u))((t))^*)^3 \to R^*$. We define a tri-multiplicative map:
\begin{equation}  \label{witt}
k((u))((t))^* \times k((u))((t))^* \times W_n(k((u))((t)))  \lrto W_n(k)  \mbox{,}
\end{equation}
where on $W_n(k((u))((t)))$ and on $W_n(k)$ we consider only the additive structure (so "multiplicativity" means "with respect to additive structure of these rings"). Let map~\eqref{witt} be denoted by
$$
(g_1, g_2 \mid y_1, \ldots, y_n]  \;  \in \;  W_n(k)  \mbox{,}
$$
where $g_i \in k((u))((t))^*$  and  $(y_1, \ldots, y_n) \in W_n(k((u))((t)))$. Then this map is defined as follows:
\begin{equation}  \label{symbol}
\prod_{i=1}^n (1 -  (g_1,g_2 \mid y_1, \ldots, y_n]_i  \cdot s^i)  \mod s^{n+1}  = (\prod_{i=1}^n (1-y_is^i) \, , g_1, g_2)  \mbox{.}
\end{equation}

\medskip

If $ \cha k =p$, and we consider only non-zero coordinates such as $(y_1, y_p, \ldots, y_{p^{m-1}})$, then  we will show that the image of the expression
$(g_1, g_2 \mid y_1, \ldots, y_{p^{m-1}} ]$  in  the group $W^p_{m}(k)$ is equal to  the generalization  of the Witt symbol given by Parshin in~\cite[\S3]{Pa}\footnote{A.~N.~Parshin considered only the case $k=\df_q$. Besides, we have to compose the above expression with the trace map from $W^p_m(\df_q)$ to $W^p_m(\df_p)= \mathbf{Z}/p^m \mathbf{Z}$ to obtain the symbol for the generalization of the Witt duality,
see~\cite[\S~3, Prop.~7]{Pa}}. Indeed, we consider the field $\Frac W^p(k)$.
We have that $\mathbf{Q} \subset \Frac W^p(k)$  and $W^p(k)/ p W^p(k) =k$. We choose some lifts $\widetilde{y_i} \mbox{,} \, \widetilde{g_j}  \in W^p(k)((u))((t))$ of elements
$y_i \mbox{,} \, g_j \in k((u))((t))$. Using formulas~\eqref{2cexp}   and~\eqref{log-for} we have
\begin{multline*}
-\log (\prod_{i=1}^{p^{m-1}} (1- \widetilde{y_i} s^i), \widetilde{g_1}, \widetilde{g_2})=
-\Res( \log \prod_{i=1}^{p^{m-1}}(1-\widetilde{y_i}s^i) \frac{d \widetilde{g_1}}{\widetilde{g_1}} \wedge \frac{d\widetilde{g_2}}{\widetilde{g_2}}) = \\ =
\Res (\sum_{i=1}^{p^{m-1}} \frac{\widetilde{y}(i) s^i}{i} \frac{d \widetilde{g_1}}{\widetilde{g_1}} \wedge \frac{d\widetilde{g_2}}{\widetilde{g_2}}) =
\sum_{i=1}^{p^{m-1}} \Res (\widetilde{y}(i)  \frac{d \widetilde{g_1}}{\widetilde{g_1}} \wedge \frac{d\widetilde{g_2}}{\widetilde{g_2}}) \frac{s^i}{i}  \mbox{.}
\end{multline*}
Using formulas~\eqref{log-for} and~\eqref{symbol} we obtain
\begin{equation}  \label{parshin1}
(\widetilde{g_1}, \widetilde{g_2}  \mid \widetilde{y_1}, \ldots , \widetilde{y_{p^{m-1}}}](i) = \Res ( \widetilde{y}(i) \frac{d \widetilde{g_1}}{\widetilde{g_1}} \wedge \frac{d\widetilde{g_2}}{\widetilde{g_2}})   \mbox{.}
\end{equation}
Besides, we have
\begin{equation} \label{parshin2}
(g_1, g_2 \mid y_1, \ldots, y_{p^{m-1}} ]_i = (\widetilde{g_1}, \widetilde{g_2}  \mid \widetilde{y_1}, \ldots , \widetilde{y_{p^{m-1}}}]_i \mod p
 \mbox{.}
\end{equation}
Now if $i =1, p, \ldots, p^{m-1}$ and $k=\df_q$, then formulas~\eqref{parshin1}-\eqref{parshin2} coincide with Parshin's Definition~5 from~\cite[\S~3.3]{Pa}.

We note that we have just constructed the following map when $\cha k =p$:
\begin{equation}  \label{witt2}
k((u))((t))^* \times k((u))((t))^* \times W_m^p(k((u))((t)))  \lrto W_m^p(k)  \mbox{.}
\end{equation}
From formula~\eqref{parshin1} we have that map~\eqref{witt2} is additive with respect to the groups $W_m^p(k((u))((t)))$ and $W_m^p(k)$ (and, consequently, is tri-multiplicative with respect to all arguments like formula~\eqref{witt}), since  if $i =1, p, \ldots, p^{m-1}$, then passage to the ghost coordinates and to the usual coordinates depends only on these integers.

\begin{rmk}
Clearly, the reciprocity laws which were proved for the two-dimensional
Contou-Carr\`{e}re symbol in Theorem~\ref{res-law}  imply analogous reciprocity laws for \linebreak  maps~\eqref{witt} and~\eqref{witt2} such that the norm map $\Nm_{k''/k'}:
(k'' \otimes_k R)^* \to (k' \otimes_k R)^* $ is converted to the trace map  $\Tr_{k''/k'} : W_n(k'') \to W_n(k')$ for any finite extensions of fields
$k'' \supset k' \supset k$.

We can interpret the reciprocity laws for the two-dimensional tame symbol and for map~\eqref{witt2} as follows. Let $X$ be a smooth projective algebraic surface over a finite field $\df_q$. Then, according to~\cite{Pa4}, some  $K_2$-adelic object $K_{2,\da_X}$ and a reciprocity map $\theta: K_{2, \da_X}  \to \Gal(\df_q(X)^{\rm ab} / \df_q(X)) $ should exist such that
$$K_{2,\da_X} = \mathop{\prod\nolimits'}\limits_{x \in \bf{C}} K_2(K_{x,\bf{C}}) \subset \mathop{\prod}\limits_{x \in \bf{C}} K_2(K_{x,\bf{C}})$$
is a complicated "restricted" product and $$\theta= \mathop{\sum}\limits_{x \in \bf{C}} \theta_{x, \bf{C}} \mbox{,}$$ where $\theta_{x,\bf{C}} : K_2(K_{x, \bf{C}}) \to
\Gal(K_{x,\bf{C}}^{\rm ab}/ K_{x, \bf{C}})$ is the local reciprocity map from two-dimensional local class field theory (for any field $L $ by $L^{\rm ab}$ we denote its maximal abelian extension and $x \in \bf{C}$ runs over all points $x \in X$ and all formal branches $\bf C$ of all irreducible curves on $X$ which contain a point $x$).

For any point $x \in  X$ the ring $K_2^M(K_x)$ is diagonally mapped to the ring $K_{2, \da_X}$ (via  maps $K_2^M(K_x) \to K_2(K_{x,\bf{C}})$ for all ${\bf{C}} \ni x$ and we put
 $1 \in K_2(K_{y, \bf F})$ for all pairs $ y \in \bf F$ such that $y \ne x$). For any irreducible curve $C$ on $X$ the ring $K_2(K_C)$ is also diagonally mapped to the ring $K_{2, \da_X}$ (via maps $K_2(K_C) \to K_2(K_{x,\bf{C}})$ for all $x \in {\bf{C}} $ and we put
 $1 \in K_2(K_{y, \bf F})$ for all pairs $ y \in \bf F$ such that ${\bf F}$ is not a formal branch of $C$). Let an extension $N \supset \df_q(X) $ be either maximal Kummer extension (i.e. the union of all finite Galois extensions of exponent $q-1$ contained in $\df_q(X)^{\rm ab}$) or the maximal abelian $p$-extension (i.e the union of all finite Galois  $p$-extensions contained in $\df_q(X)^{\rm ab}$). Let
$\gamma : \Gal(\df_q(X)^{\rm ab} / \df_q(X)) \to \Gal(N/ \df_q(X) )$ be the natural quotient map. Then, similar to the case of algebraic curves over finite fields, we obtain from reciprocity laws
that $$ \gamma \circ \theta \, (K_2^M(K_x))=0 \qquad \mbox{and } \qquad \gamma \circ \theta \, (K_2(K_C))=0$$ for any $x \in X$ and $C \subset X$.
\end{rmk}

\begin{rmk}
We obtained reciprocity laws on an algebraic surface  for the Parshin generalization of the Witt symbol as the consequence of our reciprocity laws
for the two-dimensional  Contou-Carr\`{e}re symbol. We note that earlier K.~Kato and S.~Saito obtained  by another methods similar reciprocity laws for the generalization of the Witt symbol, see Lemma~4 and Lemma~5 from~\cite[Ch.~III]{KaS}. We explain the relation between two generalizations of the Witt symbol.
Let $k$ be a perfect field, $\cha k = p$. Let $G$ be a commutative smooth connected algebraic group over $k$. The authors constructed
in~\cite{KaS} some local symbol map:
\begin{equation}  \label{Kato}
  k((u))((t))^* \times k((u))((t))^*  \times G(k((u))((t))) \lrto G(k)
\end{equation}
and proved the two-dimensional reciprocity laws for this symbol map. When $G=W_m^p$, then map~\eqref{Kato} is the generalization of the Witt symbol.
We note that the construction for this symbol map given in~\cite{KaS} was very indirect. An explicit easy formula for map~\eqref{Kato} (when
$G =W_m^p $) can be found in~\cite[\S~7.1]{KR}. From this formula one can easily see that map~\eqref{Kato} (when
$G =W_m^p $) coincides with map~\eqref{witt2}, which is described by formulas~\eqref{parshin1}-\eqref{parshin2}.
\end{rmk}

\bigskip

\noindent Denis Osipov, \\
Steklov Mathematical Institute, Gubkina str. 8,
Moscow, 119991, Russia \\
E-mail address:  ${d}_{-} osipov@mi.ras.ru$

\bigskip
\noindent Xinwen Zhu, \\
Department of Mathematics, Northwestern University, 2033 Sheridan Road, Evanston, IL 60208, USA \\
E-mail address:  $xinwenz@math.northwestern.edu$

\end{document}